\documentclass[12pt]{article} 
      
\usepackage{amssymb}       
\usepackage{amsmath}       
\usepackage{latexsym}       
\usepackage{amsthm}       
\usepackage{epic,eepic,latexsym,amsbsy,amsmath,amscd,amssymb,mathtools,enumerate}

\numberwithin{equation}{section}

\input xy
\xyoption{all}

 \newtheorem{definition}{Definition} [section]       
 \newtheorem{remark}[definition]{Remark}       
 \newtheorem{example}[definition]{Example}

\newtheorem{algorithm}[definition]{Algorithm}

 \newtheorem{proposition}[definition]{Proposition}       
 \newtheorem{theorem}[definition]{Theorem}       
 \newtheorem{corollary}[definition]{Corollary}       
  \newtheorem{lemma}[definition]{Lemma}

\newcommand{\Aut}{\mbox{\rm Aut}}
\newcommand{\Sym}{\mbox{\rm Sym}}

\newcommand{\Ind}{\mbox{\rm Ind}}

\newcommand{\Res}{\mbox{\rm Res}}
\newcommand{\id}{\mbox{\rm id}}
\newcommand{\Hom}{\mbox{\rm Hom}}
\newcommand{\tv}{\mbox{\rm tv}}
 
\begin{document} 

\title{Representation theory of the group of automorphisms of a finite rooted tree.} 
 \author{Fabio Scarabotti}  
 
\maketitle 

\begin{abstract} 
We construct the ordinary irreducible representations of the group of automorphisms of a finite rooted tree
and we get a natural parametrization of them.
To achieve these goals, we introduce and study the combinatorics of tree compositions, a natural generalization of set compositions but with new features and more complexity. These combinatorial structures lead to a family of permutation representations which have the same parametrization of the irreducible representations. 
Our trees are not necessarily spherically homogeneous and our approach is coordinate free. 
\footnote{{\it AMS 2020 Math. Subj. Class.}Primary: 20C15. Secondary: 05E10, 20B25, 20C30.

\it Keywords: Automorphisms of rooted trees, Symmetric groups, Representation theory, Clifford theory, Mackey theory.}
\end{abstract}

\tableofcontents

\section{Introduction}\label{introduction}

The symmetric group $S_n$ may be seen as the group of automorphisms of a rooted tree $\mathcal{T}$ with just one level, that is an $n$-star. Similarly, the wreath product of two symmetric groups $S_m\wr S_n$ is the group of automorphisms of a rooted tree with two levels: just add $m$ children to each leaf of the $n$-star; we may continue: the iterated wreath product of $h$ symmetric groups is naturally the group of automorphisms of a spherically homogeneous rooted tree $\mathcal{T}$ with $h$ levels. This observation was the starting point of the joint paper with Tullio Ceccherini-Silberstein and Filippo Tolli \cite{CST2}, where, in particular, we analyzed the permutation representations on the varieties of spherically homogeneous subtrees, showing that they decomposes without multiplicity. This is a clear analogy with the $h=1$ case, where they correspond to the permutation representations $M^{n-k,k}$ of the symmetric group. 

In the present paper we push the analogy further. We introduce a theory of tree compositions, which generalizes the usual theory of set compositions, and we use it to construct a family of permutation representations of $\Aut(\mathcal{T})$, the group of automorphisms of a finite rooted tree $\mathcal{T}$ (which is not necessarily spherically homogeneous). 
These permutation representations are the natural generalization of the $M^\lambda$ usually introduced for the symmetric group (see \cite{book2,JK,SaganSym}) and we parametrize them by means of what we have called trees of partitions. Then we show that these combinatorial objects also yields a natural parametrization of the conjugacy classes. The irreducible representations of $\Aut(\mathcal{T})$ are constructed in two independent ways: by means of iterated Clifford theory and then by means of Mackey theory. The latter approach stress further the analogies with the representation theory of the symmetric group: we show that we can use one of the classical approach for $S_n$, namely that in \cite[Chapter 2]{JK}, in such a way that the representation theory of $S_n$ becomes a particular case of that of $\Aut(\mathcal{T})$. In both cases, the irreducible representations are naturally parametrized by means of the trees of partitions previously introduced.

This is the plan of the paper.

In Section \ref{Sectreecomp} we define the fundamental notion of tree composition. It is a natural generalization of the composition of a set but it has two peculiarities. First of all, it is presented in a coordinate free fashion, by means of surjective maps between rooted trees. Moreover, these maps must satisfy a regularity condition imposed by the representation theory. 
In Section \ref{Secquotient} we define the quotient of a tree composition. Its name derives from its construction: we start by introducing a suitable equivalence relation and its quotient yields the desired composition.
In Section \ref{Sectreeofpart} we introduce the type of a tree composition: it is a generalization, in our setting, of the usual notion of integer composition. Then we introduce the notion of tree of partitions: it is the generalization, in our setting, of the notion of integer partition and we show that tree of partitions parametrizes the equivalence classes of tree compositions. 

In Section \ref{Secconjclasses} we describe and parametrize the conjugacy classes of $\Aut(\mathcal{T})$ by means equivalence classes of tree compositions which in turn are parametrized by the tree of partitions associated to $\mathcal{T}$. Similar results were obtained in \cite{GNS}.
In Section \ref{SecClifford} we give a first construction of the irreducible representations of $\Aut(\mathcal{T})$ by iterating Clifford theory, showing also that they are naturally subspaces of the permutation representations given by the action of the group on the tree compositions. This way, we are naturally lead to the parametrization by means of the trees of partitions.

In Section \ref{Secrefinements} we construct a theory of common refinements of two tree compositions. It is the key ingredient to parametrize diagonal actions of $\Aut(\mathcal{T})$ on products of spaces of tree compositions and therefore to apply Mackey theory.
In Section \ref{Sectransp} we introduce a total order between tree of partitions that generalizes the lexicographic order for integer partitions. Then we generalizes, in our setting, two basic combinatorial and algebraic facts which are fundamental ingredient in the representation theory of the symmetric group, namely the existence and uniqueness of the $0-1$ matrix in \cite[2.1.1]{JK} and the easier implication of the Gale-Ryser Theorem \cite[1.4.17]{JK}, with the dominance order replaced by the total order. These results are quite easy for set compositions but in the context of tree compositions they require more technical proofs.
In Section \ref{SecrepAuttree} we finally give a second construction of the ordinary irreducible representations of $\Aut(\mathcal{T})$, exactly as in \cite[Theorem 2.1.11]{JK}. This approach does not use the irreducible representation of the symmetric group and, with respect to the Clifford approach, it requires just induction of the trivial and alternating representations from two subgroups.

Other references on the approach to the representation theory of the symmetric group that we follow in Section \ref{SecrepAuttree} are \cite{CST3,Mayer}; it is also sketched in \cite[Section 3.11]{Sternberg}. James and Kerber attributed it to A. J. Coleman \cite{Coleman}. In \cite{KT} the basic property on which it is based is extended to the wreath product of two symmetric groups. The standard reference on the representation theory of wreath products is \cite[Chapter 4]{JK}; a more recent account is \cite{book}. The representation theory of groups of the form $G\wr S_n$ is an active area of research; see \cite{AdinRoich,MishSri}. The representation theory of iterated wreath product of cyclic groups is studied in \cite{OOR}. General references on groups acting on (infinite) rooted trees are \cite{BORT,Grigorchuk}. In \cite{GNS} the authors describe the conjugacy classes of the groups of automorphisms of infinite trees. In \cite{Olshanskii} and \cite[Chapter 3]{FigNeb} the irreducible representations of the groups of automorphisms of certain finite trees are used to describe the irreducible representations of the group of automorphisms of an infinite, homogeneous tree.

\section{Finite rooted trees: automorphisms and compositions}\label{Sectreecomp}

Let $\mathcal{T}$ be a finite rooted tree \cite[Section 1.5]{Diestel}; we identify $\mathcal{T}$ we the set of its vertices. As in \cite[Section 3.1]{CST2}, the root is denoted by $\{\emptyset\}$. A {\em path} in $\mathcal{T}$ of length $k$ is a finite sequence of vertices 
$v_0,v_1,v_2,\dotsc,v_k\in\mathcal{T}$
such that $v_j$ is adjacent to $v_{j+1}$, $j=0,1,2,\dotsc,k-1$. If $v_0$ is the root we say that $v_k$ is at level $k$; the $k$-\emph{level} of the tree is the set of all vertices at level $k$. If $v,w\in \mathcal{T}$ are adjacent, $v$ at level $k$ and $w$ at level $k+1$ we say that $v$ is the \emph{parent} of $w$ and that $w$ is a \emph{child} of $v$. A \emph{leaf} of $\mathcal{T}$ is a vertex without children; a vertex with at least one child is {\em internal}. If $u\in \mathcal{T}$ is not a leaf then $\text{\rm Ch}_\mathcal{T}(u)$ (or simply $\text{\rm Ch}(u)$) will denote the set of its children. If $u,w$ have the same parent $v$ we say that $u$ and $w$ are {\em siblings}. The \emph{height} of the tree $\mathcal{T}$ is the level of the lowest leaf; equivalently, the length of the longest path starting from the root. We denote by $H(\mathcal{T})$ the height of $\mathcal{T}$. If $v,w\in\mathcal{T}$, $v$ at level $j$, $w$ at level $k$, $j<k$ and there is a path from $v$ to $w$ we say that $v$ is an {\em ancestor} of $w$ and that $w$ is a {\em descendant} of $v$. 

We denote by $\Aut (\mathcal{T})$ the group of all automorphisms of $\mathcal{T}$ as a rooted tree, that is the group of all bijective maps
$g \colon \mathcal{T} \to \mathcal{T}$ which preserve the adjacency relation and that, in addition, fix the root $\emptyset$.  For each internal vertex $u\in \mathcal{T}$ we denote by $\mathcal{T}_u$ the subtree made up of $u$ and of all its descendants; if $\mathcal{P}$ is another rooted tree and $\varphi\colon \mathcal{T}\rightarrow\mathcal{P}$ is a {\em rooted tree homomorphism} (that is, a graph homomorphism \cite[Section 1.1]{Diestel} that sends the root of $\mathcal{T}$ to the root of $\mathcal{P}$) we denote by 
\begin{equation}\label{defphiu}
\varphi_u\colon \mathcal{T}_u\rightarrow \mathcal{P}_{\varphi(u)} 
\end{equation}
the restriction of $\varphi$ to $\mathcal{T}_u$.

If $h=H(\mathcal{T})$ and $\lvert\text{Ch}(u)\rvert=r_k$, for every vertex $u$ at level $k-1$, $k=1,2,\dotsc,h$, we say that $\mathcal{T}$ is a \emph{spherically homogeneous rooted tree of branching type} ${\bf r}\coloneqq(r_1,r_2, \ldots, r_h)$. If this is the case, then
\begin{equation}\label{Autiterwreath}
\Aut(\mathcal{T}) \cong S_{r_h} \wr S_{r_{h-1}} \wr \cdots \wr S_{r_2} \wr S_{r_1},
\end{equation}
where $S_n$ is the symmetric group on $n$ objects and $\wr$ denotes the wreath product; see \cite[Theorem 2.1.15]{book3}. If $r_1=r_2=\dotsb=r_h=1$ then we say that $\mathcal{T}$ is a {\em path of length} $h$.  

\begin{example}
{\rm
The first tree below is of height 3 while the second tree is spherically homogeneous of branching type $(3,2,3)$. \\

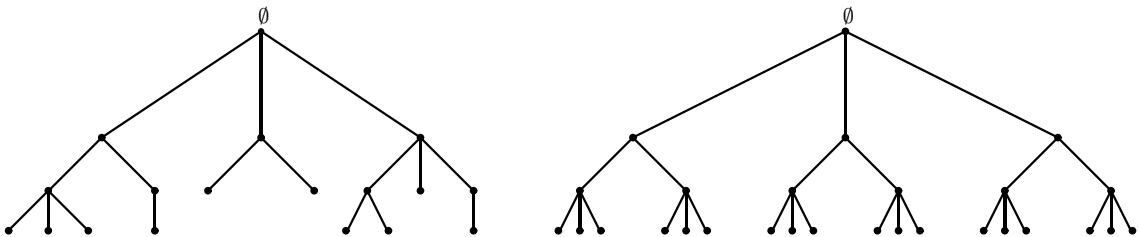
\begin{picture}(400,100)
\put(99,93){${\scriptstyle \emptyset}$}
\put(100,90){\circle*{2}}

\thicklines
\put(100,90){\line(-3,-2){60}}
\put(100,90){\line(0,-2){40}}
\put(100,90){\line(3,-2){60}}

\put(40,50){\circle*{2}}
\put(100,50){\circle*{2}}
\put(160,50){\circle*{2}}


\put(40,50){\line(-1,-1){20}}
\put(40,50){\line(1,-1){20}}

\put(20,30){\circle*{2}}
\put(60,30){\circle*{2}}


\put(100,50){\line(-1,-1){20}}
\put(100,50){\line(1,-1){20}}

\put(80,30){\circle*{2}}
\put(120,30){\circle*{2}}


\put(160,50){\line(-1,-1){20}}
\put(160,50){\line(0,-1){20}}
\put(160,50){\line(1,-1){20}}

\put(140,30){\circle*{2}}
\put(160,30){\circle*{2}}
\put(180,30){\circle*{2}}

\put(20,30){\line(-1,-1){15}}
\put(20,30){\line(0,-1){15}}
\put(20,30){\line(1,-1){15}}

\put(5,15){\circle*{2}}
\put(20,15){\circle*{2}}
\put(35,15){\circle*{2}}

\put(60,30){\line(0,-1){15}}
\put(60,15){\circle*{2}}


\put(140,30){\line(-1,-2){8}}
\put(140,30){\line(1,-2){8}}

\put(132,15){\circle*{2}}
\put(148,15){\circle*{2}}

\put(180,30){\line(0,-1){15}}
\put(180,15){\circle*{2}}

\put(319,93){${\scriptstyle \emptyset}$}
\put(320,90){\circle*{2}}

\thicklines
\put(320,90){\line(-2,-1){80}}
\put(320,90){\line(0,-2){40}}
\put(320,90){\line(2,-1){80}}

\put(240,50){\circle*{2}}
\put(320,50){\circle*{2}}
\put(400,50){\circle*{2}}


\put(240,50){\line(-1,-1){20}}
\put(240,50){\line(1,-1){20}}

\put(220,30){\circle*{2}}
\put(260,30){\circle*{2}}


\put(320,50){\line(-1,-1){20}}
\put(320,50){\line(1,-1){20}}

\put(300,30){\circle*{2}}
\put(340,30){\circle*{2}}


\put(400,50){\line(-1,-1){20}}
\put(400,50){\line(1,-1){20}}

\put(380,30){\circle*{2}}
\put(420,30){\circle*{2}}


\put(220,30){\line(-1,-2){8}}
\put(220,30){\line(0,-1){15}}
\put(220,30){\line(1,-2){8}}

\put(212,15){\circle*{2}}
\put(220,15){\circle*{2}}
\put(228,15){\circle*{2}}

\put(260,30){\line(-1,-2){8}}
\put(260,30){\line(0,-1){15}}
\put(260,30){\line(1,-2){8}}

\put(252,15){\circle*{2}}
\put(260,15){\circle*{2}}
\put(268,15){\circle*{2}}


\put(300,30){\line(-1,-2){8}}
\put(300,30){\line(0,-1){15}}
\put(300,30){\line(1,-2){8}}

\put(292,15){\circle*{2}}
\put(300,15){\circle*{2}}
\put(308,15){\circle*{2}}

\put(340,30){\line(-1,-2){8}}
\put(340,30){\line(0,-1){15}}
\put(340,30){\line(1,-2){8}}

\put(332,15){\circle*{2}}
\put(340,15){\circle*{2}}
\put(348,15){\circle*{2}}


\put(380,30){\line(-1,-2){8}}
\put(380,30){\line(0,-1){15}}
\put(380,30){\line(1,-2){8}}

\put(372,15){\circle*{2}}
\put(380,15){\circle*{2}}
\put(388,15){\circle*{2}}

\put(420,30){\line(-1,-2){8}}
\put(420,30){\line(0,-1){15}}
\put(420,30){\line(1,-2){8}}

\put(412,15){\circle*{2}}
\put(420,15){\circle*{2}}
\put(428,15){\circle*{2}}

\end{picture}
}
\end{example}

We introduce a coordinate free description of a rooted tree in the following way: we denote by $V_\mathcal{T}$ the vertices at the first level and by $W_\mathcal{T}$ the vertices in $V_\mathcal{T}$ which are leaves. Then the {\em first level decomposition} is given by:
\begin{equation}\label{firstld}
\mathcal{T}=\{\emptyset\}\sqcup V_\mathcal{T}\sqcup \left[\bigsqcup_{u\in V_\mathcal{T}\setminus W_\mathcal{T}}\left(\{u\}\times \widetilde{\mathcal{T}_u}\right);\right],
\end{equation}
where $\widetilde{\mathcal{T}_u}\coloneqq\mathcal{T}_u\setminus\{u\}$. That is, we think that to each $u\in V_\mathcal{T}\setminus W_\mathcal{T}$ is attached the tree $\mathcal{T}_u$ and we call it a {\em first level (rooted) subtree}. When we use \eqref{firstld} we usually represent the vertices of $\mathcal{T}$ below the first level as couples of the form 
\begin{equation}\label{firstldcoup}
(u,w)\quad\text{ where }\; u\in V_\mathcal{T}\setminus W_\mathcal{T}\;\text{ and }\; w\in \widetilde{\mathcal{T}_u}.
\end{equation}
Clearly, if  $H(\mathcal{T})\geq 2$ then $W_\mathcal{T}$ might be empty.

By means of \eqref{defphiu} and \eqref{firstldcoup}, a rooted tree homomorphism $\varphi\colon \mathcal{T}\rightarrow \mathcal{P}$ may be represented in the following form: a map between sets $\varphi\rvert_{W_\mathcal{T}}\colon W_\mathcal{T}\rightarrow W_\mathcal{P}$ and then
\begin{equation}\label{firstldell}
\varphi(u,w)=(\varphi(u),\varphi_u(w)),\qquad\text{ for all }u\in V_\mathcal{T}\setminus W_\mathcal{T}, w\in \widetilde{\mathcal{T}_u}.
\end{equation}
If $\alpha\colon\mathcal{P}\rightarrow\mathcal{Q}$ is another tree homomorphism and $u,w$ are as above, then $(\alpha\circ\varphi)(u,w)=\alpha(\varphi(u),\varphi_u(w))=\bigl((\alpha\circ\varphi)(u),\left(\alpha_{\varphi(u)}\circ\varphi_u\right)(w)\bigr)$, that is
\begin{equation}\label{firstldell2}
(\alpha\circ\varphi)_u=\alpha_{\varphi(u)}\circ\varphi_u.
\end{equation}
We will use \eqref{firstldell} also to express isomorphisms and automorphisms. In particular, if $g\in\Aut(\mathcal{T})$ then from \eqref{firstldell2} we deduce that 
\[
(u,w)=(g\circ g^{-1})(u,w)=\bigl(g\circ g^{-1}(u),\left(g_{g^{-1}(u)}\circ (g^{-1})_u\right)(w)\bigr)
\]
and therefore
\begin{equation}\label{firstldell3}
(g^{-1})_u=\left(g_{g^{-1}(u)}\right)^{-1}.
\end{equation}

\begin{definition}\label{deftreecomp}
{\rm
If $\mathcal{T},\mathcal{P}$ are rooted trees a $\mathcal{P}$-{\em composition} of $\mathcal{T}$ 
is a surjective rooted tree homorphism $\varphi\colon \mathcal{T}\rightarrow\mathcal{P}$, such that:
if $u,v\in\mathcal{T}$ are {\em distinct internal siblings} and $\varphi(u)=\varphi(v)\eqqcolon x$ then there exists $\tau\in\Aut(\mathcal{T})$ such that $\tau(u)=v$ and $\varphi_u=\varphi_v\circ\tau_u$, that is the following diagram is commutative:
\begin{equation}\label{diagTRSuv2old}
\xymatrix@R=5pt{
\mathcal{T}_u\ar[dr]^{\varphi_u}\ar[dd]_{\tau_u} & \\
&\mathcal{P}_x\\
\mathcal{T}_v\ar[ur]_{\varphi_v}&}
\end{equation}
We also say that $\varphi$ is a {\em tree composition} of $\mathcal{T}$ and that \eqref{diagTRSuv2old} is the {\em regularity condition}.}
\end{definition}

If there exists a $\varphi$ as above then we say that $\mathcal{P}$ is a {\em compositional tree} for $\mathcal{T}$. Note that, in \eqref{diagTRSuv2old}, it is equivalent to say that $\tau_u$ is an isomorphism between $\mathcal{T}_u$ and $\mathcal{T}_v$ because such an isomorphism may always extended to an automorphism of the whole tree.  Now we show that the regularity condition has a more extensive validity.
\begin{lemma}\label{lemmaregcond}
If $\varphi\colon\mathcal{T}\rightarrow\mathcal{P}$ is a tree composition, $u,v\in\mathcal{T}$ are distinct internal vertices and $\varphi(u)=\varphi(v)\eqcolon x$ then there exists $\tau\in\Aut(\mathcal{T})$ such that
\begin{equation}\label{diagTRSuv2}
 \tau(u)=v\qquad\text{ and }\qquad\varphi_u=\varphi_v\circ\tau_u.
\end{equation}
\end{lemma}
\begin{proof}
Let $w$ be the lowest common ancestor of $u$ and $v$, so that $w$ has distinct siblings $\widetilde{u}$ ancestor of $u$ and $\widetilde{v}$ ancestor of $v$. Clearly, $\varphi(\widetilde{u})=\varphi(\widetilde{v})=$ the ancestor of $x$ at the same level of $\widetilde{u}$ and $\widetilde{v}$ so that, by the hypothesis, there exists 
$\widetilde{\tau}\in\Aut(\mathcal{T})$ such that 
\begin{equation}\label{starp1bis}
\widetilde{\tau}\left(\widetilde{u}\right)=\widetilde{v}\qquad\text{ and }\qquad\varphi_{\widetilde{u}}=\varphi_{\widetilde{v}}\circ\widetilde{\tau}_{\widetilde{u}}. 
\end{equation}
Let $l$ be the length of the path from $w$ to $u$. If $l=1$ then $u=\widetilde{u}$, $v=\widetilde{v}$ and we are done. Otherwise we use induction on $l$. Set $u'\coloneqq \widetilde{\tau}^{-1}(v)$, so that $u'\in\mathcal{T}_{\widetilde{u}}$.
Then the restriction of \eqref{starp1bis} to $\mathcal{T}_{u'}$ yields
$\varphi_{u'}=\varphi_v\circ\widetilde{\tau}_{u'}$. If $u'=u$ we are done. Otherwise, from $\widetilde{\tau}(u')=v$ and \eqref{starp1bis} it follows that 
\[
\varphi(u')=\varphi_{\widetilde{u}}(u')=\varphi_{\widetilde{v}}\circ\widetilde{\tau}_{\widetilde{u}}(u')=\varphi_{\widetilde{v}}(v)=\varphi(v)=\varphi(u),
\]
so that we may apply the inductive hypothesis on $u'$ and $u$ (their lowest common ancestor is $\widetilde{u}$ or below $\widetilde{u}$): there exists $\tau'\in\Aut(\mathcal{T})$ such that $\tau'(u)=u'$ and $\varphi_u=\varphi_{u'}\circ\tau'_u$. Finally,
\[
\left(\widetilde{\tau}\circ\tau'\right)(u)=\widetilde{\tau}(u')=v\quad\text{and}\quad\varphi_u=\varphi_{u'}\circ\tau'_u=\varphi_v\circ\widetilde{\tau}_{u'}\circ\tau'_u=\varphi_v\circ\left(\widetilde{\tau}\circ\tau'\right)_u\quad\Longrightarrow\quad\tau=\widetilde{\tau}\circ\tau'. 
\]

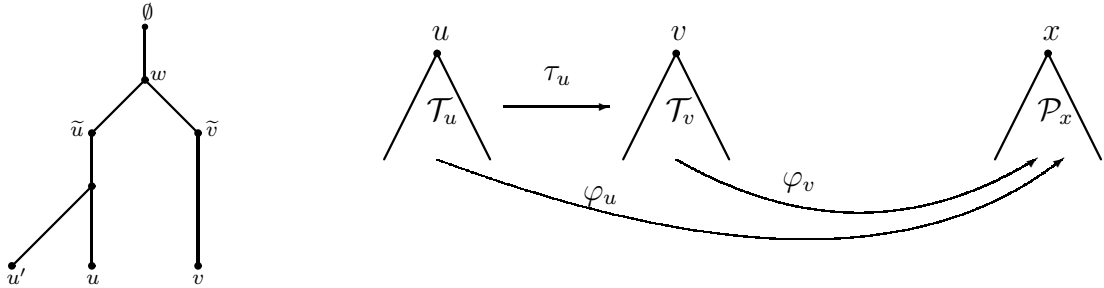
\begin{picture}(400,120)
\put(80,110){\circle*{2}}
\put(79,113){${\scriptstyle \emptyset}$}

\thicklines
\put(80,110){\line(0,-1){20}}
\put(80,90){\circle*{2}}
\put(82,90){${\scriptstyle w}$}
\put(80,90){\line(-1,-1){20}}
\put(60,70){\circle*{2}}
\put(52,69){${\scriptstyle \widetilde{u}}$}
\put(80,90){\line(1,-1){20}}
\put(100,70){\circle*{2}}
\put(103,69){${\scriptstyle \widetilde{v}}$}
\put(60,70){\line(0,-1){50}}
\put(60,20){\circle*{2}}
\put(58,13){${\scriptstyle u}$}
\put(60,50){\circle*{2}}
\put(60,50){\line(-1,-1){30}}
\put(30,20){\circle*{2}}
\put(28,13){${\scriptstyle u'}$}
\put(100,70){\line(0,-1){50}}
\put(100,20){\circle*{2}}
\put(98,13){${\scriptstyle v}$}

\put(190,100){\circle*{2}}
\put(188,105){$u$}
\thicklines
\put(190,100){\line(-1,-2){20}}
\put(190,100){\line(1,-2){20}}
\put(186,75){$\mathcal{T}_u$}

\put(280,100){\circle*{2}}
\put(278,105){$v$}
\thicklines
\put(280,100){\line(-1,-2){20}}
\put(280,100){\line(1,-2){20}}
\put(276,75){$\mathcal{T}_v$}

\put(420,100){\circle*{2}}
\put(418,105){$x$}
\thicklines
\put(420,100){\line(-1,-2){20}}
\put(420,100){\line(1,-2){20}}
\put(416,75){$\mathcal{P}_x$}

\thinlines
\put(230,90){$\tau_u$}
\put(215,80){\vector(1,0){40}}

\qbezier(280,60)(350,20)(415,60)
\put(415,60){\vector(4,3){1}}
\put(320,50){$\varphi_v$}

\qbezier(190,60)(350,0)(425,60)
\put(425,60){\vector(4,3){1}}
\put(245,45){$\varphi_u$}

\end{picture}
\end{proof}

In particular, if $\varphi$ is a tree composition, then all the maps $\varphi_u$, $u\in\mathcal{T}$ internal, are tree compositions and if $\varphi(u)=\varphi(v)$ then $\mathcal{T}_u$ and $\mathcal{T}_v$ are necessarily isomorphic. We denote by $\mathsf{C}(\mathcal{T},\mathcal{P})$ the set of all $\mathcal{P}$-compositions of $\mathcal{T}$.
The group $\Aut(\mathcal{T})\times \Aut(\mathcal{P})$ acts on $\mathsf{C}(\mathcal{T},\mathcal{P})$ in a natural way: if $\varphi\in\mathsf{C}(\mathcal{T},\mathcal{P})$, $g\in \Aut(\mathcal{T})$ and $k\in \Aut(\mathcal{P})$ we set 
\begin{equation}\label{action2Aut}
(g,k)\varphi\coloneqq k\circ\varphi\circ g^{-1}.
\end{equation}
The {\em identity composition} of $\mathcal{T}$ is given by the identity automorphism: $\id_\mathcal{T}\colon \mathcal{T}\rightarrow\mathcal{T}$.

\begin{example}\label{Examplephig}{\rm
We illustrate the notion of tree composition with a very important example, that will be used in Section \ref{Secconjclasses} to parametrize the conjugacy classes of $\Aut(\mathcal{T})$; the same construction is in \cite{GNS}, for infinite rooted trees, and a similar one in \cite[Section 9.3]{GoRo}, but with quite different applications. See also the notion of quotient graph in \cite[Section 3.1]{Serre}.
If  $g\in\Aut(\mathcal{T})$, $u,v\in\mathcal{T}$, $\{v,g(v),\dotsc,g^{\ell-1}(v)\}$ is a $g$-orbit at level $k-1$ and $\{u,g(u),\dotsc,g^{r-1}(u)\}$ is a $g$-orbit at level $k$ then we have only two possibilities: 
\begin{enumerate}
\item
for all $0\leq j\leq \ell-1$,  $0\leq i\leq r-1$, $g^i(u)$ is not a child of $g^j(v)$;
\item\label{seccase}
for all  $0\leq i\leq r-1$ there exists $0\leq j\leq \ell-1$ such that $g^i(u)$ is a child of $g^j(v)$.
\end{enumerate}
Note that, in the second case, $j$ is unique but a $j$ may have more than one child in the orbit of $u$ because $r$ may be greater than $\ell$: just consider the case $\ell=1$, that is $g$ fixes $v$ and it is transitive on the children of $v$. A more elaborated example is the so called adding machine \cite[Exercise 6.59]{CC}, which yields an automorphism of the $q$-ary tree which is transitive at every level.
Now to each $g\in\Aut(\mathcal{T})$ we associate a tree composition $\varphi^g\in\mathsf{C}\left(\mathcal{T},\mathcal{P}^g\right)$ in the following way. The vertices of $\mathcal{P}^g$ at level $k$, $k=1,2,\dotsc, h$, are the $g$-orbits of the elements of $\mathcal{T}$ at level $k$. Two vertices $x,y\in\mathcal{P}^g$, $x$ at level $k-1$ and $y$ at level $k$, are joined if the $g$-orbit corresponding to $y$ is made up of children of the vertices in the $g$-orbit corresponding to $x$ (cf, \ref{seccase}. above); $\mathcal{P}^g$ is a tree because every vertex has exactly one parent. Finally, for all $u\in\mathcal{T}$ we set
\begin{equation}\label{defphig}
\varphi^g(u)\coloneqq\text{ the }g\text{-orbit of }u.
\end{equation}
Clearly, $\varphi^g$ is a surjective rooted trees homomorphism. Moreover, if $u,v\in\mathcal{T}$, $u\neq v$, are internal and $\varphi^g(u)=\varphi^g(v)\eqqcolon x$ then there exists a positive integer $k$ such that:
\begin{equation}\label{starp25May} 
v=g^k(u)\quad\text{ and }\quad(\varphi^g)_u=(\varphi^g)_v\circ(g^k)_u.
\end{equation}
Indeed, $k$ exists because $u$ and $v$ belong to the same $g$-orbit (corresponding to $x$). Moreover, if $u'\in\mathcal{T}_u$ and $\varphi^g(u')=x'\in\left(\mathcal{P}^g\right)_{x}$ then $\varphi^g(g^k(u'))=x'$ because $x'$ corresponds to the $g$-orbit of $u'$ which clearly contains also $g^k(u')$. Then we have
$\varphi^g(g^k(u'))=\varphi^g(u')$ and this ends the proof of \eqref{starp25May} and $\varphi^g$ is a tree composition.
}
\end{example}

\begin{remark}{\rm
If $\varphi\in\mathsf{C}(\mathcal{T},\mathcal{P})$ and $\psi\in\mathsf{C}(\mathcal{P},\mathcal{S})$ then $\psi\circ\varphi\colon\mathcal{T}\rightarrow\mathcal{S}$ is not necessarily a tree composition. Indeed, consider the trees below and set $\varphi(w_1)=\varphi(w_2)=x_1$, $\varphi(w_3)=x_2$ and $\psi(x_1)=\psi(x_2)=s$.
Then $\varphi$ and $\alpha$ are tree compositions but $\alpha\circ\varphi$ is not a tree composition because $\alpha\circ\varphi(u)=\alpha\circ\varphi(v)$ but $\mathcal{T}_u$ and $\mathcal{T}_v$ are not isomorphic.

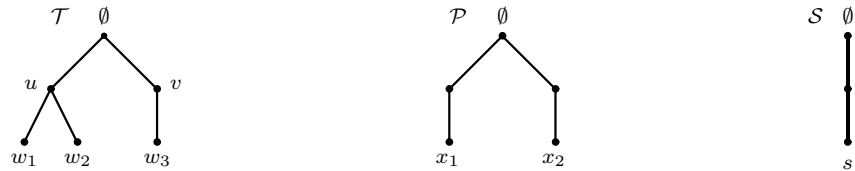
\begin{picture}(400,80)
\put(80,65){${\scriptstyle\mathcal{T}}$}
\put(100,60){\circle*{2}}
\put(98,65){${\scriptstyle\emptyset}$}
\thicklines
\drawline(100,60)(80,40)
\put(80,40){\circle*{2}}
\put(70,40){${\scriptstyle u}$}
\drawline(100,60)(120,40)
\put(120,40){\circle*{2}}
\put(125,40){${\scriptstyle v}$}
\drawline(80,40)(70,20)
\put(70,20){\circle*{2}}
\put(65,12){${\scriptstyle w_1}$}
\drawline(80,40)(90,20)
\put(90,20){\circle*{2}}
\put(85,12){${\scriptstyle w_2}$}

\drawline(120,40)(120,20)
\put(120,20){\circle*{2}}
\put(115,12){${\scriptstyle w_3}$}

\put(230,65){${\scriptstyle\mathcal{P}}$}
\put(250,60){\circle*{2}}
\put(248,65){${\scriptstyle\emptyset}$}
\thicklines
\drawline(250,60)(230,40)
\put(230,40){\circle*{2}}
\drawline(250,60)(270,40)
\put(270,40){\circle*{2}}
\drawline(230,40)(230,20)
\put(230,20){\circle*{2}}
\put(225,12){${\scriptstyle x_1}$}

\drawline(270,40)(270,20)
\put(270,20){\circle*{2}}
\put(265,12){${\scriptstyle x_2}$}

\put(365,65){${\scriptstyle\mathcal{S}}$}
\put(378,65){${\scriptstyle\emptyset}$}
\put(380,60){\circle*{2}}
\thicklines
\put(380,60){\line(0,-1){40}}
\put(380,40){\circle*{2}}
\put(380,20){\circle*{2}}
\put(378,10){${\scriptstyle s}$}
\end{picture}
}
\end{remark}

\begin{definition}\label{defvarie}{\rm
If $\varphi_j\in\mathsf{C}(\mathcal{T}_j,\mathcal{P}_j)$, $j=1,2$, are two tree compositions an {\em equivalence} between them is given by a couple of rooted trees isomorphisms $\tau \colon\mathcal{T}_1\rightarrow\mathcal{T}_2$, $\omega \colon\mathcal{P}_1\rightarrow\mathcal{P}_2$ such that $\omega\circ \varphi_1=\varphi_2\circ \tau$. 
If $\mathcal{P}_1=\mathcal{P}_2\eqqcolon\mathcal{P}$ and $\omega=\id_\mathcal{P}$ we say that $\varphi_1$ and $\varphi_2$ are {\em strictly equivalent}. If $\mathcal{T}_1=\mathcal{T}_2\eqqcolon\mathcal{T}$ and $\alpha=\id_\mathcal{T}$ we say that $\varphi_1$ and $\varphi_2$ are {\em isomorphic}.
}
\end{definition}

Definition \ref{defvarie} may be illustrated by means of the following commutative diagrams:
\begin{equation}\label{defequiviso}\xymatrix{
&\mathcal{T}_1\ar[r]^{\varphi_1}\ar[d]_{\tau} & \mathcal{P}_1\ar[d]_{\omega}&&&\mathcal{T}_1\ar[rd]^{\varphi_1}\ar[d]_{\tau} &&&&&\mathcal{P}_1\ar[d]^{\omega}\\
&\mathcal{T}_2\ar[r]_{\varphi_2}_*++++\txt{equivalence}& \mathcal{P}_2&&&\mathcal{T}_2\ar[r]_{\varphi_2}_*++++\txt{strict equivalence}&\mathcal{P}&&&\mathcal{T}\ar[ru]^{\varphi_1}\ar[r]_{\varphi_2}_*++++\txt{isomorphism}&\mathcal{P}_2\\
}
\end{equation}
Clearly both strict equivalence and isomorphism imply equivalence.
Moreover, the regularity condition in Definition \ref{deftreecomp} and Lemma \ref{lemmaregcond} is a strict equivalence: if $x\in\mathcal{P}$ is internal, then all the tree compositions $\varphi_u\colon\mathcal{T}_u\rightarrow\mathcal{P}_x$, $u\in\varphi^{-1}(x)$, are strictly equivalent; that is, $\mathcal{P}_x$ determines a precise way to ``decompose" the subtrees that map to it.

\begin{example}\label{Remobviousfacts2}
{\rm
The orbits of $\Aut(\mathcal{T})\times \Aut(\mathcal{P})$ on $\mathsf{C}(\mathcal{T},\mathcal{P})$ with respect to the action \eqref{action2Aut} coincide with the equivalence classes of $\mathcal{P}$-compositions of $\mathcal{T}$; the orbits of $\Aut(\mathcal{T})$ coincide with the strict equivalence classes; the obits of $\Aut(\mathcal{P})$ coincide with the isomorphism classes. See also Example \ref{exampleh1} and Theorem \ref{Theoorbits}. Finally, note that the strict equivalence class of the identity composition $\id_\mathcal{T}$ is made up of all $g\in\Aut(\mathcal{T})$.
}
\end{example}

\begin{remark}\label{Remobviousfacts}
{\rm The fact that the equivalence of tree compositions is transitive is obvious, but we state it explicitly for future reference, together with some similar facts. Suppose that $\varphi_1,\varphi_2$ are equivalent as in \eqref{defequiviso}. If $\varphi_3\in\mathsf{C}(\mathcal{T}_3,\mathcal{P}_3)$ is equivalent to $\varphi_2$ then also $\varphi_1$ and $\varphi_3$ 
 are equivalent.
If $\alpha\in\mathsf{C}(\mathcal{P}_2,\mathcal{Q})$ and $\alpha\circ\varphi_2\colon\mathcal{T}_2\rightarrow\mathcal{Q}$ is a tree composition then also $\alpha\circ\omega\circ\varphi_1\colon\mathcal{T}_1\rightarrow\mathcal{Q}$ is a tree composition.
Another case: if $\varphi'_j$ is strictly equivalent to $\varphi_j$ , $j=1,2$, then $\varphi_1'$ and $\varphi_2'$ are equivalent. Finally, note that if $\beta\in\Aut(\mathcal{P}_1)$ then $\varphi_1$ and $\beta\circ\varphi_1$ are isomorphic.
}
\end{remark}

We recall that if $n$ is a positive integer, an {\em integer composition} of $n$ into $k$-parts is a $k$-tuple of positive integers $(a_1,a_2,\dotsc,a_k)$ such that $a_1+a_2+\dotsc+a_k=n$; an {\em integer partition} of $n$ into $k$ parts is a $k$-tuple of positive integers $(\lambda_1,\lambda_2,\dotsc,\lambda_k)$ such that $\lambda_1+\lambda_2+\dotsc+\lambda_k=n$ but also $\lambda_1\geq\lambda_2\geq\dotsb\geq\lambda_k$; see \cite[Chapter 1]{Sagan}). The {\em length} of $\lambda$ is the number $k$ of its parts and it will be denoted by $\ell(\lambda)\coloneqq k$. Moreover, $\lvert\lambda\rvert\coloneqq\lambda_1+\lambda_2+\dotsb+\lambda_k$ denotes the sum of its parts. That is, 
\begin{equation}\label{partnotat}
\lambda\text{ is a partition of }\lvert\lambda\rvert\text{ into }\ell(\lambda)\text{ parts.}
\end{equation} 
If $Y$ is a finite set, $\Sym(Y)$ denotes the group of all permutations of $Y$; cf. \cite[p. 4]{GoRo}. 

\begin{example}\label{exampleh1}
{\rm Suppose that $H(\mathcal{T})=1$, identify $\mathcal{T}$ with $V\coloneqq V_\mathcal{T}$ and $\mathcal{P}$ with $X\coloneqq V_\mathcal{P}$.
Then a $\mathcal{P}$-composition of $\mathcal{T}$ is just a surjective map $\varphi\colon V\rightarrow X$. It determines a  {\em set composition} (=an ordered partition: \cite[Section 2.2]{Sagan}) of $V$ into $\lvert X\rvert$ parts: 
\begin{equation}\label{setcomp}
V=\bigsqcup\limits_{x\in X}\varphi^{-1}(x).
\end{equation}
In other words, the elements of $X$ are distinguishable and if we order these elements linearly, say $X=\{x_1,x_2,\dotsc, x_k\}$, then we have an ordered sequence of non empty, pairwise disjoint subsets of $V$: $\left(\varphi^{-1}(x_1),\varphi^{-1}(x_2),\dotsc,\varphi^{-1}(x_k)\right)$
whose union is the whole $V$.
See entry 3 of the Twelvefold way in \cite[Section 1.4]{Stanley} and \cite[Sections 1.4]{Sagan}. If we set $a_k=\left\lvert\varphi^{-1}(x_k)\right\rvert$ then the integer composition $(a_1,a_2,\dotsc,a_k)$ is the {\em type} of the set composition \eqref{setcomp}. In most cases, we will avoid the numbering and we will use \eqref{setcomp} to indicate a set composition. 
We define the {\em quotient } of $\varphi:V\rightarrow X$ by taking a singleton $\{a\}$ and setting $\alpha:X\rightarrow \{a\}$, that is $\alpha(x)=a$ for all $x\in X$. This a trivial case of a more general construction in Section \ref{Secquotient}.

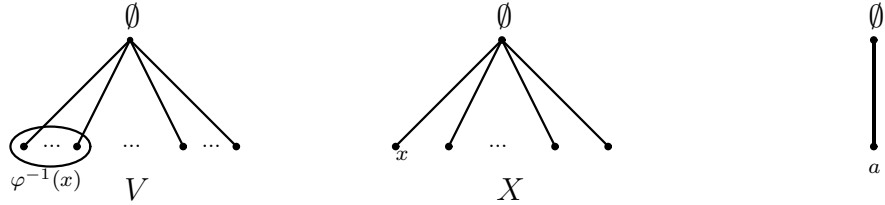
\begin{picture}(400,80)
\put(100,60){\circle*{2}}
\put(98,65){$\emptyset$}
\thicklines
\put(100,60){\line(-1,-1){40}}
\put(60,20){\circle*{2}}
\put(67,20){${\scriptstyle \dotsc}$}
\put(70,20){\ellipse{30}{15}}
\put(55,5){${\scriptstyle \varphi^{-1}(x)}$}
\put(100,60){\line(-1,-2){20}}
\put(80,20){\circle*{2}}
\put(97,20){${\scriptstyle \dotsc}$}
\put(100,60){\line(1,-2){20}}
\put(120,20){\circle*{2}}
\put(127,20){${\scriptstyle \dotsc}$}
\put(100,60){\line(1,-1){40}}
\put(140,20){\circle*{2}}
\put(98,0){$V$}

\put(238,65){$\emptyset$}
\put(240,60){\circle*{2}}
\thicklines
\put(240,60){\line(-1,-1){40}}
\put(200,20){\circle*{2}}
\put(200,15){${\scriptstyle x}$}
\put(240,60){\line(-1,-2){20}}
\put(220,20){\circle*{2}}
\put(235,20){${\scriptstyle \dotsc}$}
\put(240,60){\line(1,-2){20}}
\put(260,20){\circle*{2}}
\put(240,60){\line(1,-1){40}}
\put(280,20){\circle*{2}}
\put(238,0){$X$}


\put(378,65){$\emptyset$}
\put(380,60){\circle*{2}}
\thicklines
\put(380,60){\line(0,-1){40}}
\put(380,20){\circle*{2}}
\put(378,10){${\scriptstyle a}$}
\end{picture}

\noindent
Now $\mathsf{C}(V,X)$ is the set of all surjective maps $\varphi\colon V\rightarrow X$ and the automorphisms groups are just $\Sym(V)$ and $\Sym(X)$; see \cite[Chapter 1, Section 2]{Aigner}. As a particular case of Example \ref{Remobviousfacts2}, we get:
the orbits of $\Sym(V)$ coincide with the strict equivalence classes and their set may be parametrized by the {\em integer compositions} of $\lvert V\rvert$ into $\lvert X\rvert$ parts;  
the orbits of $\Sym(X)$ coincide with the isomorphism classes and also with the (unordered!) {\em partitions} of $V$ into $\lvert X\rvert$ parts;
 the orbits of the whole $\Sym(V)\times\Sym(X)$ are the equivalence classes and they may be parametrized by the {\em integer partitions} of $\lvert V\rvert$ into $\lvert X\rvert$ parts.
That is if $\varphi\in \mathcal{C}(V,X)$ and $x_1,x_2,\dotsc,x_k$ are the elements of $X$ numbered in such a way that  $\left\lvert \varphi^{-1}(x_1)\right \rvert\geq \left\lvert \varphi^{-1}(x_2)\right \rvert\geq\dotsb\geq \left\lvert \varphi^{-1}(x_k)\right \rvert$, then we set $\lambda_i\coloneqq\left\lvert \varphi^{-1}(x_i)\right \rvert$ and $\lambda=(\lambda_1,\lambda_2,\dotsc,\lambda_k)$ is the integer partition of $n\coloneqq\lvert V\rvert$ associated to the $\left(\Sym(V)\times\Sym(X)\right)$-orbit of $\varphi$. We also say that 
$V=\{\varphi^{-1}(x)\colon x\in X\}$
is a partition of $V$ of {\em type} $\lambda$ and this will be the standard notation to indicate an unordered partition of a set. 
}
\end{example}

\begin{example}\label{Examplesubtrees}{\rm
Suppose that $\mathcal{T}$ is a spherically homogeneous rooted tree of branching type ${\bf r}=(r_1,r_2, \ldots, r_h)$, choose integers $1\leq s_k\leq r_k$, $k=1,2,\dotsc, h$, with $s_k< r_k$ for at least one $k$, and set ${\bf s}=(s_1,s_2, \ldots, s_h)$. Define $\mathsf{V}({\bf r}, {\bf s})$ as the variety of all subtrees $\mathcal{S}\subseteq\mathcal{T}$ which are spherically homogeneous of branching type ${\bf s}$. In \cite{CST3} and \cite[Section 3.2.6]{book3} it is analyzed the permutation representation of $\Aut(\mathcal{T})$ on $\mathsf{V}({\bf r}, {\bf s})$, as an application of a more general construction, the generalized Johnson scheme. In particular, this permutation representation decomposes without multiplicity (i.e. gives rise to a Gelfand pair) and it is a rooted tree analogous of the permutation representation of $S_n$ on $S_n/(S_{n-m}\times S_m)$, also called Johnson scheme; \cite[Chapter 6]{book}. Now we want to describe the trees in $\mathsf{V}({\bf r}, {\bf s})$ by means of tree compositions. First of all, note that a tree $\mathcal{S}\in \mathsf{V}({\bf r}, {\bf s})$ has the following recursive description:
$\emptyset\in\mathcal{S}$ and if $u$ is a vertex of $\mathcal{S}$ at level $k-1$ then $\text{\rm Ch}_\mathcal{S}(u)$ is a subset of $\text{\rm Ch}_\mathcal{T}(u)$ of cardinality $s_k$, $k=1,2,\dotsc,h$.
Then we define the rooted tree $\mathcal{P}$ of height $h$ inductively as follows:
\begin{itemize}
\item
denote the root by $b_0$; for $k=1,2,\dotsc,h$, if $s_k<r_k$ then $\text{\rm Ch}_\mathcal{P}(b_{k-1})=\{a_k,b_k\}$ and $\mathcal{P}_{a_k}$ is a path of length $h-k$; if $s_k=r_k$ then $\text{\rm Ch}_\mathcal{P}(b_{k-1})=\{b_k\}$.
\end{itemize}

\noindent
Finally, to each $\mathcal{S}\in \mathsf{V}({\bf r}, {\bf s})$ we associate the tree composition $\varphi^\mathcal{S}\in\mathsf{C}(\mathcal{T},\mathcal{P})$ defined inductively in the following way:
\begin{itemize}
\item
if $u$ is a vertex of $\mathcal{S}$ at level $k-1$ then we set $\varphi^\mathcal{S}\left(\text{\rm Ch}_\mathcal{S}(u)\right)\coloneqq b_k$, while for all $v\in\text{\rm Ch}_\mathcal{T}(u)\setminus\text{\rm Ch}_\mathcal{S}(u)$ (if it is not empty) we set $\varphi^\mathcal{S}\left(T_v\right)\coloneqq \mathcal{P}_{a_k}$, $k=1,2,\dotsc,h$. 
\end{itemize}

\noindent
In the following picture, ${\bf r}=(3,3,2)$, ${\bf s}=(2,2,1)$ and the subtree $\mathcal{S}$ and its image in $\mathcal{P}$ are indicated by bolder lines and vertices.

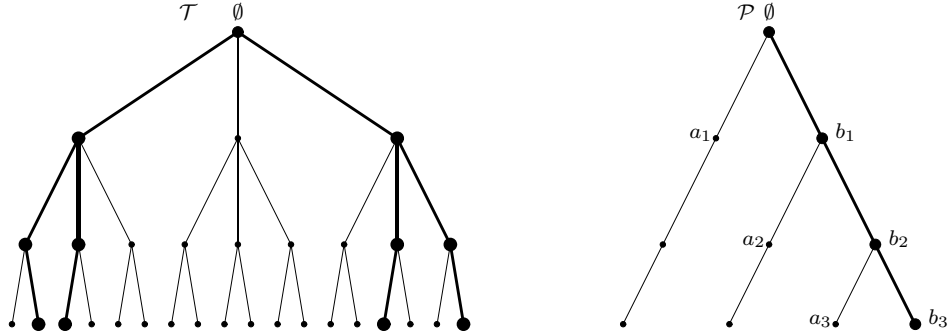
\begin{picture}(400,120)
\put(150,110){\circle*{4}}
\put(148,115){${\scriptstyle\emptyset}$}
\put(128,115){${\scriptstyle\mathcal{T}}$}
\Thicklines
\put(150,110){\line(-3,-2){60}}
\put(90,70){\circle*{4}}
\thinlines
\put(150,110){\line(0,-1){40}}
\put(150,70){\circle*{2}}
\Thicklines
\put(150,110){\line(3,-2){60}}
\put(210,70){\circle*{4}}
\thinlines

\Thicklines
\put(90,70){\line(-1,-2){20}}
\put(70,30){\circle*{4}}
\put(90,70){\line(0,-1){40}}
\put(90,30){\circle*{4}}
\thinlines
\put(90,70){\line(1,-2){20}}
\put(110,30){\circle*{2}}

\put(150,70){\line(-1,-2){20}}
\put(130,30){\circle*{2}}
\put(150,70){\line(0,-1){40}}
\put(150,30){\circle*{2}}
\put(150,70){\line(1,-2){20}}
\put(170,30){\circle*{2}}

\put(210,70){\line(-1,-2){20}}
\put(190,30){\circle*{2}}
\Thicklines
\put(210,70){\line(0,-1){40}}
\put(210,30){\circle*{4}}
\put(210,70){\line(1,-2){20}}
\put(230,30){\circle*{4}}
\thinlines


\put(70,30){\line(-1,-6){5}}
\put(65,0){\circle*{2}}
\Thicklines
\put(70,30){\line(1,-6){5}}
\put(75,0){\circle*{4}}
\thinlines


\Thicklines
\put(90,30){\line(-1,-6){5}}
\put(85,0){\circle*{4}}
\thinlines
\put(90,30){\line(1,-6){5}}
\put(95,0){\circle*{2}}


\put(110,30){\line(-1,-6){5}}
\put(105,0){\circle*{2}}
\put(110,30){\line(1,-6){5}}
\put(115,0){\circle*{2}}


\put(130,30){\line(-1,-6){5}}
\put(125,0){\circle*{2}}
\put(130,30){\line(1,-6){5}}
\put(135,0){\circle*{2}}


\put(150,30){\line(-1,-6){5}}
\put(145,0){\circle*{2}}
\put(150,30){\line(1,-6){5}}
\put(155,0){\circle*{2}}


\put(170,30){\line(-1,-6){5}}
\put(165,0){\circle*{2}}
\put(170,30){\line(1,-6){5}}
\put(175,0){\circle*{2}}


\put(190,30){\line(-1,-6){5}}
\put(185,0){\circle*{2}}
\put(190,30){\line(1,-6){5}}
\put(195,0){\circle*{2}}


\Thicklines
\put(210,30){\line(-1,-6){5}}
\put(205,0){\circle*{4}}
\thinlines
\put(210,30){\line(1,-6){5}}
\put(215,0){\circle*{2}}


\put(230,30){\line(-1,-6){5}}
\put(225,0){\circle*{2}}
\Thicklines
\put(230,30){\line(1,-6){5}}
\put(235,0){\circle*{4}}
\thinlines


\put(350,110){\circle*{4}}
\put(348,115){${\scriptstyle\emptyset}$}
\put(338,115){${\scriptstyle\mathcal{P}}$}
\Thicklines
\put(350,110){\line(1,-2){55}}
\thinlines
\put(350,110){\line(-1,-2){55}}

\put(330,70){\circle*{2}}
\put(320,70){${\scriptstyle a_1}$}
\put(370,70){\circle*{4}}
\put(370,70){\line(-1,-2){35}}
\put(375,70){${\scriptstyle b_1}$}

\put(310,30){\circle*{2}}
\put(350,30){\circle*{2}}
\put(340,30){${\scriptstyle a_2}$}
\put(390,30){\circle*{4}}
\put(395,30){${\scriptstyle b_2}$}
\put(390,30){\line(-1,-2){15}}

\put(295,0){\circle*{2}}
\put(335,0){\circle*{2}}
\put(375,0){\circle*{2}}
\put(365,0){${\scriptstyle a_3}$}
\put(405,0){\circle*{4}}
\put(410,0){${\scriptstyle b_3}$}

\end{picture}

\noindent
In other words, if $v\in\mathcal{S}$ then $\varphi^\mathcal{S}(v)$ is a $b$-vertex, otherwise it is a descendant of a suitable $a$-vertex. We set $\Omega({\bf r}, {\bf s})\coloneqq\left\{\varphi^\mathcal{S}\colon \mathcal{S}\in\mathsf{V}({\bf r}, {\bf s})\right\}$ and clearly the map 
$\mathsf{V}({\bf r}, {\bf s})\ni\mathcal{S}\longmapsto\varphi^\mathcal{S}\in\Omega({\bf r}, {\bf s})$
is a bijection. Moreover, for all $g\in\Aut(\mathcal{T})$ and $\mathcal{S}\in \mathsf{V}({\bf r}, {\bf s})$, we have (cf. \eqref{action2Aut})
\[
\varphi^{g\mathcal{S}}=\varphi^\mathcal{S}\circ g^{-1}.
\]
Indeed, for all $u\in\mathcal{T}$, $\left(\varphi^\mathcal{S}\circ g^{-1}\right)(u)$  is a $b\text{-vertex }\Longleftrightarrow$ $g^{-1}(u)\in\mathcal{S} \Longleftrightarrow$ $u\in g\mathcal{S} \Longleftrightarrow$ $\varphi^{g\mathcal{S}}(u)$ is a $b\text{-vertex}$. 
Therefore the bijection $\mathcal{S}\mapsto\varphi^\mathcal{S}$ is equivariant, $\mathsf{V}({\bf r}, {\bf s})$ and $\Omega({\bf r}, {\bf s})$ are isomorphic as homogeneous $\Aut(\mathcal{T})$-spaces and the respective permutation representations are equivalent. 
}
\end{example}

\begin{example}\label{Exampleflags}{\rm
Now suppose again that $\mathcal{T}$ is a spherically homogeneous rooted tree of branching type ${\bf r}=(r_1,r_2, \ldots, r_h)$ and also that ${\bf S}=\left[s_i^j\right]_{1\leq i\leq h,\; 1\leq j\leq l}$
is an $h\times l$ matrix of positive integers such that $1\leq s_k^l\leq s_k^{l-1}\leq\dotsb\leq s_k^1\leq s_k^0= r_k$ for $k=1,2,\dotsc, h$; we assume that for each $0\leq j\leq l-1$ there exists at least one $k$ such that $s_k^{j+1}<s_k^j$. Then a {\em flag} $\mathcal{F}$ {\em of subtrees of} $\mathcal{T}$ {\em of type} ${\bf S}$ is an increasing sequence subtrees
\begin{equation}\label{flag}
\mathcal{F}\colon\;\mathcal{S}^{l+1}\coloneqq\{\emptyset\}\subseteq\mathcal{S}^l\subseteq\mathcal{S}^{l-1}\subseteq\dotsb\subseteq \mathcal{S}^1\subseteq\mathcal{S}^0\coloneqq \mathcal{T}
\end{equation}
such that $\mathcal{S}^j$ is spherically homogeneous of branching type $\mathcal{\bf s}^j\coloneqq (s_1^j,s_2^j\dotsc,s_h^j)$, $j=0,1,\dotsc,l-1,l$ (and $\mathcal{S}^{l+1}$ is just the root). We denote by $\mathsf{V}({\bf r},{\bf S})$ the variety of all flags $\mathcal{F}$ as in \eqref{flag}. The flag $\mathcal{F}$ has the following iterative description:
\begin{itemize}
\item
$V_\mathcal{T}=\bigsqcup\limits_{j=0}^{l}\left(V_{\mathcal{S}^j}\setminus V_{\mathcal{S}^{j+1}}\right)\quad$ and $\quad\lvert V_{\mathcal{S}^j}\rvert=s_1^j\;$, $j=0,1,\dotsc,l$.
\item
For $k=2,3\dotsc,h$, if $u\in\mathcal{S}^j\setminus\mathcal{S}^{j+1}$ with $0\leq j\leq l$ and $u$ is at level $k-1$ then 
\[
\text{\rm Ch}_\mathcal{T}(u)=\text{\rm Ch}_{\mathcal{S}^j}(u)\bigsqcup\left[\bigsqcup\limits_{i=0}^{j-1}\left(\text{\rm Ch}_{\mathcal{S}^i}(u)\setminus \text{\rm Ch}_{\mathcal{S}^{i+1}}(u)\right)\right]\quad\text{and}\quad\lvert \text{\rm Ch}_{\mathcal{S}^i}(u)\rvert=s_k^i,\; 0\leq i\leq j. 
\]
\end{itemize}

\noindent
After that we define a rooted tree $\mathcal{P}$ of height $h$ as follows. The vertices of $\mathcal{P}\setminus \{\emptyset\}$ are labeled with the integers $0,1,\dotsc,l$, distinct siblings have different labels and the label of $x\in \mathcal{P}$ is denoted by $\ell(x)$. Then inductively
\begin{itemize}
\item
$V_\mathcal{P}$ has $l+1$ vertices, labeled $0,1,\dotsc,l$;
\item
for $k=2,3,\dotsc,h$, if we have constructed $\mathcal{P}$ up to level $k-1$, $x$ is a vertices at level $k-1$ and $\ell(x)=j$ then $x$ has $j+1$ children, labeled $0,1,2,\dotsc,j$;
\item
if $s_k^{i-1}=s_k^i$, $1\leq k\leq h$, $1\leq i\leq l$, then we delete all the subtrees $\mathcal{P}_x$ such that $x$ is at level $k$ and $\ell(x)=i-1$.
\end{itemize}

\noindent
The following is the tree $\mathcal{P}$ for $h=3$ and $l=2$.

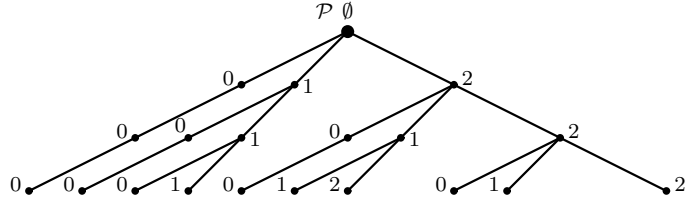
\begin{picture}(400,70)
\thicklines


\put(200,60){\circle*{4}}
\put(198,65){${\scriptstyle\emptyset}$}
\put(188,65){${\scriptstyle\mathcal{P}}$}

\put(200,60){\line(2,-1){120}}

\put(200,60){\line(-2,-1){120}}
\put(200,60){\line(-1,-1){60}}

\put(160,40){\circle*{2}}
\put(153,40){${\scriptstyle 0}$}
\put(180,40){\circle*{2}}
\put(183,37){${\scriptstyle 1}$}
\put(180,40){\line(-2,-1){80}}
\put(240,40){\circle*{2}}
\put(243,40){${\scriptstyle 2}$}
\put(240,40){\line(-2,-1){80}}
\put(240,40){\line(-1,-1){40}}
\put(120,20){\circle*{2}}
\put(113,20){${\scriptstyle 0}$}
\put(140,20){\circle*{2}}
\put(135,22){${\scriptstyle 0}$}
\put(160,20){\circle*{2}}
\put(163,17){${\scriptstyle 1}$}
\put(160,20){\line(-2,-1){40}}
\put(200,20){\circle*{2}}
\put(193,20){${\scriptstyle 0}$}
\put(220,20){\circle*{2}}
\put(223,17){${\scriptstyle 1}$}
\put(220,20){\line(-2,-1){40}}
\put(280,20){\circle*{2}}
\put(283,20){${\scriptstyle 2}$}
\put(280,20){\line(-2,-1){40}}
\put(280,20){\line(-1,-1){20}}

\put(80,0){\circle*{2}}
\put(73,0){${\scriptstyle 0}$}
\put(100,0){\circle*{2}}
\put(93,0){${\scriptstyle 0}$}
\put(120,0){\circle*{2}}
\put(113,0){${\scriptstyle 0}$}
\put(140,0){\circle*{2}}
\put(133,0){${\scriptstyle 1}$}
\put(160,0){\circle*{2}}
\put(153,0){${\scriptstyle 0}$}
\put(180,0){\circle*{2}}
\put(173,0){${\scriptstyle 1}$}
\put(200,0){\circle*{2}}
\put(193,0){${\scriptstyle 2}$}
\put(240,0){\circle*{2}}
\put(233,0){${\scriptstyle 0}$}
\put(260,0){\circle*{2}}
\put(253,0){${\scriptstyle 1}$}
\put(320,0){\circle*{2}}
\put(323,0){${\scriptstyle 2}$}

\end{picture}

\noindent
To each flag $\mathcal{F}$ we associate a tree composition $\varphi^\mathcal{F}\colon\mathcal{T}\rightarrow\mathcal{P}$ defined recursively

\begin{itemize}
\item
for $j=0,1,\dotsc,l$, $\varphi^\mathcal{F}\left(V_{\mathcal{S}^j}\setminus V_{\mathcal{S}^{j+1}}\right)\coloneqq$ the vertex of $V_\mathcal{P}$ with label $j$;
\item 
for $k=2,3,\dotsc,h$, if we have defined $\varphi^\mathcal{F}$ up to level $k-1$, $u$ is at level $k$ of $\mathcal{T}$, $v$ is the parent of $u$, $\varphi^\mathcal{F}(v)=x$ and $u\in \mathcal{S}^j\setminus \mathcal{S}^{j+1}$ then $\varphi^\mathcal{F}(u)\coloneqq$ the child of $x$ with label $j$.
\end{itemize}

\noindent
In other words, starting from the first level, the vertices in $\mathcal{S}^j\setminus \mathcal{S}^{j+1}$ are sent to the vertices of $\mathcal{P}$ with label $j$.
It is easy to check that $\varphi^\mathcal{F}$ is a tree composition; if we set $\Omega({\bf r},{\bf S})\coloneqq\left\{\varphi^\mathcal{F}\colon\mathcal{F}\in\mathsf{V}({\bf r},{\bf S})\right\}$ then the map 
$\mathsf{V}({\bf r},{\bf S})\ni\mathcal{F}\longmapsto\varphi^\mathcal{F}\in\Omega({\bf r},{\bf S})$
is a bijection. Moreover, for all $g\in\Aut(\mathcal{T})$ and $\mathcal{F}\in \mathsf{V}({\bf r}, {\bf S})$, we have 
\[
\varphi^{g\mathcal{F}}=\varphi^\mathcal{F}\circ g^{-1}.
\]
Indeed, for all $u\in\mathcal{T}$,
$\;\ell\left(\left(\varphi^\mathcal{F}\circ g^{-1}\right)(u)\right)= j\Longleftrightarrow$ $g^{-1}(u)\in\mathcal{S}^j\setminus \mathcal{S}^{j+1}\Longleftrightarrow$ $u\in \left(g\mathcal{S}^j\right)\setminus\left(g\mathcal{S}^{j+1}\right)
 \Longleftrightarrow$ $\ell\left(\left(\varphi^{g\mathcal{F}}\right)(u)\right)=j$.
That is, the bijection $\mathcal{F}\mapsto\varphi^\mathcal{F}$ is equivariant, $\mathsf{V}({\bf r},{\bf S})$ and $\Omega({\bf r},{\bf S})$ are isomorphic as homogeneous $\Aut(\mathcal{T})$-spaces and the respective permutation representations are equivalent. 
}
\end{example}

\section{The quotient of a tree composition}\label{Secquotient}

In this section we introduce the fundamental construction of the quotient composition. We need a preliminary notion. Let $\mathcal{T}$ be a rooted tree with $H(\mathcal{T})\geq2$ and $V_\mathcal{T}\setminus W_\mathcal{T}=\bigsqcup_{a\in A}U_a$ a set composition of its first level with the property that: if $u,v$ belong to the same part then $\mathcal{T}_u$ and $\mathcal{T}_v$ are isomorphic. A {\em first level set of isomorphisms adapted to this composition} is a family $\tau_{u,v}\colon\mathcal{T}_v\rightarrow\mathcal{T}_u$, $u,v\in U_a,a\in A$, of isomorphisms such that and if also $t\in U_a$ then
\begin{equation}\label{proptau2b}
\tau_{w,u}\circ\tau_{u,t}=\tau_{w,t},\qquad\tau_{u,w}^{-1}=\tau_{w,u}\quad\text{\rm and }\quad\tau_{u,u}=\id_{\mathcal{T}_u}\quad (\text{\rm composition properties}).
\end{equation}

\begin{theorem}\label{Prop1p5}
If $\varphi\in\mathsf{C}(\mathcal{T}, \mathcal{P})$ then there exist a tree $\mathcal{Q}$ and a composition $\alpha\in\mathsf{C}(\mathcal{P},\mathcal{Q})$ (that we call {\em the quotient of} $\varphi$) such that
\begin{enumerate}[\rm(a)]
\item\label{Prop1p51}
if $x\in\mathcal{P}$ is internal and $Y$ is the set of all children of $x$ which are leaves then $\alpha_x(Y)=a$, a singleton $a\in \mathcal{Q}$
which is the unique child of $\alpha(x)$ which is a leaf.

\item\label{Prop1p52} If $a,b\in \mathcal{Q}$ are {\em distinct siblings} which are are internal, then for all pairs of siblings $u,v\in\mathcal{T}$ such that $\alpha(\varphi(u))=a$, $\alpha(\varphi(v))=b$ the tree compositions $\varphi_u\colon\mathcal{T}_u\rightarrow \mathcal{P}_{\varphi(u)}$ and $\varphi_v\colon\mathcal{T}_v\rightarrow \mathcal{P}_{\varphi(v)}$ {\em are not} equivalent.

\item\label{Prop1p53}
If $u,w\in\mathcal{T}$ are internal and $\alpha(\varphi(u))=\alpha(\varphi(w))\eqqcolon a$ then the tree compositions $\varphi_u\colon\mathcal{T}_u\rightarrow \mathcal{P}_{\varphi(u)}$ and $\varphi_w\colon\mathcal{T}_w\rightarrow \mathcal{P}_{\varphi(w)}$ are equivalent while $\alpha_{\varphi(u)}\circ \varphi_u\colon\mathcal{T}_u\rightarrow \mathcal{Q}_{\alpha(\varphi(u))}$ and $\alpha_{\varphi(w)}\circ \varphi_w\colon\mathcal{T}_w\rightarrow \mathcal{Q}_{\alpha(\varphi(w))}$ are strictly equivalent. In particular, $\alpha\circ \varphi\colon\mathcal{T}\rightarrow\mathcal{Q}$ is a tree composition.

\item\label{Prop1p53bis}
There exist first level sets of isomorphisms $\tau_{u,v}$, adapted to the set composition $V_\mathcal{T}=\bigsqcup_{a\in V_\mathcal{Q}}(\alpha\circ\varphi)^{-1}(a)$, and $\omega_{x,y}$, adapted to the set composition $V_\mathcal{P}=\bigsqcup_{a\in V_\mathcal{Q}}\alpha^{-1}(a)$, such that such that the following diagram is commutative:
\begin{equation}\label{diagTRSuv}
\xymatrix@R=2pt{
\mathcal{T}_u\ar[rr]^{\varphi_u}\ar[dd]_{\tau_{w,u}} && \mathcal{P}_{\varphi(u)}\ar[dr]^{\alpha_{\varphi(u)}}\ar[dd]_{\omega_{\varphi(w),\varphi(u)}}&\\
&&&\mathcal{Q}_a\\
\mathcal{T}_w\ar[rr]^{\varphi_w} && \mathcal{P}_{\varphi(w)}\ar[ur]_{\alpha_{\varphi(w)}},&
}
\end{equation}
for all $a\in V_\mathcal{Q}\setminus W_\mathcal{Q}$, $u,v\in(\alpha\circ\varphi)^{-1}(a)$. In particular, 
\begin{equation}\label{diagTRSuv2b}
\text{ if } u,v\in V_\mathcal{T}\text{ and }\varphi(u)=\varphi(w)\text{ then }\tau_{u,w}\text{ is a strict equivalence}.
\end{equation}

\item\label{Prop1p54}
The quotient is {\em unique up to isomorphisms}: if $\beta\colon\mathcal{P}\rightarrow \mathcal{B}$ is another quotient then there exists a rooted tree isomorphism $\gamma\colon\mathcal{Q}\rightarrow \mathcal{B}$ such that the following diagram is commutative:
\[
\xymatrix@R=5pt{
&\mathcal{Q}\ar[dd]^{\gamma}  \\
\mathcal{P}\ar[ur]^{\alpha}\ar[dr]_{\beta} &\\
&\mathcal{B}}
\]
\end{enumerate}
\end{theorem}
\noindent
In the following diagrams we indicate the first level compositions corresponding to the maps $\mathcal{T}\stackrel{\varphi}{\longrightarrow}\mathcal{P}\stackrel{\alpha}{\longrightarrow}\mathcal{Q}$.
\[
\xymatrix@R=10pt{
{\scriptstyle V_\mathcal{T}=\bigsqcup\limits_{x\in V_\mathcal{P}}\varphi^{-1}(x)=\bigsqcup\limits_{a\in V_\mathcal{Q}}\bigsqcup\limits_{x\in \alpha^{-1}(a)}\varphi^{-1}(x)}&{\scriptstyle V_\mathcal{P}=\bigsqcup\limits_{a\in V_\mathcal{Q}}\alpha^{-1}(a)}\ar[l]&\ar[l]{\scriptstyle V_\mathcal{Q}}\\
} 
\]
\begin{proof}
By induction on $H(\mathcal{T})$. For $H(\mathcal{T})=1$ it is trivial: condition \eqref{Prop1p51} forces $\lvert V_\mathcal{Q}\rvert=1$, $\alpha$ sends the whole $V_\mathcal{P}$ to the unique vertex in $V_\mathcal{Q}$ and the $\tau$'s and the $\omega$'s do not exist; cf. Example \ref{exampleh1}. For $H(\mathcal{T})\geq 2$ again \eqref{Prop1p51} forces $\lvert W_\mathcal{Q}\rvert\leq1$ and $\alpha$ sends the whole $W_\mathcal{P}$ to the unique vertex in $W_\mathcal{Q}$ (if $W_\mathcal{T}$ is not empty). Then we introduce the following equivalence relation: for $x,x'\in V_\mathcal{P}\setminus W_\mathcal{P}$ we set $x\sim x'$ if 
\[
u,v\in V_\mathcal{T},\;\varphi(u)=x \text{ and }\varphi(v)=x'\;\Longrightarrow\;\varphi_u\colon \mathcal{T}_u\rightarrow\mathcal{P}_x\text{ and }\varphi_v\colon \mathcal{T}_v\rightarrow\mathcal{P}_{x'}\text{ are equivalent; }
\]
by Remark \ref{Remobviousfacts}, it suffices to verify this condition for just one choice of $u\in\varphi^{-1}(x),v\in\varphi^{-1}(y)$, and $\sim$ is indeed an equivalence relation. 
Set $A\coloneqq \left(V_\mathcal{P}\setminus W_\mathcal{P}\right)/\sim$ and let $\pi\colon V_\mathcal{P}\setminus W_\mathcal{P}\rightarrow A$ be the projection map, so that 
\begin{equation}\label{2quad5}
V_\mathcal{P}\setminus W_\mathcal{P}=\bigsqcup_{a\in A}\pi^{-1}(a),\qquad V_\mathcal{T}\setminus W_\mathcal{T}=\bigsqcup_{a\in A}(\pi\circ \varphi)^{-1}(a)=\bigsqcup_{a\in A}\left(\bigsqcup_{x\in\pi^{-1}(a)}\varphi^{-1}(x)\right).
\end{equation}

Now we define the $\tau$'s, the $\omega$'s and also $\alpha$ at the first level.
First of all, for each $a\in A$ fix an element $x_a\in\pi^{-1}(a)$ and then for all $y\in\pi^{-1}(a)$, fix an element $v_y\in\varphi^{-1}(y)$.
If we consider the tree compositions $\varphi_{v_{x_a}}\colon\mathcal{T}_{v_{x_a}}\rightarrow\mathcal{P}_{\varphi(v_{x_a})}$, $a\in A$, then, by the inductive hypothesis, there exist unique (up to isomorphisms) tree compositions $\alpha_{\varphi(v_{x_a})}\colon\mathcal{P}_{\varphi(v_{x_a})}\rightarrow \mathcal{Q}_a$ that satisfy the properties in \eqref{Prop1p51}, \eqref{Prop1p52}, and \eqref{Prop1p53}. If $y\in \pi^{-1}(a)$ and $u\in\varphi^{-1}(y)$ then, by \eqref{diagTRSuv2}, $\varphi_u$ and $\varphi_{v_y}$ are strictly equivalent, so that we may fix an isomorphism $\tau_{v_y,u}\colon\mathcal{T}_u\rightarrow\mathcal{T}_{v_y}$ such that 
\begin{equation}\label{tauomfl1}
\varphi_u=\varphi_{v_y}\circ\tau_{v_y,u}
\end{equation}
and also set $\tau_{u,v_y}\coloneqq\tau_{v_y,u}^{-1}$ and $\tau_{v_y,v_y}\coloneqq\id_{\mathcal{T}_{v_y}}$. Moreover, if again $y\in\pi^{-1}(a)$, then $\varphi_{v_y}$ and $\varphi_{v_{x_a}}$ are equivalent (by the definition of $\sim$) so that there we may fix isomorphisms $\tau_{v_y,v_{x_a}}\colon\mathcal{T}_{v_{x_a}}\rightarrow\mathcal{T}_{v_y}$ and $\omega_{y,x_a}\colon\mathcal{P}_{x_a}\rightarrow\mathcal{P}_y$ such that (cf. \eqref{defequiviso})
\begin{equation}\label{tauomfl2}
\varphi_{v_y}\circ\tau_{v_y,v_{x_a}}=\omega_{y,x_a}\circ\varphi_{v_{x_a}}. 
\end{equation}
We also set $\tau_{v_{x_a},v_y}\coloneqq\tau_{v_y,v_{x_a}}^{-1}$, $\omega_{x_a,y}\coloneqq\omega_{y,x_a}^{-1}$ and $\omega_{x_a,x_a}\coloneqq\id_{\mathcal{P}_{x_a}}$.

If $u,w\in(\pi\circ\varphi)^{-1}(a)$, $y,z\in\alpha^{-1}(a)$, then using the $\tau$'s and the $\omega$'s in \eqref{tauomfl1} and \eqref{tauomfl2} we set
\begin{equation}\label{deftauomega0}
\tau_{w,u}\coloneqq\tau_{w,v_z}\circ\tau_{v_z,v_{x_a}}\circ\tau_{v_{x_a},v_y}\circ\tau_{v_y,u}\quad\text{ and }\quad\omega_{z,y}\coloneqq\omega_{z,x_a}\circ\omega_{x_a,y}. 
\end{equation}
See also the following diagram. 
\begin{equation}\label{bigdiagp12}
\xymatrix@R=15pt{
\mathcal{T}_u\ar@/_1.5pc/[dddd]_(.55){\tau_{w,u}}\ar[rrrrd]^{\varphi_u}\ar[d]^{\tau_{v_y,u}} &&&\\
\mathcal{T}_{v_y}\ar[d]^{\tau_{v_y,v_{x_a}}}\ar[rrrr]_{\varphi_{v_y}} &&&& \mathcal{P}_y\ar[d]_{\omega_{x_a,y}}\ar@/^1.5pc/[dd]^(.45){\omega_{z,y}}\\
\mathcal{T}_{v_{x_a}}\ar[rrrr]_{\varphi_{v_{x_a}}}\ar[d]^{\tau_{v_z,v_{x_a}}} &&&& \mathcal{P}_{x_a}\ar[d]_{\omega_{z,x_a}}\\
\mathcal{T}_{v_z}\ar[d]^{\tau_{w,v_z}}\ar[rrrr]^{\varphi_{v_z}} &&&& \mathcal{P}_z\\
\mathcal{T}_{w}\ar[rrrru]_{\varphi_w} &&&& 
}
\end{equation}

It is obvious that $\omega_{x,x}=\id_{\mathcal{P}_x}$. Also the composition properties \eqref{proptau2b} are easy to check while, from the particular cases \eqref{tauomfl1} and \eqref{tauomfl2}, it follows that
\begin{equation}\label{tauomfl}
\begin{split}
\varphi_w\circ\tau_{w,u}=&\left(\varphi_{v_z}\circ\tau_{v_z,w}\right)\circ\left(\tau_{w,v_z}\circ\tau_{v_z,v_{x_a}}\circ\tau_{v_{x_a},v_y}\circ\tau_{v_y,u}\right)\\
=&\omega_{z,x_a}\circ\varphi_{v_{x_a}}\circ\tau_{v_{x_a},v_y}\circ\tau_{v_y,u}\\
=&\omega_{z,x_a}\circ\omega_{x_a,y}\circ\varphi_{v_y}\circ\tau_{v_y,u}\\
=&\omega_{\varphi(w),\varphi(u)}\circ\varphi_u,
\end{split}
\end{equation}
that is the square in \eqref{diagTRSuv} is commutative and \eqref{tauomfl} proves the full commutativity of \eqref{bigdiagp12}. 

Now we can complete the definition of $\alpha$ and $\mathcal{Q}$: if $a\in V_\mathcal{Q}\setminus W_\mathcal{Q}$ then for all $v\in (\alpha\circ\varphi)^{-1}(a)$ we define a surjective tree homomorphism $\alpha_{\varphi(v)}\colon\mathcal{P}_{\varphi(v)}\rightarrow\mathcal{Q}_a$ by setting:
\begin{equation}\label{alphomfl}
\alpha_{\varphi(u)}\coloneqq\alpha_{\varphi(v_{x_a})}\circ \omega_{\varphi(v_{x_a}),\varphi(u)}.
\end{equation}
The tree $\mathcal{Q}$ is defined by means of its first level decomposition \eqref{firstld}:
\[
\mathcal{Q}\coloneqq\{\emptyset\}\sqcup \left(A\sqcup W_\mathcal{Q}\right)\sqcup \left[\bigsqcup_{a\in A}\left(\{a\}\times \widetilde{\mathcal{Q}}_a\right)\right]
\]
and the surjective rooted tree homomorphism $\alpha\colon\mathcal{P}\rightarrow\mathcal{Q}$ by means of \eqref{firstldell}, requiring that $\alpha$ sends the whole $W_\mathcal{P}$ to the unique vertex in $W_\mathcal{Q}$, $\alpha$ is equal to $\pi$ on $V_\mathcal{P}\setminus W_\mathcal{P}$, $\alpha$ is equal to $\alpha_x$ on $\mathcal{P}_x$, $x\in V_\mathcal{P}\setminus W_\mathcal{P}$. From \eqref{tauomfl} and \eqref{alphomfl} we deduce the commutative diagram \eqref{diagTRSuv}. 

At the first level, \eqref{Prop1p51} and \eqref{Prop1p52} are verified by the definition of $\sim$ and of $\alpha$ on $V_\mathcal{T}\setminus W_\mathcal{T}$ while \eqref{Prop1p53} requires also \eqref{diagTRSuv}.
From \eqref{tauomfl}, \eqref{alphomfl}, the inductive hypothesis applied to the $\mathcal{T}_{v_{x_a}}$'s and Remark \ref{Remobviousfacts}, it follows that \eqref{Prop1p51} and \eqref{Prop1p52} are verified also at the lower levels. We limit ourselves to proving \eqref{Prop1p53}. If $u,w$ are as in the statement and $u\in\mathcal{T}_{\widetilde{u}}$, $w\in\mathcal{T}_{\widetilde{w}}$ with $\widetilde{u},\widetilde{w}\in V_\mathcal{T}$, we set $u'\coloneqq\tau_{v_{x_a},\widetilde{u}}(u)$ and $w'\coloneqq\tau_{v_{x_a},\widetilde{w}}(w)$. Then
\[
\begin{split}
(\alpha\circ\varphi)(u')=&\left(\alpha_{\varphi(v_{x_a})}\circ \varphi_{v_{x_a}}\circ\tau_{v_{x_a},\widetilde{u}}\right)(u)\\
(\text{by }\eqref{tauomfl})\quad=&\left(\alpha_{\varphi(v_{x_a})}\circ\omega_{\varphi(v_{x_a}),\varphi(\widetilde{u})}\circ \varphi_{\widetilde{u}}\right)(u)\\
(\text{by }\eqref{alphomfl}\text{ and }\eqref{firstldell2})\quad=&\left(\alpha_{\varphi(\widetilde{u})}\circ \varphi_{\widetilde{u}}\right)(u)=(\alpha\circ\varphi)(u)=a\\
\end{split}
\]
and similarly $(\alpha\circ\varphi)(w')=a$. By the inductive hypothesis, $\varphi_{u'}$ and $\varphi_{w'}$ are equivalent while $\alpha_{\varphi(u')}\circ\varphi_{u'}$ and $\alpha_{\varphi(w')}\circ\varphi_{w'}$ are strictly equivalent (we are within $\mathcal{T}_{v_{x_a}}$). From \eqref{tauomfl} we deduce also that
\[
\varphi_u=\left[\varphi_{\widetilde{u}}\right]_u=\left[\omega_{\varphi(\widetilde{u}),\varphi(v_{x_a})}\right]_{\varphi(u')}\circ\varphi_{u'}\circ\left[\tau_{v_{x_a},\widetilde{u}}\right]_u,
\]
and therefore $\varphi_u$ and $\varphi_{u'}$ are equivalent. Similarly, $\varphi_w$ and $\varphi_{w'}$ are equivalent so that also $\varphi_u$ and $\varphi_w$ are equivalent. Finally, from \eqref{alphomfl} it follows that
\[
\alpha_{\varphi(u)}\circ\varphi_{u}=\left[\alpha_{\varphi(\widetilde{u})}\circ\omega_{\varphi(\widetilde{u}),\varphi(v_{x_a})}\right]_{\varphi(u')}\circ\varphi_{u'}\circ\left[\tau_{v_{x_a},\widetilde{u}}\right]_u=\alpha_{\varphi(u')}\circ\varphi_{u'}\circ\left[\tau_{v_{x_a},\widetilde{u}}\right]_u
\]
so that $\alpha_{\varphi(u)}\circ\varphi_u$ and $\alpha_{\varphi(u')}\circ\varphi_{u'}$ are strictly equivalent. Similarly, $\alpha_{\varphi(w)}\circ\varphi_w$ and $\alpha_{\varphi(u')}\circ\varphi_{w'}$ are strictly equivalent so that also $\alpha_{\varphi(u)}\circ\varphi_u$ and $\alpha_{\varphi(w)}\circ\varphi_w$ are strictly equivalent.

From the commutativity of \eqref{diagTRSuv} and from \eqref{Prop1p53} we deduce that both $\alpha$ and $\alpha\circ\varphi$ are tree compositions. The uniqueness of $\alpha$ up to isomorphisms follows from the uniqueness of $\pi$ and of each of the $\alpha_x\colon\mathcal{P}_x\rightarrow \mathcal{Q}_a$, $x\in\pi^{-1}(a)$, $a\in V_\mathcal{Q}$ (by induction). 
\end{proof}

\begin{definition}\label{Defpseudo}{\rm
Those in Theorem \ref{Prop1p5}\eqref{Prop1p53bis} will be called {\em a first level coordinate system of automorphism} $\tau$'s and $\omega$'s for the tree composition $\varphi$.}
\end{definition}

\begin{remark}\label{quotstequiv}{\rm
By means of the results in Remark \ref{Remobviousfacts} it is easy to check that if $\varphi\in\mathsf{C}(\mathcal{T},\mathcal{P})$, $\psi\in\mathsf{C}(\mathcal{T}',\mathcal{P})$ are strictly equivalent and $\alpha\in\mathsf{C}(\mathcal{P},\mathcal{Q})$ is the quotient of $\varphi$ then $\alpha$ is also the quotient of $\psi$. Similarly, suppose that $\varphi$ and $\psi$ are equivalent and $\alpha$, $\beta$ are their respective quotients, as in the following diagram, where $\tau$ and $\omega$ are isomorphisms, the left square is commutative and $\gamma$ is to be defined. 
\begin{equation}\label{Rtauomeg}
\xymatrix{
\mathcal{T}\ar[d]_\tau\ar[r]^{\varphi}&\mathcal{P}\ar[d]_\omega\ar[r]^{\alpha}&\mathcal{Q}\ar[d]_\gamma\\
\mathcal{T}'\ar[r]_{\psi}&\mathcal{P}'\ar[r]_{\beta}&\mathcal{Q}'
}
\end{equation}
Again by means of Remark \ref{Remobviousfacts}, it easy easy to check that also $\beta\circ\omega\colon\mathcal{P}\rightarrow\mathcal{Q}'$ is a quotient of $\varphi$.
Then from Theorem \ref{Prop1p5}\eqref{Prop1p54} it follows that there exists an isomorphism $\gamma$ such that the entire diagram is commutative. 
}
\end{remark}

\begin{example}\label{Examplequotient}{\rm
In Example \ref{Examplesubtrees} the quotient $\mathcal{Q}$ is obtained from $\mathcal{P}$ by deleting $a_h$, $\alpha\colon\mathcal{P}\rightarrow\mathcal{Q}$ sends $a_h$ to $b_h$ and it fixes all the other vertices. In Example \ref{Exampleflags} the quotient $\mathcal{Q}$ is equal to $\mathcal{P}$ up to level $h-1$; if $x\in\mathcal{P}$ is at level $h-1$ and $\ell(x)=j\geq 1$ then the $j+1$ children of $x$ are replaced by just one vertex which is the image of all the children of $x$.
}
\end{example}

\begin{example}\label{idtreecomp}{\rm
If $\mathcal{T}=\mathcal{P}$ and $\varphi=\id_\mathcal{T}$ is the identity composition then the relative quotient is called the {\em trivial composition}  $\tv_\mathcal{T}\colon\mathcal{T}\rightarrow \mathcal{C}$. It is characterized by the following properties:
\begin{itemize} 
\item
If $c,c'\in \mathcal{C}$ are distinct siblings which are internal, then for all pairs of siblings $u,v\in\mathcal{T}$ such that $\tv_\mathcal{T}(\varphi(u))=c$, $\tv_\mathcal{T}(\varphi(v))=c'$ the trees $\mathcal{T}_u$ and $\mathcal{T}_v$ are not isomorphic;
\item
if $v\in\mathcal{T}$ and $Y\subseteq\mathcal{T}_v$ is the subsets of all children of $v$ which are leaves then $\left(\tv_\mathcal{T}\right)_v$ sends the whole $Y$ to a singleton $c\in \mathcal{C}$, which is the unique child of $\tv_\mathcal{T}(v)$ which is a leaf;
\item
if $u,v\in\mathcal{T}$ are internal and $\tv_\mathcal{T}(u)=\tv_\mathcal{T}(v)$ then $\mathcal{T}_u$ and $\mathcal{T}_v$ are isomorphic. 
\end{itemize}
An equivalent description is the following: $\mathcal{C}_\mathcal{T}$ coincides with the quotient graph $\mathcal{T}/\Aut(\mathcal{T})$ (cf. \cite[Section 3.1]{Serre}); that is, the vertices of $\mathcal{C}_\mathcal{T}$ are the orbits of $\Aut(\mathcal{T})$ on $\mathcal{T}$, two orbits $c,c'\in\mathcal{C}_\mathcal{T}$ are joined if there exist $u\in c, u'\in c'$ such that $u,u'$ are joined in $\mathcal{T}$ and $\tv_\mathcal{T}(u)\coloneqq$ the $\Aut(\mathcal{T})$-orbit containing $u$, for every $u\in\mathcal{T}$.
In particular, $\mathcal{C}$ is a path if and only if $\mathcal{T}$ is spherically homogeneous.  Moreover, we have the following iterative description of $\Aut(\mathcal{T})$, which generalizes \eqref{Autiterwreath}:
\begin{equation}\label{Autp17}
\Aut(\mathcal{T})\cong\Sym(W_\mathcal{T})\prod\left\{\prod_{c\in V_\mathcal{C}\setminus W_\mathcal{C}}\left[\Aut\left(\mathcal{T}_{u_c}\right)\wr\Sym\left(\tv_\mathcal{T}^{-1}(c)\right)\right]\right\},
\end{equation}
where $\prod$ indicates a direct product of groups and $u_c$ is the choice of an element in $\tv_\mathcal{T}^{-1}(c)$. Finally, note that $\tv_\mathcal{T}\circ g=\tv_\mathcal{T}$ for all $g\in\Aut(\mathcal{T})$.
}
\end{example}

In what follow, $\Aut(\mathcal{T})\rvert_{V_\mathcal{T}}$ will denote the set of all restrictions $g\rvert_{V_\mathcal{T}}$, $g\in\Aut(\mathcal{T})$, so that from \eqref{Autp17} we deduce that
\begin{equation}\label{squaresigmap17}
\Aut(\mathcal{T})\rvert_{V_\mathcal{T}}\cong\Sym(W_\mathcal{T})\prod\left\{\prod_{c\in V_\mathcal{C}\setminus W_\mathcal{C}}\Sym\left(\tv_\mathcal{T}^{-1}(c)\right)\right\}.
\end{equation}

\begin{proposition}\label{Propepsphitv}
With the notation in Example \ref{idtreecomp}, for every $\varphi\in\mathsf{C}(\mathcal{T},\mathcal{P})$ there exists a surjective tree homomorphism $\varepsilon\colon\mathcal{P}\rightarrow\mathcal{C}$ (which in general is not a tree composition) such that $\varepsilon\circ\varphi=\tv_\mathcal{T}$.
\end{proposition}
\begin{proof}
By induction on $H(\mathcal{T})$. For $H(\mathcal{T})=1$ (and $W_\mathcal{T}$) see Example \ref{exampleh1}. If $H(\mathcal{T})>1$, for every $c\in V_\mathcal{C}\setminus W_\mathcal{C}$ we apply the inductive hypothesis to the trees $\mathcal{T}_u$, $u\in \tv_\mathcal{T}^{-1}(c)$, getting maps $\varepsilon_x\colon\mathcal{P}_x\rightarrow\mathcal{C}_c$, $x\in \varphi\left(\tv_\mathcal{T}^{-1}(c)\right)$ such that if $u\in \varphi^{-1}(x)$ then $\varepsilon_x\circ\varphi_u=\tv_{\mathcal{T}_u}$ (the subtrees $\mathcal{T}_u$, $u\in\varphi^{-1}(x)$ are all isomorphic and $\tv$ is invariant under right multiplication by automorphisms); we complete the definition of $\varepsilon$ by setting $\varepsilon(x)=c$.
\end{proof}
\begin{remark}{\rm
In general, the map $\varepsilon$ in Proposition \ref{Propepsphitv} is not a tree composition. We give a very simple example, the trees below with $\varphi(u_1)=x_1$, $\varphi(u_2)=x_2$, $\varphi(u_3)=\varphi(u_4)=x_3$, $\varepsilon(x_1)=\varepsilon(x_2)=\varepsilon(x_3)=c$.

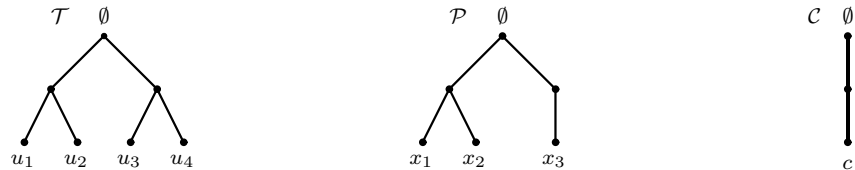
\begin{picture}(400,80)
\put(80,65){${\scriptstyle\mathcal{T}}$}
\put(100,60){\circle*{2}}
\put(98,65){${\scriptstyle\emptyset}$}
\thicklines
\drawline(100,60)(80,40)
\put(80,40){\circle*{2}}
\drawline(100,60)(120,40)
\put(120,40){\circle*{2}}
\drawline(80,40)(70,20)
\put(70,20){\circle*{2}}
\put(65,12){${\scriptstyle u_1}$}
\drawline(80,40)(90,20)
\put(90,20){\circle*{2}}
\put(85,12){${\scriptstyle u_2}$}

\drawline(120,40)(110,20)
\put(110,20){\circle*{2}}
\put(105,12){${\scriptstyle u_3}$}
\drawline(120,40)(130,20)
\put(130,20){\circle*{2}}
\put(125,12){${\scriptstyle u_4}$}

\put(230,65){${\scriptstyle\mathcal{P}}$}
\put(250,60){\circle*{2}}
\put(248,65){${\scriptstyle\emptyset}$}
\thicklines
\drawline(250,60)(230,40)
\put(230,40){\circle*{2}}
\drawline(250,60)(270,40)
\put(270,40){\circle*{2}}
\drawline(230,40)(220,20)
\put(220,20){\circle*{2}}
\put(215,12){${\scriptstyle x_1}$}
\drawline(230,40)(240,20)
\put(240,20){\circle*{2}}
\put(235,12){${\scriptstyle x_2}$}

\drawline(270,40)(270,20)
\put(270,20){\circle*{2}}
\put(265,12){${\scriptstyle x_3}$}

\put(365,65){${\scriptstyle\mathcal{C}}$}
\put(378,65){${\scriptstyle\emptyset}$}
\put(380,60){\circle*{2}}
\thicklines
\put(380,60){\line(0,-1){40}}
\put(380,40){\circle*{2}}
\put(380,20){\circle*{2}}
\put(378,10){${\scriptstyle c}$}
\end{picture}
}
\end{remark}

\section{Trees of partitions}\label{Sectreeofpart}

The {\em type} of a tree composition $\varphi\in\mathsf{C}(\mathcal{T},\mathcal{P})$ is the function $\lvert\varphi^{-1}\rvert\colon\mathcal{P}\rightarrow \mathbb{N}$ defined by setting $\lvert\varphi^{-1}\rvert(\emptyset)=0$ and, for all $y\in \mathcal{P}\setminus{\emptyset}$,
\begin{equation}\label{typepart}
\lvert\varphi^{-1}\rvert (y)\coloneqq\left\lvert \varphi^{-1}(y)\cap\text{Ch}(u)\right\rvert,
\end{equation}
where $u$ is any counter-image of the parent $x$ of $y$; by the regularity condition in Definition \ref{deftreecomp}, the cardinality of $\varphi^{-1}(y)\cap\text{Ch}(u)$ does not depend on the choice of $u$. If $u\in\mathcal{T}$ then the type of $\varphi_u\colon\mathcal{T}_u\rightarrow\mathcal{P}_{\varphi(u)}$ is just the restriction of $\left\lvert\varphi^{-1}\right\rvert$ to $\mathcal{P}_{\varphi(u)}$. For a statement similar to the next Theorem, see \cite[Lemma 3.5]{GNS}.

\begin{theorem}\label{lemma2p15}
Two tree compositions $\varphi\in\mathsf{C}(\mathcal{T},\mathcal{P})$ and $\psi\in\mathsf{C}(\mathcal{T}',\mathcal{P})$ are strictly equivalent if and  only if they are of the same type. 
\end{theorem}
\begin{proof}
The only if part is obvious. Now suppose that $\varphi$ and $\psi$ are of the same type. We want to show that there exists  a tree isomorphism $\tau\colon\mathcal{T}\rightarrow\mathcal{T}'$ such that $\varphi=\psi\circ \tau$,
\[\xymatrix@R=5pt{
\mathcal{T}\ar[dr]^{\varphi}\ar[dd]_{\tau} & \\
&\mathcal{P}\\
\mathcal{T}'\ar[ur]_{\psi}&}
\] 
By induction on $H(\mathcal{T})$, the number of levels of these trees. If $H(\mathcal{T})=1$ this is trivial (cf. Example \ref{exampleh1} and also the definition of $\tau\rvert_{W_\mathcal{T}}$ below). Suppose $H(\mathcal{T})>1$ and consider the first level decomposition \eqref{firstld}. There exists a bijection $\pi\colon V_\mathcal{T}\rightarrow V_{\mathcal{T}'}$ such that 
\begin{equation}\label{formp13bis}
\pi\left(\varphi^{-1}(x)\right)=\psi^{-1}(x), \quad\text{ for all }x\in V_\mathcal{P},
\end{equation}
because $\lvert\varphi^{-1}(x)\rvert=\lvert \psi^{-1}(x)\rvert$. In particular, $\pi$ is a bijection between $W_\mathcal{T}$ and $W_{\mathcal{T}'}$ (their images are $W_\mathcal{P}$) and we may set $\tau\rvert_{W_\mathcal{T}}\coloneqq\pi\rvert_{W_\mathcal{T}}$, so that \eqref{formp13bis} yields $\varphi\rvert_{W_\mathcal{T}}=\psi\rvert_{W_\mathcal{T}}\circ\tau\rvert_{W_\mathcal{T}}$.
If $u \in V_\mathcal{T}\setminus W_\mathcal{T}$ set $\varphi(u)\equiv \psi(\pi(u))\eqqcolon x$, so that $x\in V_\mathcal{P}\setminus W_\mathcal{P}$.
Then $\varphi_u\colon\mathcal{T}_u\rightarrow\mathcal{P}_x$ and $\psi_{\pi(u)}\colon\mathcal{T}'_{\pi(u)}\rightarrow\mathcal{P}_x$ have the same type as tree compositions of $\mathcal{T}_u$ and $\mathcal{T}'_{\pi(u)}$, respectively, so that, by the inductive hypothesis, they are strictly equivalent: there exists an isomorphism $\tau_u\colon\mathcal{T}_u\rightarrow\mathcal{T}_{\pi(u)}'$ such that 
\begin{equation}\label{formp13}
\varphi_u=\psi_{\pi(u)}\circ\tau_u.
\end{equation}
This way we get an isomorphism as in \eqref{firstldell} by setting $\tau(u,w)\coloneqq(\pi (u),\tau_u(w))$ so that, from \eqref{firstldell2} and \eqref{formp13} it follows that
$\left(\psi\circ\tau\right)(u,w)=\left(\psi(\pi(u))),\psi_{\pi(u)}(\tau_u(w))\right)
=\left(\varphi(u),\varphi_u(w)\right)=\varphi(u,w)$.
\end{proof}

Now suppose that $\alpha\in\mathsf{C}(\mathcal{P},\mathcal{Q})$ is the quotient of $\varphi\in\mathsf{C}(\mathcal{T},\mathcal{P})$, cf. Theorem \ref{Prop1p5}. If $a\in\mathcal{Q}$ is internal, $x\in\alpha^{-1}(a)$ and $u\in\varphi^{-1}(x)$, then 
\[
\text{\rm Ch}(u)=\bigsqcup_{b\in \text{\rm Ch}(a)}\left[(\alpha\circ \varphi)^{-1}(b)\cap\text{\rm Ch}(u)\right]
=\bigsqcup_{b\in \text{\rm Ch}(a)}\;\bigsqcup_{y\in\left[\text{\rm Ch}(x)\cap \alpha^{-1}(b)\right]}\left[\varphi^{-1}(y)\cap\text{\rm Ch}(u)\right].
\]
We introduce the following notation: we denote by $\lambda^b$ the {\em integer partition} which is the type of the set partition
$(\alpha\circ \varphi)^{-1}(b)\cap\text{\rm Ch}(u)=\left\{\varphi^{-1}(y)\cap\text{\rm Ch}(u)\colon y\in\text{\rm Ch}(x)\cap \alpha^{-1}(b)\right\}$;
cf. Example \ref{exampleh1}. By Theorem \ref{Prop1p5}, $\lambda^b$ does not depend on the choice of $u$. In the following picture the ellipse containing $\lambda^b$ represents the partition of type $\lambda^b$ of the vertices in ($\alpha\circ \varphi)^{-1}(b)\cap\text{\rm Ch}(u)$.\\

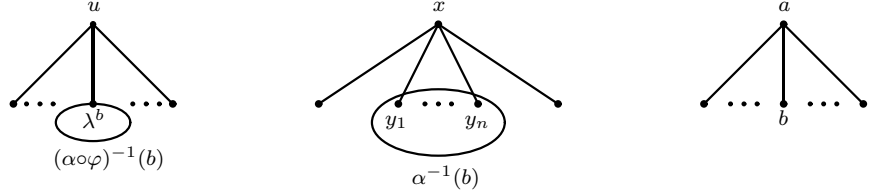
\begin{picture}(400,70)
\put(100,60){\circle*{2}}
\put(98,65){${\scriptstyle u}$}
\thicklines
\put(100,60){\line(-1,-1){30}}
\put(70,30){\circle*{2}}
\put(75,30){\circle*{1}}
\put(80,30){\circle*{1}}
\put(85,30){\circle*{1}}
\put(100,60){\line(0,-1){30}}
\put(100,30){\circle*{2}}
\put(100,23){\ellipse{28}{14}}
\put(96,21){${\scriptstyle \lambda^{b}}$}
\put(85,7){${\scriptstyle (\alpha\circ\varphi)^{-1}(b)}$}
\put(100,60){\line(1,-1){30}}
\put(130,30){\circle*{2}}
\put(125,30){\circle*{1}}
\put(120,30){\circle*{1}}
\put(115,30){\circle*{1}}

\put(230,60){\circle*{2}}
\put(228,65){${\scriptstyle x}$}
\thicklines
\put(230,60){\line(-3,-2){45}}
\put(185,30){\circle*{2}}
\put(230,60){\line(-1,-2){15}}
\put(215,30){\circle*{2}}
\put(210,22){${\scriptstyle y_1}$}
\put(225,30){\circle*{1}}
\put(230,30){\circle*{1}}
\put(235,30){\circle*{1}}
\put(230,60){\line(1,-2){15}}
\put(230,23){\ellipse{50}{25}}
\put(220,0){${\scriptstyle \alpha^{-1}(b)}$}
\put(245,30){\circle*{2}}
\put(240,22){${\scriptstyle y_n}$}
\put(230,60){\line(3,-2){45}}
\put(275,30){\circle*{2}}

\put(360,60){\circle*{2}}
\put(358,65){${\scriptstyle a}$}
\thicklines
\put(360,60){\line(-1,-1){30}}
\put(330,30){\circle*{2}}
\put(340,30){\circle*{1}}
\put(345,30){\circle*{1}}
\put(350,30){\circle*{1}}
\put(360,60){\line(0,-1){30}}
\put(360,30){\circle*{2}}
\put(370,30){\circle*{1}}
\put(375,30){\circle*{1}}
\put(380,30){\circle*{1}}
\put(360,60){\line(1,-1){30}}
\put(390,30){\circle*{2}}
\put(358,22){${\scriptstyle b}$}

\end{picture}

\noindent
In conclusion, for all $b\in\mathcal{Q}$ and $y\in\alpha^{-1}(b)$,
\begin{equation}\label{lamalphell}
\lambda^b_y\coloneqq\left\lvert \varphi^{-1}\right\rvert(y),\qquad\ell(\lambda^b)=\left\lvert\alpha^{-1}\right\rvert(b),\qquad \lvert\lambda^b\rvert=\left\lvert(\alpha\circ\varphi)^{-1}\right\rvert(b),
\end{equation}
(cf. \eqref{partnotat} and \eqref{typepart}) and the parts of $\lambda^b$ are the $\lambda^b_y$'s
listed in the weakly decreasing order.

Let $\mathcal{Q}$ be a rooted tree. A {\em partitional labeling} of $\mathcal{Q}$ is a function $\Lambda\colon a\mapsto \lambda^a$ that associates to each $a\in\mathcal{Q}$ {\em which is not the root} an integer partition $\lambda^a$ (we use $\Lambda$ to denote the function, and $\lambda^a$ to denote the value of $\Lambda$ in $a$). Another partitional labeling  $\Xi\colon b\mapsto \xi^b$ of tree $\mathcal{Q}'$ is {\em isomorphic} to $\Lambda\colon a\mapsto \lambda^a$ if there exists a rooted tree isomorphism $\gamma\colon\mathcal{Q}\rightarrow\mathcal{Q}'$ such that $\lambda^a=\xi^{\gamma(a)}$ for every $a\in\mathcal{Q}$.
  If $a\mapsto \lambda^a$ is constructed by means of \eqref{lamalphell}, we say that it is {\em determined by} (or {\em associated to}) the tree composition $\varphi\in\mathsf{C}(\mathcal{T},\mathcal{P})$.

\begin{theorem}\label{Theorem38}
Two tree compositions $\varphi\in\mathsf{C}(\mathcal{T},\mathcal{P})$ (with quotient $\alpha\in\mathsf{C}(\mathcal{P}, \mathcal{Q})$ and partitional labeling $a\mapsto\lambda^a$) and $\psi\in\mathsf{C}(\mathcal{T}',\mathcal{P}')$ (with quotient $\beta\in\mathsf{C}(\mathcal{P}',\mathcal{Q}')$ and partitional labeling $b\mapsto \xi^b$) are equivalent if and only if there exists an isomorphism $\gamma\colon\mathcal{Q}\rightarrow\mathcal{Q}'$ of the partitional labelings. Moreover, if this the case then the equivalence and the isomorphism form a commutative diagram as in \eqref{Rtauomeg}.
\end{theorem}

\begin{proof}
If $\varphi$ and $\psi$ are equivalent and $\gamma$ is the isomorphism in \eqref{Rtauomeg} then clearly $\lambda^a=\xi^{\gamma(a)}$.

Now we prove the converse, by induction on $H(\mathcal{T})$: for $H(\mathcal{T})=1$ it is obvious: by Example \ref{exampleh1} and Theorem \ref{Prop1p5}\eqref{Prop1p51}, $V_\mathcal{Q}$ and $V_{\mathcal{Q}'}$ are singletons and $\varphi$ and $\psi$ define equivalent set compositions, because they are associated to the same integer partition. Suppose that $H(\mathcal{T})> 1$.
First of all, we show how to construct $\omega$: more precisely, starting from an isomorphism $\gamma$ such that $\lambda^a=\xi^{\gamma(a)}$ we prove that there exists $\omega$ such that the second square in \eqref{Rtauomeg} is commutative and also $\lambda^a_x=\xi^{\gamma(a)}_{\omega(x)}$, for all $x\in\alpha^{-1}(a)$, $a\in \mathcal{Q}$.  By \eqref{lamalphell}, for all $a\in V_\mathcal{Q}$ we have 
$\left\lvert\alpha^{-1}(a)\right\rvert=\ell\left(\lambda^a\right)=\ell\left(\xi^{\gamma(a)}\right)=\left\lvert\beta^{-1}(\gamma(a))\right\rvert$. 
Therefore we can choose a bijection $\sigma\colon V_\mathcal{P}\rightarrow V_{\mathcal{P}'}$ in such a way that
\begin{equation}\label{gammaomeg}
\sigma\left(\alpha^{-1}(a)\right)=\beta^{-1}(\gamma(a)),
\end{equation}
for all $a\in V_\mathcal{Q}$, and also $\lambda^a_x=\xi^{\gamma(a)}_{\sigma(x)}$, for all $x\in\alpha^{-1}(a)$. Again, $W_\mathcal{Q}$ and $ W_{\mathcal{Q}'}$ are both empty or both singletons, and in the latter case may set $\omega\rvert_{W_\mathcal{P}}=\sigma\rvert_{W_\mathcal{P}}$.

After that, for each $x\in V_\mathcal{P}\setminus W_\mathcal{P}$ consider the tree compositions $\alpha_x\colon\mathcal{P}_x\rightarrow \mathcal{Q}_{\alpha(x)}$ and $\beta_{\sigma(x)}\colon\mathcal{P}'_{\sigma(x)}\rightarrow \mathcal{Q}_{\gamma(\alpha(x))}'\equiv\mathcal{Q}_{\beta(\sigma(x))}'$ (cf. \eqref{gammaomeg}). Since $\gamma_{\alpha(x)}\colon\mathcal{Q}_{\alpha(x)}\rightarrow \mathcal{Q}_{\gamma(\alpha(x))}'$ satisfies again the condition $\lambda^a=\xi^{\gamma(a)}$ for all $a\in\mathcal{Q}_{\alpha(x)}$, by the inductive hypothesis there exists an isomorphism $\omega_x\colon\mathcal{P}_x\rightarrow\mathcal{P}'_{\sigma(x)}$ such that 
\begin{equation}\label{betasig}
\beta_{\sigma(x)}\circ\omega_x=\gamma_{\alpha(x)}\circ\alpha_x\qquad\text{ and }\qquad\lambda^a_z=\xi^{\gamma(a)}_{\omega_x(z)}, 
\end{equation}
for all $z\in\alpha^{-1}(a)\cap \mathcal{P}_x$, $a\in\mathcal{Q}_{\alpha(x)}$.
Then we can define an isomorphism written as in \eqref{firstldell} by setting $\omega(x,z)=(\sigma (x),\omega_x(z))$ for all $x\in V_\mathcal{P}\setminus W_\mathcal{P}$, $z\in\widetilde{\mathcal{P}}_x$, so that
\begin{align*}
\beta\circ\omega(x,z)=&\bigl((\beta\circ\sigma)(x),(\beta_{\sigma(x)}\circ\omega_x)(z))\bigr)&(\text{by }\eqref{firstldell2})\\
=&\bigl((\gamma\circ\alpha)(x),(\gamma_{\alpha(x)}\circ\alpha_x)(z)\bigr)=\gamma\circ\alpha(x,z).&(\text{by }\eqref{gammaomeg}, \eqref{betasig} \text{ and }\eqref{firstldell2})\\
\end{align*}
Finally, consider the tree compositions $\omega\circ\varphi\colon\mathcal{T}\rightarrow\mathcal{P}'$ and $\psi\colon\mathcal{T}'\rightarrow\mathcal{P}'$ (cf. \eqref{Rtauomeg}); they have the same type because $x\in\alpha^{-1}(a)$, $a\in \mathcal{Q}$, implies that 
\begin{equation}\label{betasig2}
\omega(x)\in\beta^{-1}(\gamma(a)).
\end{equation}
Therefore, by \eqref{betasig} and \eqref{betasig2},
$\left\lvert\left(\omega\circ\varphi\right)^{-1}\right\rvert(\omega(x))=\left\lvert\varphi^{-1}\right\rvert(x)=\lambda_x^a
=\xi^{\gamma(a)}_{\omega(x)}=\left\lvert\psi^{-1}\right\rvert(\omega(x))$. Then the existence of $\tau$ is guaranteed by Theorem \ref{lemma2p15}.
\end{proof}

\begin{definition}\label{Deftreepart}{\rm 
A {\em tree of partitions} is a partitional labeling of a rooted tree $\mathcal{Q}$ such that: (i) among the children of an internal vertex of $\mathcal{Q}$ at most one is a leaf; (ii) if $a,b\in\mathcal{Q}$ are distinct siblings which are internal then the partitional labelings $\Lambda_a$ and $\Lambda_b$ are {\em not isomorphic}. 
}
\end{definition}

Note that, in the previous definition, in the partitional labelings $\Lambda_a$ and $\Lambda_b$ the vertices $a$ and $b$ are the roots of the respective trees and therefore they must considered {\em unlabeled}. Note also that, by (ii) in Definition \ref{Deftreepart}, a tree of partitions has no non-trivial automorphisms. We will usually indicate a tree of partitions as a couple $(\Lambda,\mathcal{Q})$; sometimes only by $\Lambda$.

\begin{proposition}\label{ProptreepartT}
A partitional labeling $\Lambda\colon a\mapsto\lambda^a$ of a rooted tree $\mathcal{Q}$ is determined by a tree composition if and only if it is a tree of partitions.
\end{proposition}
\begin{proof}
The only if part follows from \eqref{Prop1p51} and \eqref{Prop1p52} in Theorem \ref{Prop1p5} and from Theorem \ref{Theorem38}. The if part may be obtained by reversing the proof of Theorem \ref{Prop1p5}: by induction on $H(\mathcal{Q})$, we may construct $\mathcal{T}, \mathcal{P}, \varphi,\alpha$ from $\mathcal{Q}$ and $\Lambda$. Now \eqref{2quad5} may be used to {\em define} $V_\mathcal{P}\setminus W_\mathcal{P}$, $\pi\equiv\alpha_{V_\mathcal{P}\setminus W_\mathcal{P}}$, $V_\mathcal{T}\setminus W_\mathcal{T}$ and $\varphi_{V_\mathcal{T}\setminus W_\mathcal{T}}$, taking into account that the cardinalities are given by the partitions associated to the vertices in $V_\mathcal{Q}\setminus W_\mathcal{Q}$, by means of \eqref{lamalphell}. The inductive hypothesis may be used to construct the trees $\mathcal{T}_u$, $u\in V_\mathcal{T}\setminus W_\mathcal{T}$, $\mathcal{P}_x$, $x\in V_\mathcal{P}\setminus W_\mathcal{P}$ and the relative tree compositions. Finally, observe that Theorem \ref{Theorem38} and Definition \ref{Deftreepart} ensure that the condition in Theorem \ref{Prop1p5}\eqref{Prop1p52} is verified. 
\end{proof}

\begin{corollary}\label{CortreepartT}
\begin{enumerate}[(a)]
\item\label{CortreepartTa}
If $(\Lambda,\mathcal{Q})$ is a tree of partitions then there exists a rooted tree $\mathcal{T}$ with a tree composition $\delta\in\mathsf{C}(\mathcal{T},\mathcal{Q})$ such that, for every internal vertex $a\in\mathcal{Q}$ and for every $u\in\delta^{-1}(a)$, we have
\begin{equation}\label{deltaTQ}
\sum_{b\in\text{\rm Ch}(a)}\left\lvert\lambda^b\right\rvert=\lvert\text{\rm Ch}(u)\rvert.
\end{equation}
\item\label{CortreepartTb}
If $\mathcal{T}'$ is another rooted tree with a tree composition $\delta'\in\mathsf{C}(\mathcal{T}',\mathcal{Q})$ that satisfies \eqref{deltaTQ} then there exists an isomorphism $\tau\colon\mathcal{T}\rightarrow\mathcal{T}'$ such that the diagram
\begin{equation}\label{Rtauomeg2}
\xymatrix@R=5pt{
\mathcal{T}\ar[dr]^{\delta}\ar[dd]_{\tau} & \\
&\mathcal{Q}\\
\mathcal{T}'\ar[ur]_{\delta'}&}
\end{equation}
is commutative.
\item\label{CortreepartTc}
Given $(\Lambda,\mathcal{Q})$ and $\mathcal{T}$ as in \eqref{CortreepartTa}, a rooted tree $\mathcal{P}$ has $(\Lambda,\mathcal{Q})$ as the tree of partitions associated to some $\varphi\in\mathsf{C}(\mathcal{T},\mathcal{P})$ if and only if there exists a tree composition $\alpha\in\mathsf{C}(\mathcal{P},\mathcal{Q})$ such that, for every internal vertex $a\in\mathcal{Q}$ and for every $x\in\alpha^{-1}(a)$, we have
\begin{equation}\label{alphaPQ}
\sum_{b\in\text{\rm Ch}(a)}\ell(\lambda^b)=\lvert\text{\rm Ch}(x)\rvert.
\end{equation}
\end{enumerate}
\end{corollary}
\begin{proof}
The existence of $\mathcal{T}$ and of $\delta$ in \eqref{CortreepartTa} follows from Proposition \ref{ProptreepartT}, by setting $\delta\coloneqq\alpha\circ\varphi$ in the notation of Theorem \ref{Prop1p5}; then \eqref{deltaTQ} follows from \eqref{lamalphell}. After that, \eqref{CortreepartTb} follows from Theorem \ref{Theorem38}, recalling that now $\gamma=\id_\mathcal{Q}$ because a tree of partitions has no non-trivial automorphisms. Now assume \eqref{alphaPQ} and let $\mathcal{T}$, $\delta$ be as in \eqref{CortreepartTa}. Then $\varphi$ is any map such that $\varphi\left(\delta^{-1}(b)\right)=\alpha^{-1}(b)$ and $\varphi\rvert_{\delta^{-1}(b)\cap\text{\rm Ch}(u)}$ determines a set composition of type $\lambda^b$, for every $u\in\mathcal{T}$ such that $\delta(u)$ is the parent of $b$. Such a $\varphi$ may be easily constructed inductively, taking into account that if $\delta(u)=\delta(v)$ then $\mathcal{T}_u$ and $\mathcal{T}_v$ are isomorphic.
 \end{proof}

\begin{example}{\rm
The tree of partitions of $\varphi^\mathcal{S}$ in Example \ref{Examplesubtrees} (see also Example \ref{Examplequotient}) is the following (we assume $r_h-s_h\geq s_h$, otherwise the last partition at the lower level is $(s_h,r_h-s_h)$):  

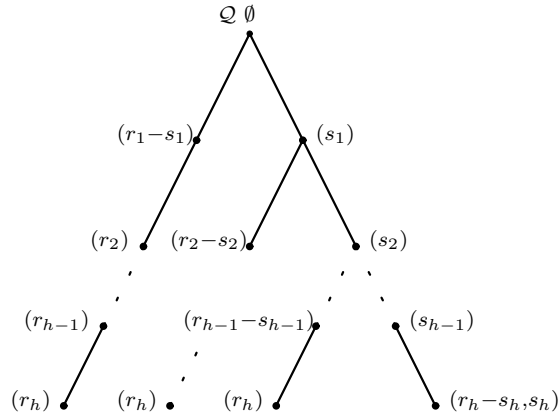
\begin{picture}(400,150)


\put(200,140){\circle*{2}}
\put(198,145){${\scriptstyle\emptyset}$}
\put(188,145){${\scriptstyle\mathcal{Q}}$}
\thicklines
\put(200,140){\line(1,-2){40}}

\put(200,140){\line(-1,-2){40}}

\put(180,100){\circle*{2}}
\put(150,100){${\scriptstyle (r_1-s_1)}$}
\put(220,100){\circle*{2}}
\put(220,100){\line(-1,-2){20}}
\put(225,100){${\scriptstyle (s_1)}$}

\put(160,60){\circle*{2}}
\put(140,60){${\scriptstyle (r_2)}$}
\put(200,60){\circle*{2}}
\put(170,60){${\scriptstyle (r_2-s_2)}$}
\put(240,60){\circle*{2}}
\put(245,60){${\scriptstyle (s_2)}$}
\dashline{2}(240,60)(225,30)
\dashline{2}(160,60)(145,30)
\dashline{2}(240,60)(255,30)
\put(145,30){\circle*{2}}
\put(145,30){\line(-1,-2){15}}
\put(115,30){${\scriptstyle (r_{h-1})}$}
\put(225,30){\circle*{2}}
\put(225,30){\line(-1,-2){15}}
\put(175,30){${\scriptstyle (r_{h-1}-s_{h-1})}$}
\put(255,30){\circle*{2}}

\put(255,30){\line(1,-2){15}}
\put(260,30){${\scriptstyle (s_{h-1})}$}

\put(130,0){\circle*{2}}
\put(110,0){${\scriptstyle (r_h)}$}
\dashline{2}(180,20)(170,0)
\put(170,0){\circle*{2}}
\put(150,0){${\scriptstyle (r_h)}$}
\put(210,0){\circle*{2}}
\put(190,0){${\scriptstyle (r_h)}$}
\put(270,0){\circle*{2}}
\put(275,0){${\scriptstyle (r_h-s_h,s_h)}$}
\end{picture}

\noindent
The tree of partitions of a tree composition as in Example \ref{Exampleflags} has the following structure ($\mathcal{Q}$ is described in Example \ref{Examplequotient}). The tree $\mathcal{Q}$ coincides with $\mathcal{P}$ up to level $h-1$; if $x\in\mathcal{Q}$ is at level $k\leq h-1$ and $\ell(x)=j$ then its associated partition is $\lambda^x=(s_k^j-s_k^{j+1})$ (recall that $s_k^{l+1}=0$). If $a\in\mathcal{Q}$ is at level $h$, $x$ is its parent at level $h-1$ and $\ell(x)=j\geq1$ then $\lambda^a=(s_h^0-s_h^1,s_h^1-s_h^2,\dotsc,s_h^{j-1}-s_h^j, s_h^j)$ while if $\ell(x)=0$ then $\lambda^a=(s_h^0)\equiv(r_h)$.
}
\end{example}

\begin{example}\label{treeparttvid}{\rm
It is easy to check that, if $\mathcal{T}$ is spherically homogeneous of branching type $(r_1,r_2,\dotsc,r_h)$, then the trees of partitions of the trivial composition and of the identity composition (cf. Example \ref{idtreecomp}) are, respectively,

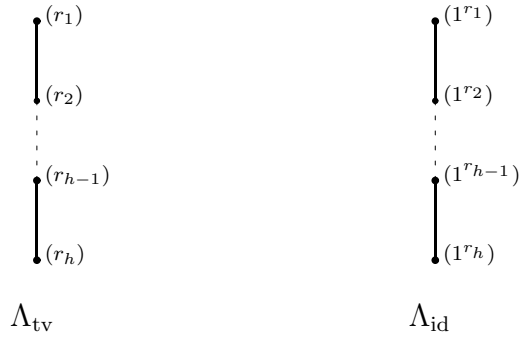
\begin{picture}(400,150)
\thicklines
\put(150,140){\circle*{2}}
\put(153,140){${\scriptstyle (r_1)}$}
\put(150,140){\line(0,-1){30}}
\thinlines
\put(150,110){\circle*{2}}
\put(153,110){${\scriptstyle (r_2)}$}
\dashline{2}(150,110)(150,80)
\thicklines
\put(150,80){\circle*{2}}
\put(153,80){${\scriptstyle (r_{h-1})}$}
\put(150,80){\line(0,-1){30}}
\put(150,50){\circle*{2}}
\put(153,50){${\scriptstyle (r_h)}$}
\put(140,25){$\Lambda_{\text{\rm tv}}$}
\thicklines
\put(300,140){\circle*{2}}
\put(303,140){${\scriptstyle \left(1^{r_1}\right)}$}
\put(300,140){\line(0,-1){30}}
\thinlines
\put(300,110){\circle*{2}}
\put(303,110){${\scriptstyle \left(1^{r_2}\right)}$}
\dashline{2}(300,110)(300,80)
\thicklines
\put(300,80){\circle*{2}}
\put(303,80){${\scriptstyle \left(1^{r_{h-1}}\right)}$}
\put(300,80){\line(0,-1){30}}
\put(300,50){\circle*{2}}
\put(303,50){${\scriptstyle \left(1^{r_h}\right)}$}
\put(290,25){$\Lambda_{\text{\rm id}}$}
\end{picture}
}
\end{example}

\begin{definition}\label{defIrrG}{\rm
If $\mathcal{T}$ is a rooted tree then we denote by $\mathsf{Par}(\mathcal{T})$ the set of all trees of partitions $(\Lambda,\mathcal{Q})$ for which there exists $\delta\in\mathsf{C}(\mathcal{T},\mathcal{Q})$
such that \eqref{deltaTQ} is verified; this $\delta$ will be called a {\em realization} of $(\Lambda,\mathcal{Q})$ as a tree of partitions of $\mathcal{T}$. If $\mathcal{P}$ is a compositional tree for $\mathcal{T}$ we denote by $\mathsf{Par}(\mathcal{T},\mathcal{P})$ the set of all $(\Lambda,\mathcal{Q})\in\mathsf{Par}(\mathcal{T})$ for which there exists $\alpha\in\mathsf{C}(\mathcal{P},\mathcal{Q})$ such that also \eqref{alphaPQ} is verified.
Finally, we set $\mathsf{Type}(\mathcal{T},\mathcal{P})\coloneqq\left\{\lvert\varphi^{-1}\rvert\colon\varphi\in\mathsf{C}(\mathcal{T},\mathcal{P})\right\}$, the set of all possible types of the $\mathcal{P}$-compositions of $\mathcal{T}$. }
\end{definition}

In other words, $\mathsf{Par}(\mathcal{T})$ is the set of all trees of partitions that can be obtained from the tree compositions of $\mathcal{T}$, $\mathsf{Par}(\mathcal{T},\mathcal{P})$ is the set of all trees of partitions that can be obtained from the $\mathcal{P}$-compositions of $\mathcal{T}$, while $\mathsf{Type}(\mathcal{T},\mathcal{P})$ is the set of all the possible types of the $\mathcal{P}$-compositions of $\mathcal{T}$. In the $H(\mathcal{T})=1$ case (cf. Example \ref{exampleh1}) $\mathsf{Par}(\mathcal{T})$ corresponds to the set of all integer partitions of $\lvert V\rvert$, while both $\mathsf{Par}(\mathcal{T},\mathcal{P})$ and $\mathsf{Type}(\mathcal{T},\mathcal{P})$ correspond to the set of all integer partitions of $\lvert V\rvert$ of length $\lvert X\rvert$. 

We denote by $\mathsf{C}(\mathcal{T},\mathcal{P})/\left(\Aut(\mathcal{T})\times\Aut(\mathcal{P})\right)$ the set of $\Aut(\mathcal{T})\times\Aut(\mathcal{P})$-orbits on $\mathsf{C}(\mathcal{T},\mathcal{P})$, which correspond to the equivalence classes of $\mathcal{P}$-compositions of $\mathcal{T}$ (cf. Example \ref{Remobviousfacts2}); similarly, we denote by $\mathsf{C}(\mathcal{T},\mathcal{P})/\Aut(\mathcal{T})$ the set of $\Aut(\mathcal{T})$-orbits on $\mathsf{C}(\mathcal{T},\mathcal{P})$, which correspond to the strict equivalence classes of $\mathcal{P}$-compositions of $\mathcal{T}$. For $\varphi\in\mathsf{C}(\mathcal{T},\mathcal{P})$ we denote by $[\varphi]$ (respectively by $\{\varphi\}$) the equivalence class (respectively  the strict equivalence class) containing $\varphi$.

\begin{theorem}\label{Theoorbits}
\begin{enumerate}[(a)]
\item\label{Theoorbitsa}
The map
\[
\begin{array}{ccc}
\mathsf{C}(\mathcal{T},\mathcal{P})/\left(\Aut(\mathcal{T})\times\Aut(\mathcal{P})\right)&\longrightarrow&\mathsf{Par}(\mathcal{T},\mathcal{P})\\ 
\;[\varphi]&\longmapsto&(\Lambda,\mathcal{Q}),
\end{array}
\]
where $(\Lambda,\mathcal{Q})$ is the tree of partitions associated to $\varphi$, is a bijection. In particular, $\mathsf{Par}(\mathcal{T},\mathcal{P})$ parametrizes the orbits of $\Aut(\mathcal{T})\times\Aut(\mathcal{P})$ on $\mathsf{C}(\mathcal{T},\mathcal{P})$ and $\mathsf{Par}(\mathcal{T})$ parametrizes the set of all equivalence classes of tree compositions of $\mathcal{T}$.
\item\label{Theoorbitsb}
The map
\[
\begin{array}{ccc}
\mathsf{C}(\mathcal{T},\mathcal{P})/\Aut(\mathcal{T})&\longrightarrow&\mathsf{Type}(\mathcal{T},\mathcal{P})\\ 
\;\{\varphi\}&\longmapsto&\lvert\varphi^{-1}\rvert
\end{array}
\]
is a bijection. In particular, $\mathsf{Type}(\mathcal{T},\mathcal{P})$ parametrizes the orbits of $\Aut(\mathcal{T})$ on $\mathsf{C}(\mathcal{T},\mathcal{P})$.
\end{enumerate}
\end{theorem}
\begin{proof}
By Theorem \ref{Theorem38}, Proposition \ref{ProptreepartT} and Corollary \ref{CortreepartT}, $\mathsf{Par}(\mathcal{T},\mathcal{P})$ parametrizes the equivalence classes of $\mathcal{P}$-compositions of $\mathcal{T}$. Since $(\Lambda,\mathcal{Q})$ determines $\varphi$ up to equivalences, by varying $\mathcal{P}$ we get all the elements of $\mathsf{Par}(\mathcal{T})$, that is $\mathsf{Par}(\mathcal{T})=\bigsqcup_\mathcal{P}\mathsf{Par}(\mathcal{T},\mathcal{P})$, where the (disjoint) union is over all compositional trees $\mathcal{P}$ of $\mathcal{T}$. Finally, \eqref{Theoorbitsb} follows from Theorem \ref{lemma2p15}.
\end{proof}

By Corollary \ref{CortreepartT}\eqref{CortreepartTb}, the set of all realizations of $(\Lambda,\mathcal{Q})$ form an orbit of $\Aut(\mathcal{T})$ of the form $\{\delta\circ g\colon g\in\Aut(\mathcal{T})\}$. In the notation of Theorem \ref{Prop1p5},
 $\delta\coloneqq \alpha\circ\varphi$ is a realization of the tree of partitions constructed by means of \eqref{lamalphell} and if $\varphi'\in[\varphi]$ then the realization of $(\Lambda,\mathcal{Q})$ associated to $\varphi'$ is in the $\Aut(\mathcal{T})$-orbit of $\delta$.

\begin{algorithm}\label{alg1}{\rm
Now we describe an iterative algorithm for the construction of every element of $\mathsf{Par}(\mathcal{T})$ along with one of its realizations. We will use the notation in Example \ref{idtreecomp}.
\begin{enumerate}
\item
If $H(\mathcal{T})=1$ or if $W_\mathcal{T}$ is not empty use Example \ref{exampleh1}.
\item
If $H(\mathcal{T})>1$ then for each $c\in V_\mathcal{C}\setminus W_\mathcal{C}$ take a set $A_c$ such that 
\begin{equation}\label{defAc}
\lvert A_c\rvert\leq \min\left\{\left\lvert \tv_\mathcal{T}^{-1}(c)\right\rvert,\left\lvert\mathsf{Par}(\mathcal{T}_{u_c})\right\rvert \right\}
\end{equation}
and set $V_\mathcal{Q}\setminus W_\mathcal{Q}\coloneqq \bigsqcup_{c\in V_\mathcal{C}\setminus W_\mathcal{C}}A_c$.
\item\label{alg3}
Then for all $a\in A_c$, $c\in V_\mathcal{C}\setminus W_\mathcal{C}$, choose $(\Lambda_a,\mathcal{Q}_a)\in\mathsf{Par}(\mathcal{T}_{u_a})$ in such a way that if $a,a'\in A_c$ and $a\neq a'$ then $(\Lambda_a,\mathcal{Q}_a)$ and $(\Lambda_{a'},\mathcal{Q}_{a'})$ are not isomorphic; this requires the preventive construction of the elements of $\mathsf{Par}(\mathcal{T}_{u_a})$ and the condition $\lvert A_c\rvert\leq\lvert\mathsf{Par}(\mathcal{T}_{u_c})\rvert$ in \eqref{defAc}. Fix also a $\delta_a\in\mathsf{C}(\mathcal{T}_{u_c},\mathcal{Q}_a)$ which is a realization of $\Lambda_a$.

\item\label{alg4}
Choose integer partitions $\lambda^a$ in such a way that $\sum_{a\in A_c}\lvert\lambda^a\rvert=\left\lvert \tv_\mathcal{T}^{-1}(c)\right\rvert$, for all $c\in V_\mathcal{C}\setminus W_\mathcal{C}$.
This requires the condition $\lvert A_c\rvert\leq \lvert \tv_\mathcal{T}^{-1}(c)\rvert$ in \eqref{defAc}. Then we may define $\mathcal{Q}$ by means of \eqref{firstld} and also $\Lambda$ is completely defined and we have the tree of partitions.

\item
Finally, define $\delta$ at the first level in such a way that $\varphi\bigl(\tv_\mathcal{T}^{-1}(c)\bigr)=A_c$ and \eqref{deltaTQ} is satisfied.

\end{enumerate}

\noindent
By choosing the $(\Lambda_a,\mathcal{Q}_a)$'s in \ref{alg3}. and the $\lambda^a$'s in \ref{alg4}. in all possible way we get all the elements in $\mathsf{P}(\mathcal{T})$.
}
\end{algorithm}

\begin{example}\label{Exbinary}{\rm
Suppose that $\mathcal{T}_h$ is the spherically homogeneous binary tree of height $h$, that is set $r_1=r_2=\dotsb r_h=2$ in \eqref{Autiterwreath}.
Now we apply Algorithm \ref{alg1} to get a recursive formula for $\left\lvert\mathsf{Par}\left(\mathcal{T}_h\right)\right\rvert$. First of all, $\left\lvert\mathsf{Par}\left(\mathcal{T}_1\right)\right\rvert=2$ and the tree $\mathcal{C}$ of the trivial composition is a path (cf. Example \ref{treeparttvid}). When we apply the algorithm we have three possibilities:
\begin{enumerate}[(a)]
\item\label{Exbinary1} take $A_c=\{a\}$, $\lambda^a=(2)$ and choose a tree of partitions in $\mathsf{Par}\left(\mathcal{T}_{h-1}\right)$;
\item\label{Exbinary2} take $A_c=\{a\}$, $\lambda^a=(1,1)$ and choose a tree of partitions in $\mathsf{Par}\left(\mathcal{T}_{h-1}\right)$;
\item\label{Exbinary3} take $A_c=\{a,a'\}$, so that necessarily $\lambda^a=(1)=\lambda^{a'}$ and choose a couple of distinct trees of partitions in $\mathsf{Par}\left(\mathcal{T}_{h-1}\right)$.
\end{enumerate}
Then we have
\[
\left\lvert\mathsf{Par}\left(\mathcal{T}_h\right)\right\rvert=2\left\lvert\mathsf{Par}\left(\mathcal{T}_{h-1}\right)\right\rvert+\binom{\left\lvert\mathsf{Par}\left(\mathcal{T}_{h-1}\right)\right\rvert}{2},
\]
where the first term in the right hand side comes from \eqref{Exbinary1} and \eqref{Exbinary2} while the second term from \eqref{Exbinary3}. The same formula is in \cite[p. 539]{OOR}, where it is obtained by iterating Clifford theory.
}
\end{example}

\section{The conjugacy classes of $\Aut(\mathcal{T})$}\label{Secconjclasses}

This section is partially inspired to the classical theory of conjugacy classes of wreath products; see \cite[Section 2.3]{book3} and \cite[Section 4.3]{JK}. But now we use the new tool represented by the theory of tree compositions; in particular, the $\varphi^g$ in Example \ref{Examplephig}.

\begin{remark}{\rm The paper \cite{GNS} contains very similar results for the automorphisms group of an infinite, rooted tree. The orbit type in that paper is obtained by assigning to each vertex $x\in \mathcal{P}^g$ the cardinality of $\left(\varphi^g\right)^{-1}(x)$. This notion is  equivalent to the type defined in \eqref{typepart}: if $x_0=\emptyset,x_1,\dotsc,x_k=x$ is the path from the root to $x$ then $\left\lvert\left(\varphi^g\right)(x)\right\rvert=\prod_{i=1}^k\left\lvert\left(\varphi^g\right)^{-1}\right\rvert(x_i)$; see also \cite[Condition (i) p. 541]{GNS}. We have taken a further step by introducing the notion of tree of partitions, which is a more concise notation, also more natural from the point of view of the representation theory. For instance, in our notation the identity element is represented by the tree $\Lambda_{\text{\rm id}}$ in Example \ref{treeparttvid} while in \cite{GNS} it is represented by the whole tree $\mathcal{T}$ with each vertex labeled by $1$.
}
\end{remark}

Let $g\in\Aut(\mathcal{T})$ and denote by $\pi$ its restriction to $V_\mathcal{T}$. Suppose $\{u,\pi(u),\dotsc,\pi^{\ell-1}(u)\}$ is an orbit of $\pi$ on $V_\mathcal{T}\setminus W_\mathcal{T}$, with $\ell$ the smallest positive integer such that $\pi^\ell(u)=u$. Set 
\begin{equation}\label{defXi}
\left[g^\ell\right]_u\coloneqq g_{\pi^{\ell-1}(u)}\circ g_{\pi^{\ell-2}(u)}\circ\dotsb\circ g_{\pi(u)}\circ g_u.
\end{equation}
By \eqref{firstldell2}, $\left[g^\ell\right]_u\in\Aut(\mathcal{T}_u)$; note that \eqref{defXi} is just a particular notation for $g^\ell$ restricted to $\mathcal{T}_u$, but when we use it we assume that $u$ is at the first level and $\ell$ is the length of its $g$-orbit.
Now suppose that also $t\in\Aut(\mathcal{T})$ and denote by $\sigma$ the restriction of $t$ to $V_\mathcal{T}$. Then $\sigma\circ\pi\circ\sigma^{-1}$ is the restriction of $tgt^{-1}$ to $V_{\mathcal{T}}$ while from \eqref{firstldell2} and \eqref{firstldell3} we deduce that for $u\in V_\mathcal{T}\setminus W_\mathcal{T}$ we have:
\begin{equation}\label{congtgt1}
(tgt^{-1})_u=t_{\pi(\sigma^{-1}(u))}\circ g_{\sigma^{-1}(u)}\circ\left(t_{\sigma^{-1}(u)}\right)^{-1}.
\end{equation}

\begin{lemma}\label{lemmacongtgt2}
With the notation above, for all $v\in V_\mathcal{T}\setminus W_{\mathcal{T}}$ we have:
\[
\left[\left(tgt^{-1}\right)^\ell\right]_{\sigma(v)}=t_v\circ\left[g^\ell\right]_v\circ\left(t_v\right)^{-1}.
\]
\end{lemma}
\begin{proof}
From \eqref{defXi} and \eqref{congtgt1} it follows that
\[
\begin{split}
\left[\left(tgt^{-1}\right)^\ell\right]_u=&\left[\left(tgt^{-1}\right)^\ell\right]_{\sigma\pi^{\ell-1}\sigma^{-1}(u)}\circ\dotsb\circ\left[\left(tgt^{-1}\right)^\ell\right]_{\sigma\pi\sigma^{-1}(u)}\circ\left[\left(tgt^{-1}\right)^\ell\right]_u\\
=&\left[t_{\pi^\ell(\sigma^{-1}(u))}\circ g_{\pi^{\ell-1}(\sigma^{-1}(u))}\circ\left(t_{\pi^{\ell-1}(\sigma^{-1}(u))}\right)^{-1}\right]\circ\dotsb\circ\\
&\circ\left[t_{\pi^2(\sigma^{-1}(u))}\circ g_{\pi(\sigma^{-1}(u))}\circ\left(t_{\pi(\sigma^{-1}(u))}\right)^{-1}\right]\circ\left[t_{\pi(\sigma^{-1}(u))}\circ g_{\sigma^{-1}(u)}\circ\left(t_{\sigma^{-1}(u)}\right)^{-1}\right]\\
=&t_{\sigma^{-1}(u)}\circ\left[g^\ell\right]_{\sigma^{-1}(u)}\circ\left(t_{\sigma^{-1}(u)}\right)^{-1}.
\end{split}
\]
Finally, set $v\coloneqq \sigma^{-1}(u)$.
\end{proof}

For $g\in\Aut(\mathcal{T})$ denote by $U_g$ a set of representatives for the orbits of $g$ on $V_\mathcal{T}\setminus W_\mathcal{T}$.
 
\begin{lemma}\label{lemmaconj}
Let $g$ be as above and suppose that also $g'\in\Aut(\mathcal{T})$, with $\pi'$ the restriction of $g'$ to $V_\mathcal{T}$. Then $g$ and $g'$ are conjugate in $\Aut(\mathcal{T})$ if and only if both the following conditions are verified
\begin{enumerate}[\rm(a)]
\item\label{lemmaconj1} $\pi,\pi'$ are conjugate in $\Aut(\mathcal{T})\rvert_{V_\mathcal{T}}$: $\pi'=\sigma\circ\pi\circ\sigma^{-1}$ with $\sigma\in\Aut(\mathcal{T})\rvert_{V_\mathcal{T}}$ (cf. \eqref{squaresigmap17});

\item  for every $v\in U_g$ there exists an isomorphism $t_v\colon\mathcal{T}_v\rightarrow\mathcal{T}_{\sigma(v)}$ ($\sigma$ as in \eqref{lemmaconj1}) such that 
\begin{equation}\label{congtgt3}
\left[\left(g'\right)^\ell\right]_{\sigma(v)}=t_v\circ\left[g^\ell\right]_v\circ\left(t_v\right)^{-1}.
\end{equation}
\end{enumerate}
\end{lemma}
\begin{proof}
The only if part follows from Lemma \ref{lemmacongtgt2}. Now suppose that \eqref{congtgt3} is verified for all $v\in U_g$. We define the isomorphisms $t_{\pi^k(v)}\colon\mathcal{T}_{\pi^k(v)}\rightarrow\mathcal{T}_{\sigma(\pi^k(v))}$ by setting inductively
\begin{equation}\label{congtgt4}
t_{\pi^k(v)}\coloneqq  g'_{\sigma(\pi^{k-1}(v))}\circ t_{\pi^{k-1}(v)}\circ\left(g_{\pi^{k-1}(v)}\right)^{-1},
\end{equation}
for $k=1,2,\dotsc,\ell-1$, where $\ell$ is the cardinality of the $\pi$-orbit of $v$. Note that \eqref{congtgt4} is just \eqref{congtgt1} with $tgt^{-1}$ replaced by $g'$ and $u=\sigma(\pi^{k-1}(v))$, so that we have just to verify \eqref{congtgt1} also in the case $k=\ell$, which is the relation involving $g_{\pi^{\ell-1}(v)}$ and $g'_{\sigma(\pi^{\ell-1}(v))}$. This will be achieved by using \eqref{congtgt3} as a compatibility condition: 

\begin{align*}
\left[g^\ell\right]_v=&g_{\pi^{\ell-1}(v)}\circ g_{\pi^{\ell-2}(v)}\circ\dotsb\circ g_{\pi(v)}\circ g_v&(\text{by }\eqref{defXi})\\
=&g_{\pi^{\ell-1}(v)}\circ\left[\left(t_{\pi^{\ell-1}(v)}\right)^{-1}\circ g'_{\sigma(\pi^{\ell-2}(v))}\circ t_{\pi^{\ell-2}(v)}\right]\circ&\\
&\circ\left[\left(t_{\pi^{\ell-2}(v)}\right)^{-1}\circ g'_{\sigma(\pi^{\ell-3}(v))}\circ t_{\pi^{\ell-2}(v)}\right]\circ&\\
&\circ\dotsb\circ\left[\left(t_{\pi^2(v)}\right)^{-1}\circ g'_{\sigma(\pi(v))}\circ t_{\pi(v)}\right]\circ\left[\left(t_{\pi(v)}\right)^{-1}\circ g'_{\sigma(v)}\circ t_v\right]&(\text{by }\eqref{congtgt4})\\
=&g_{\pi^{\ell-1}(v)}\circ\left(t_{\pi^{\ell-1}(v)}\right)^{-1}\circ g'_{\sigma(\pi^{\ell-2}(v))}\circ\dotsb\circ g'_{\sigma(\pi(v))}\circ g'_{\sigma(v)}\circ t_v&\\
=&g_{\pi^{\ell-1}(v)}\circ\left(t_{\pi^{\ell-1}(v)}\right)^{-1}\circ \left(g'_{\sigma(\pi^{\ell-1}(v))}\right)^{-1}\circ\left[\left(g'\right)^\ell\right]_{\sigma(v)}\circ t_v&(\text{by }\eqref{defXi})\\
=&g_{\pi^{\ell-1}(v)}\circ\left(t_{\pi^{\ell-1}(v)}\right)^{-1}\circ \left(g'_{\sigma(\pi^{\ell-1}(v))}\right)^{-1}\circ t_v\circ\left[g^\ell\right]_v.&(\text{by }\eqref{congtgt3})\\
\end{align*}
Therefore
\[
g'_{\sigma(\pi^{\ell-1}(v))}=t_v\circ g_{\pi^{\ell-1}(v)}\circ\left(t_{\pi^{\ell-1}(v)}\right)^{-1}
\]
which ends the proof of \eqref{congtgt1} in all cases so that, defining $t$ by means of $\sigma$ and the $t_u$'s, $u\in V_\mathcal{T}\setminus W_{\mathcal{T}}$, we get $tgt^{-1}=g'$.
\end{proof}

\begin{proposition}\label{Propconj}
With the notation in \eqref{defphig} and \eqref{defXi}, for all $u\in V_\mathcal{T}\setminus W_\mathcal{T}$ we have $(\varphi^g)_u=\varphi^{\left[g^\ell\right]_u}$.
\end{proposition}
\begin{proof}
Since $g$ sends $\mathcal{T}_{\pi^j(u)}$ to $\mathcal{T}_{\pi^{j+1}(u)}$, then $u',v'\in \mathcal{T}_u$ belong to the same $g$-orbit if and only if $v'=g^k(u')$ with $k$ a multiple of $\ell$. By \eqref{defXi} this happens if and only if $\left(\left[g^\ell\right]_u\right)^{k/\ell}(u')=v'$, that is if and only if $u',v'$ belongs to the same $\left[g^\ell\right]_u$-orbit.
\end{proof}

\begin{theorem}\label{Theoconj}
\begin{enumerate}[\rm(a)]
\item
Two elements $g,g'\in\Aut(\mathcal{T})$ are conjugate if and only if the corresponding tree compositions $\varphi^g$ and $\varphi^{g'}$ are equivalent.
\item\label{Theoconj2}
For every compositional tree $\mathcal{P}$ of $\mathcal{T}$ and for every $\psi\in\mathsf{C}(\mathcal{T},\mathcal{P})$ there exists $g\in\Aut(\mathcal{T})$ such that $\varphi^g=\psi$.
\end{enumerate}
\end{theorem}
\begin{proof}
\begin{enumerate}[\rm(a)]
\item
Suppose that $g,g'$ are conjugate: there exists $t\in\Aut(\mathcal{T})$ such that $g'=tgt^{-1}$. We define a tree isomorphism $\gamma\colon\mathcal{P}^g\rightarrow \mathcal{P}^{g'}$ by means of the natural correspondence between $g$-orbits and $g'$-orbits: if $\mathcal{P}^g\ni x=\{u,g(u),\dotsc,g^{\ell-1}(u)\}$ we set 
\[
\begin{split}
\gamma(x)\coloneqq x'\quad\text{ where }\quad\mathcal{P}^{g'}\ni x'=&\{t(u),t(g(u)),\dotsc,t(g^{\ell-1}(u))\}\\
\equiv&\{t(u),g'(t(u)),\dotsc,(g')^{\ell-1}(t(u))\}.
\end{split}
\]
Then we have $\left(\gamma\circ\varphi^g\right)(u)=x'=\left(\varphi^{g'}\circ t\right)(u)$, that is $\gamma\circ\varphi^g=\varphi^{g'}\circ t$ and $\varphi^g$ is equivalent to $\varphi^{g'}$. 

Conversely, suppose that $\varphi^g$ and $\varphi^{g'}$ are equivalent: there exist $\beta\in\Aut(\mathcal{T})$ and a tree isomorphism $\gamma\colon\mathcal{P}^g\rightarrow \mathcal{P}^{g'}$ such that the following diagram is commutative:
\begin{equation}\label{stardiagp25}
\xymatrix{
\mathcal{T}\ar[d]_\beta\ar[r]^{\varphi^g}&\mathcal{P}^g\ar[d]^\gamma\\\
\mathcal{T}\ar[r]^{\varphi^{g'}}&\mathcal{P}^{g'}}
\end{equation}
Suppose that $\pi,\pi'$ and $U_g$ are as in Lemma \ref{lemmaconj}. The proof is by induction on $H(\mathcal{T})$. For $H(\mathcal{T})=1$ it is obvious but we need to precise some facts. Choose $u\in U_g$ and set $x\coloneqq \varphi^g(u)$, so that $\left(\varphi^g\right)^{-1}(x)$ is the $\pi$-orbit of $u$. The $\beta$-image of $\left(\varphi^g\right)^{-1}(x)$ is $\left(\varphi^{g'}\right)^{-1}(\gamma(x))$, the $\pi'$-orbit of $\beta(u)$. Therefore these orbits are cyclic permutations of the same length, say $\ell$, so that there exists a bijection $\sigma^{(u)}\colon\left(\varphi^g\right)^{-1}(x)\rightarrow\left(\varphi^{g'}\right)^{-1}(\gamma(x))$ such that $\sigma^{(u)}\left(\pi^k(u)\right)=\pi'^k(\sigma^{(u)}(u)))$, $k=0,1,\dotsc,\ell-1$; we may also require that $\sigma^{(u)}(u)=\beta(u)$. The set of all $\sigma^{(u)}, u\in U_g$, define a permutation of $V_\mathcal{T}$ such that $\sigma\circ\pi\circ\sigma^{-1}=\pi'$, with the further property that $\sigma\Bigl(\left(\varphi^g\right)^{-1}(x)\Bigr)=\left(\varphi^{g'}\right)^{-1}(\gamma(x))$, which implies $\gamma\circ\varphi^g=\varphi^{g'}\circ\sigma$.

Now suppose that $H(\mathcal{T})\geq 2$. Arguing as in the case $H(\mathcal{T})=1$, we may construct a permutation $\sigma\in\Aut(\mathcal{T})\rvert_{V_\mathcal{T}}$ such that 
\[
\sigma\circ\pi\circ\sigma^{-1}=\pi',\quad \left(\gamma\rvert_{V_\mathcal{T}}\right)\circ\left(\varphi^g\rvert_{V_\mathcal{T}}\right)=\left(\varphi^{g'}\rvert_{V_\mathcal{T}}\right)\circ\sigma\quad\text{and}\quad \sigma(u)=\beta(u) \text{ for all }u\in U_g.
\]
We have just to notice that now each $\sigma^{(u)}\in\Aut(\mathcal{T})\rvert_{V_\mathcal{T}}$ because all the subtrees 
$\mathcal{T}_v$, $v\in \left(\varphi^g\right)^{-1}(x)\bigcup\left(\varphi^{g'}\right)^{-1}(\gamma(x))$,
are isomorphic; cf. \eqref{squaresigmap17} and \eqref{stardiagp25}. Now consider the automorphism $\beta^{-1}g'\beta$. By the first part of the proof, $\varphi^{g'}$ and $\varphi^{\beta^{-1}g'\beta}$ are equivalent and there exists a tree isomorphism $\gamma'\colon\mathcal{P}^{g'}\rightarrow\mathcal{P}^{\beta^{-1}g'\beta}$ such that 
\begin{equation}\label{phigammap26}
\gamma'\circ\varphi^{g'}=\varphi^{\beta^{-1}g'\beta}\circ\beta^{-1}. 
\end{equation}
By transitivity, also $\varphi^g$ and $\varphi^{\beta^{-1}g'\beta}$ are equivalent; in particular, for all $u\in U_g$, setting $x\coloneqq\varphi^g(u)$, $x'\coloneqq \gamma(x)$, we have, by \eqref{firstldell2}, \eqref{stardiagp25} and \eqref{phigammap26}
\[
\gamma'_{x'}\circ\gamma_x\circ\left(\varphi^g\right)_u=\gamma'_{x'}\circ\left(\varphi^{g'}\right)_{\beta(u)}\circ\beta_u
=\left(\varphi^{\beta^{-1}g'\beta}\right)_u,
\]
that is $\left(\varphi^g\right)_u$ and $\left(\varphi^{\beta^{-1}g'\beta}\right)_u$ are isomorphic, so that also equivalent. By Proposition \ref{Propconj}, also 
$\varphi^{\left[g^\ell\right]_u}$ and $\varphi^{\left[\left(\beta^{-1}g'\beta\right)^\ell\right]_u}$
are equivalent. Now we may invoke the inductive hypothesis: there exists $\widetilde{t}_u\in\Aut(\mathcal{T}_u)$ such that: $\widetilde{t}_u\circ \left[g^\ell\right]_u\circ\left(\widetilde{t}_u\right)^{-1}=\left[(\beta^{-1}g'\beta)^\ell\right]_u$. Therefore, from Lemma \ref{lemmacongtgt2} we deduce that 
\[
\widetilde{t}_u\circ \left[g^\ell\right]_u\circ\left(\widetilde{t}_u\right)^{-1}=\left(\beta^{-1}\right)_{\beta(u)}\circ\left[(g')^\ell\right]_{\beta(u)}\circ\left[\left(\beta^{-1}\right)_{\beta(u)}\right]^{-1};
\]
recalling that $\sigma(u)=\beta(u)$ and setting $t_u\coloneqq \left[\left(\beta^{-1}\right)_{\beta(u)}\right]^{-1}\circ\widetilde{t}_u\equiv\beta_u\circ\widetilde{t}_u$ (cf. \eqref{firstldell3}), we finally have
\[
t_u\circ\left[g^\ell\right]_u\circ(t_u)^{-1}=\left[(g')^\ell\right]_{\sigma(u)},\quad\text{ for all }u\in U_g.
\]
We may end the proof by invoking Lemma \ref{lemmaconj}
\item
By induction on $H(\mathcal{T})$. If $H(\mathcal{T})=1$ this is trivial, but we need to fix the notation. Suppose that $\psi\colon V_\mathcal{T}\rightarrow V_\mathcal{P}$ is a surjective map, for every $x\in V_\mathcal{P}$ fix $u_x\in\psi^{-1}(x)$ and choose an arbitrary cycle $\pi_x$ in $\psi^{-1}(x)$, say $u_x\rightarrow \pi_x(u_x)\rightarrow \dotsb\rightarrow \left(\pi_x\right)^{\ell_x-1}(u_x)\rightarrow u_x$, where $\ell_x\coloneqq\left\lvert \psi^{-1}(x)\right\rvert$. Then the product $\pi\coloneqq\prod_{x\in V_\mathcal{P}}\pi_x$ yields a permutation whose associated composition is $\psi$.

Now suppose that $H(\mathcal{T})\geq 2$ and fix a system of $\tau$'s for $\psi$ (cf. Definition \ref{Defpseudo}); actually, we will use just these $\tau$'s in the case $\psi(u)=\psi(v)$, when $\omega_{x,x}$ is the identity and we have strict equivalence as in \eqref{diagTRSuv2b}.
Then, at the first level, we argue as in the case $H(\mathcal{T})=1$: we construct the $\pi_x$'s and their product $\pi$. After that we apply induction to each subtree $\mathcal{T}_{u_x}$, $x\in V_\mathcal{P}\setminus W_\mathcal{P}$, getting automorphisms $\widetilde{g}_x\in\Aut(\mathcal{T}_{u_x})$ such that 
\begin{equation}\label{phipsigux}
\varphi^{\widetilde{g}_x}=\psi_{u_x}. 
\end{equation}
Then for every $u\in V_\mathcal{T}\setminus W_\mathcal{T}$ we define an isomorphisms $g_u\colon\mathcal{T}_u\rightarrow\mathcal{T}_{\pi(u)}$ by setting
\begin{equation}\label{defgpik}
g_{\pi^k(u_x)}\coloneqq \tau_{\pi^{k+1}(u_x),u_x}\circ\widetilde{g}_x\circ\tau_{u_x,\pi^k(u_x)}
\end{equation}
for all $x\in V_\mathcal{P}\setminus W_\mathcal{P}$ and $k=0,1,\dotsc,\ell_x-1$. At this point, we can define $g\in\Aut(\mathcal{T})$ by means of \eqref{firstldell}, using $\pi$ at the first level and the $g_u$'s. By the case $H(\mathcal{T})=1$, we have $\psi\rvert_{V_\mathcal{T}}=\varphi^g\rvert_{V_\mathcal{T}}$. Moreover, for every $x\in V_\mathcal{P}\setminus W_\mathcal{P}$, we have:
\[
\left[g^{\ell_x}\right]_{u_x}=g_{\pi^{\ell_x-1}(u_x)}\circ g_{\pi^{\ell_x-2}(u_x)}\circ\dotsb\circ g_{\pi(u_x)}\circ g_{u_x}
=\widetilde{g}_x\qquad(\text{by }\eqref{defXi}, \eqref{proptau2b}\text{ and }\eqref{defgpik})
\]
and therefore
\[
\left(\varphi^g\right)_{u_x}=\varphi^{\left[g^{\ell_x}\right]_{u_x}}
=\varphi^{\widetilde{g}_x}
=\psi_{u_x}\qquad(\text{by Proposition }\ref{Propconj}, \left[g^{\ell_x}\right]_{u_x}=\widetilde{g}_x \text{ and }\eqref{phipsigux})
\]
so that finally
\begin{align*}
\left(\varphi^g\right)_{\pi^k(u_x)}=&\left(\varphi^g\right)_{u_x}\circ \left[(g^k)_{u_x}\right]^{-1}&(\text{by }\eqref{starp25May})\\
=&\left(\varphi^g\right)_{u_x}\circ\left(\widetilde{g}_x\right)^{-k}\circ\tau_{u_x,\pi^k(u_x)} &(\text{by }\eqref{firstldell2}, \eqref{proptau2b}\text{ and }\eqref{defgpik})\\
=&\psi_{u_x}\circ\tau_{u_x,\pi^k(u_x)} &(\text{by }\left(\varphi^g\right)_{u_x}=\varphi^{\widetilde{g}_x}\text{ and }\varphi^{\widetilde{g}_x}\circ\widetilde{g}_x=\varphi^{\widetilde{g}_x})\\
=&\psi_{\pi^k(u_x)}. &(\text{by }\eqref{diagTRSuv2b})\\
\end{align*}
This ends the proof that $\varphi^g=\psi$.
\end{enumerate}
\end{proof}

\begin{corollary}\label{Corconjclasses}
There is a natural bijection between the set of all conjugacy classes of $\Aut(\mathcal{T})$ and $\mathsf{Par}(\mathcal{T})$ (cf. Definition \ref{defIrrG}). It is obtained by associating to every conjugacy class $C$ of $\Aut(\mathcal{T})$ the tree of partition $(\Lambda,\mathcal{Q})$ representing the equivalence class of tree compositions $\left\{\varphi^g\colon g\in C\right\}$.
\end{corollary}
\begin{proof}
It follows immediately from Theorem \ref{Theoorbits}\eqref{Theoorbitsa} and Theorem \ref{Theoconj}.
\end{proof}

We recall that a group $G$ is {\em ambivalent} when every element of $G$ is conjugate to its inverse.

\begin{corollary}
The group $\Aut(\mathcal{T})$ is ambivalent.
\end{corollary}
\begin{proof}
Just note that $\varphi^g=\varphi^{g^{-1}}$ because $g$ and $g^{-1}$ have the same orbits on $\mathcal{T}$.
\end{proof}

The following corollary has an obvious elementary proof but we prefer to deduce it from the theory of this section. 

\begin{corollary}
The group $\Aut(\mathcal{T})$ is trivial if and only if $\mathcal{T}$ has only the identity composition.
\end{corollary}

The following Corollary follows by applying Corollary \ref{Corconjclasses} to the trivial composition in Example \ref{treeparttvid}. See also \cite[Theorem (4.6)]{BORT} and \cite[Corollary 4.1]{GNS}, where the fact that the spherically transitive automorphisms form a single conjugacy class is proved for infinite spherically homogeneous trees.

\begin{corollary}\label{Corgenaddmach}
There exists $g\in\Aut(\mathcal{T})$ which has the same orbits of the whole group $\Aut(\mathcal{T})$. Moreover, two automorphisms with this property are conjugated in $\Aut(\mathcal{T})$.
\end{corollary}

\section{The irreducible representations via Clifford theory}\label{SecClifford}

In the present section, we want to apply the representation theory of groups of the form $F\wr S_n$ to $\Aut(\mathcal{T})$, taking the decomposition in \eqref{Autp17} into account;  cf. \cite[Theorem 2.6.1]{book3}, \cite[Theorem 4.4.3]{JK} . This way we prove that there is a 
natural bijection between $\mathsf{Par}(\mathcal{T})$ and the set of equivalence classes of irreducible representations of $\Aut(\mathcal{T})$.
We first make a preliminary observation. In Clifford theory, if $N\trianglelefteq I$ then the extension of a representation from $N$ to $I$, if it exists, in general it is not unique and all the extensions may be obtained by tensoring a fixed one with the one dimensional representations of $I/N$; cf. \cite[Theorem 1.3.9]{book3}, \cite[Theorem 2.18]{book6}. In our setting, we have to choose carefully the extension because the tensor product with a one dimensional representation as above would change some of the involved partitions with their transposes. We will follow this strategy: we will use the permutation representations that arise from the action on the tree compositions and we will construct the irreducible representations as multiplicity one subrepresentations of them. 

In what follows, the action \eqref{action2Aut} in the case $k=1_{\text{\rm Aut}(\mathcal{P})}$ will be denoted by:
\begin{equation}\label{defgvaract}
g\varphi\coloneqq\varphi\circ g^{-1},
\end{equation} 
for all $g\in\Aut(\mathcal{T})$; in particular, by \eqref{firstldell3} we have 
\begin{equation}\label{defgvaract2}
[g\varphi]_u=\varphi_{g^{-1}(u)}\circ\left(g^{-1}\right)_u=\varphi_{g^{-1}(u)}\circ\left(g_{g^{-1}(u)}\right)^{-1}, \quad\text{ for all } u\in V_\mathcal{T}\setminus W_\mathcal{T}.
\end{equation} 
If $\varphi\in\mathsf{C}(\mathcal{T},\mathcal{P})$ we set $(\mathcal{P}, \lvert\varphi^{-1}\rvert)\coloneqq \{g\varphi\colon g\in \Aut(\mathcal{T})\}\equiv$ the $\Aut(\mathcal{T})$-orbit on $\mathsf{C}(\mathcal{T},\mathcal{P})$ containing $\varphi$, cf. Theorem \ref{Theoorbits}\eqref{Theoorbitsb}. The stabilizer of $\varphi$ in $\Aut(\mathcal{T})$ will be denoted $K_\varphi$ so that $(\mathcal{P},\lvert\varphi^{-1}\rvert)=\Aut(\mathcal{T})/K_\varphi$ as homogeneous spaces and 
\begin{equation}\label{deforbitphi}
\mathsf{C}(\mathcal{T},\mathcal{P})=\bigsqcup_{\lvert\varphi^{-1}\rvert\in\mathsf{Type}(\mathcal{T},\mathcal{P})}(\mathcal{P},\lvert\varphi^{-1}\rvert).
\end{equation}

\begin{definition}\label{defYoungmodules}{\rm
If $\varphi\in\mathsf{C}(\mathcal{T},\mathcal{P})$ we denote by $M^{\lvert\varphi^{-1}\rvert}$ the permutation representation of $\Aut(\mathcal{T})$ on $(\mathcal{P},\lvert\varphi^{-1}\rvert)$.
}
\end{definition}

That is, $M^{\lvert\varphi^{-1}\rvert}$ is the vector space of all functions $f\colon(\mathcal{P},\lvert\varphi^{-1}\rvert)\rightarrow\mathbb{C}$ with the left action $[gf](\psi)\coloneqq f(g^{-1}\psi)$, for all $g\in\Aut(\mathcal{T})$, $\psi\in(\mathcal{P},\lvert\varphi^{-1}\rvert)$.

\begin{proposition}\label{PropinvMMt}
If $\varphi$ and $\psi$ are equivalent as tree compositions of $\mathcal{T}$ then $M^{\lvert\varphi^{-1}\rvert}$ and $M^{\lvert\psi^{-1}\rvert}$ are equivalent
 as $\Aut(\mathcal{T})$-representations.
\end{proposition}
\begin{proof}
If $\varphi$ and $\psi$ are equivalent tree compositions as in the first diagram in \eqref{defequiviso}, but with $\mathcal{T}_1=\mathcal{T}_2=\mathcal{T}$, then $\psi=\omega\circ\varphi\circ\tau^{-1}$, with $\tau\in\Aut(\mathcal{T})$ and $\omega$ an isomorphism between the corresponding compositional trees. Therefore $\psi\circ g^{-1}=\omega\circ\varphi\circ\left(\tau^{-1}\circ g^{-1}\circ\tau\right)\circ\tau^{-1}$ for all $g\in\Aut(\mathcal{T})$ and it follows that 
\begin{equation}\label{Kconj}
K_{\varphi}=\tau^{-1}K_{\psi}\tau.
\end{equation}
From \cite[Proposition 3.1.11]{book} we deduce that the homogeneous spaces $\Aut(\mathcal{T})/K_{\varphi_1}$ and $\Aut(\mathcal{T})/K_{\varphi_2}$ are isomorphic and therefore the corresponding permutation representations are equivalent. 
\end{proof}

From Theorem \ref{Theoorbits}\eqref{Theoorbitsa} and Proposition \ref{PropinvMMt}, we  deduce that there is a natural bijection between $\mathsf{Par}(\mathcal{T})$ and the permutation module $M^{\lvert\varphi^{-1}\rvert}$'s.  By means of the tools developed in the preceding sections, we will analyze these permutation representations in the setting of Clifford theory. Suppose that $\varphi\in\mathsf{C}(\mathcal{T},\mathcal{P})$ and that $(\Lambda,\mathcal{Q})$ is its associated tree of partitions, with $\alpha$ as in Theorem \ref{Prop1p5}. By Proposition \ref{Propepsphitv}, applied to the realization $\delta\coloneqq\alpha\circ\varphi$, there exists a surjective tree homomorphism $\varepsilon\colon\mathcal{Q}\rightarrow\mathcal{C}$ such that $\varepsilon\circ\delta=\tv_\mathcal{T}$. In the following diagrams we summarize the maps that will be used and the corresponding first level compositions; the first is commutative.
\[
\xymatrix@R=8pt{
\mathcal{T}\ar@/^1.4pc/[rr]^(.55){\delta}\ar@/_1.4pc/[rrr]_(.55){\text{\rm tv}_\mathcal{T}}\ar[r]^{\varphi}&\mathcal{P}\ar[r]^{\alpha}&\mathcal{Q}\ar[r]^{\varepsilon}&\mathcal{C}\\
{\scriptstyle \bigsqcup\limits_{c\in V_\mathcal{C}}\bigsqcup\limits_{a\in\varepsilon^{-1}(c)}\delta^{-1}(a)}&&&\\
{\scriptstyle V_\mathcal{T}=\bigsqcup\limits_{c\in V_\mathcal{C}}\bigsqcup\limits_{a\in\varepsilon^{-1}(c)}\bigsqcup\limits_{x\in \alpha^{-1}(a)}\varphi^{-1}(x)}\ar@{=}[d]\ar@{=}[u]&{\scriptstyle V_\mathcal{P}=\bigsqcup\limits_{c\in V_\mathcal{C}}\bigsqcup\limits_{a\in\varepsilon^{-1}(c)}\alpha^{-1}(a)}\ar[l]&\ar[l]{\scriptstyle V_\mathcal{Q}=\bigsqcup\limits_{c\in V_\mathcal{C}}\varepsilon^{-1}(c)}\ar@/_1.4pc/[llu]&\ar[l]{\scriptstyle V_\mathcal{C}}\ar@/^1.4pc/[llld]\\
{\scriptstyle \bigsqcup\limits_{c\in V_\mathcal{C}}\text{\rm tv}_\mathcal{T}^{-1}(c)}&&&
} 
\]
We set 
\begin{equation}\label{irredsigcbis}
N\coloneqq\prod_{u\in V_\mathcal{T}}\Aut\left(\mathcal{T}_u\right)\qquad\text{ and }\qquad I\coloneqq\prod_{c\in V_\mathcal{C}}\left[\Aut\left(\mathcal{T}_{u_c}\right)\wr\prod_{a\in\varepsilon^{-1}(c)}\Sym\left(\delta^{-1}(a)\right)\right],
\end{equation}
where we think of $I$ as the subgroup of all $g\in\Aut(\mathcal{T})$ such that $g\left(\delta^{-1}(a)\right)=\delta^{-1}(a)$, for all $a\in V_\mathcal{Q}$. Clearly,
$N$ is a normal subgroup of $\Aut(\mathcal{T})$; we will show that $I$ an inertia group. 
Every $\psi\in(\mathcal{P},\lvert\varphi^{-1}\rvert)$ determines a set composition $V_\mathcal{T}=\bigsqcup_{x\in V_\mathcal{P}}\psi^{-1}(x)$ and we denote by $\widetilde{\Gamma}_\varphi$ the set of all the set compositions that arise in this way. The group \eqref{squaresigmap17} acts on $\widetilde{\Gamma}_\varphi$ as in Example \ref{exampleh1}: $\sigma\gamma\coloneqq\gamma\circ\sigma^{-1}$; equivalently, $\sigma\left(\bigsqcup_{x\in V_\mathcal{P}}\gamma^{-1}(x)\right)=\bigsqcup_{x\in V_\mathcal{P}}\sigma\left(\gamma^{-1}(x)\right)$. For $\gamma\in\widetilde{\Gamma}_\varphi$ we denote by $(\mathcal{P},\lvert\varphi^{-1}\rvert)_\gamma$ the set of all $\psi$ such that $\psi^{-1}(x)=\gamma^{-1}(x)$, for all $x\in V_\mathcal{P}$. In particular, $\gamma_0$ will denote that element of $\widetilde{\Gamma}_\varphi$ such that $\gamma_0^{-1}(x)=\varphi^{-1}(x)$, for all $x\in V_\mathcal{P}$. After that, we denote by $\Gamma_\varphi$ the set of all $\gamma\in\widetilde{\Gamma}_\varphi$ such that 
\begin{equation}\label{gamxdela}
\gamma^{-1}(x)\subseteq\delta^{-1}(a)\;\text{ if }\;\alpha(x)=a.
\end{equation}
From \eqref{defgvaract}, \eqref{defgvaract2} and \eqref{gamxdela} we deduce that the group $I$
acts transitively on the set $\bigsqcup_{\gamma\in\Gamma_\varphi}(\mathcal{P},\lvert\varphi^{-1}\rvert)_\gamma$. Each set $(\mathcal{P},\lvert\varphi^{-1}\rvert)_\gamma$, $\gamma\in\Gamma_\varphi$, is $N$-invariant and, with respect to this group, all of them are isomorphic homogeneous spaces (we are taking automorphisms that fixes all the vertices in $V_\mathcal{T}$); all the corresponding permutation representations of $N$ are equivalent to $\bigotimes_{u\in V_\mathcal{T}}M^{\lvert\varphi_u^{-1}\rvert}$. By Theorem \ref{Prop1p5} and Proposition \ref{PropinvMMt}, $M^{\lvert\varphi_u^{-1}\rvert}$ and $M^{\lvert\varphi_w^{-1}\rvert}$ are equivalent if and only if $\delta(u)=\delta(w)$. 

From \eqref{defgvaract} and \eqref{defgvaract2} it follows that $(\mathcal{P},\lvert\varphi^{-1}\rvert)_{\gamma_0}$ is invariant also under the action of
\[
K\coloneqq\prod_{c\in V_\mathcal{C}}\left\{\Aut\left(\mathcal{T}_{u_c}\right)\wr\prod_{a\in \varepsilon^{-1}(c)}\left[\prod_{x\in\alpha^{-1}(a)}\Sym\left(\varphi^{-1}(x)\right)\right]\right\},
\]
which is the subgroup of all $g\in\Aut(\mathcal{T})$ such that $g\left(\varphi^{-1}(x)\right)=\varphi^{-1}(x)$, for all $x\in V_\mathcal{P}$ (equivalently, $(g\varphi)(u)=\varphi(u)$ for all $u\in V_\mathcal{T}$).
Let the $\omega$'s be as in Definition \ref{Defpseudo} for $\varphi$.

\begin{lemma}\label{lemcdotact}
 For $g\in I$ and $\psi\in(\mathcal{P},\lvert\varphi^{-1}\rvert)_{\gamma_0}$ define $g\cdot\psi$ by settting $\left[g\cdot\psi\right](u)\coloneqq \psi(u)$ for all $u\in V_\mathcal{T}$ and then
\begin{equation}\label{defcdotact}
[g\cdot\psi]_u\coloneqq\omega_{\psi(u),\psi\left(g^{-1}u\right)}\circ\psi_{g^{-1}u}\circ\left(g^{-1}\right)_u.
\end{equation}
\[
\xymatrix{
\mathcal{T}_u\ar[d]_{\left(g^{-1}\right)_u}\ar[rr]^{[g\cdot\psi]_u}&&\mathcal{P}_{\psi(u)}\\
\mathcal{T}_{g^{-1}u}\ar[rr]_{\psi_{g^{-1}u}}&&\mathcal{P}_{\psi\left(g^{-1}u\right)}\ar[u]_{\omega_{\psi(u),\psi\left(g^{-1}u\right)}}}.
\]
Then this is an action of $I$ on $(\mathcal{P},\lvert\varphi^{-1}\rvert)_{\gamma_0}$ which extends the actions of $N$ and $K$.
\end{lemma}
\begin{proof}
It is obvious that $g\cdot\psi\in(\mathcal{P},\lvert\varphi^{-1}\rvert)_{\gamma_0}$ for all $\psi\in(\mathcal{P},\lvert\varphi^{-1}\rvert)_{\gamma_0}$, $g\in I$.
If also $h\in I$ then
\[
\begin{split}
[(h\cdot(g\cdot\psi)]_u=&\omega_{\psi(u),(g\cdot\psi)\left(h^{-1}u\right)}\circ(g\cdot\psi)_{h^{-1}u}\circ\left(h^{-1}\right)_u\\
(\text{by } \eqref{firstldell3})\quad=&\omega_{\psi(u),\psi\left(h^{-1}u\right)}\circ\omega_{\psi\left(h^{-1}u\right),\psi\left(g^{-1}h^{-1}u\right)}\circ\psi_{g^{-1}h^{-1}u}\circ\left(g_{g^{-1}h^{-1}u}\right)^{-1}\circ\left(h_{h^{-1}u}\right)^{-1}\\
(\text{by }\eqref{proptau2b})\quad=&\omega_{\psi(u),\psi\left((hg)^{-1}u\right)}\circ\psi_{(gh)^{-1}u}\circ\left(h_{h^{-1}u}\circ g_{(hg)^{-1}u}\right)^{-1}\\
=&\omega_{\psi(u),\psi\left((hg)^{-1}u\right)}\circ\psi_{(gh)^{-1}u}\circ\left[\left(h g\right)_{(hg)^{-1}u}\right]^{-1}
=[(hg)\cdot\psi)]_u,
\end{split}
\]
so that we have proved that it is an action. Moreover, its restriction to both $N$ and $K$ coincides with the natural actions: $(g\cdot\psi)_u=\psi_u\circ g_u^{-1}$ if $g\in N$ and, by $\omega_{x,x}=\id_{\mathcal{P}_x}$, $(g\cdot\psi)_u=\psi_{g^{-1}u}\circ \left(g^{-1}\right)_u$ if $g\in K$. 
\end{proof}

\begin{remark}
{\rm
If $g\in I$ and $\gamma=g^{-1}\gamma_0$ then the map 
\[
(\mathcal{P},\lvert\varphi^{-1}\rvert)_\gamma\ni\xi\rightarrow T_g\xi\in(\mathcal{P},\lvert\varphi^{-1}\rvert)_{\gamma_0},
\]
defined by setting $\left[T_g\xi\right](u)\coloneqq \xi\left(g^{-1}u\right)$ and $\left[T_g\xi\right]_u\coloneqq\omega_{\xi\left(g^{-1}u\right),\xi(u)}\circ\xi_u$, for all $u\in V_\mathcal{T}$, is a bijection. This way we may recover \eqref{defcdotact}: if $g\in I$, $\psi\in(\mathcal{P},\lvert\varphi^{-1}\rvert)_{\gamma_0}$ and $g\psi\in(\mathcal{P},\lvert\varphi^{-1}\rvert)_\gamma$ then $g\cdot\psi=T_{g^{-1}}(g\psi)$, that is $g\;\cdot\equiv T_{g^{-1}}\circ g$. In other words, by means of the bijections $T_g$ the natural $g$-image of $\psi$ (given by \eqref{defgvaract}) is returned back to $(\mathcal{P},\lvert\varphi^{-1}\rvert)_{\gamma_0}$ and this way we get an action of $I$ on $(\mathcal{P},\lvert\varphi^{-1}\rvert)_{\gamma_0}$.
}
\end{remark}

We denote by $W^\varphi$ the permutation representation of $I$ corresponding to \eqref{defcdotact}, so that $\Res^I_NW^\varphi=\bigotimes_{u\in V_\mathcal{T}}M^{\lvert\varphi_u^{-1}\rvert}$, that is we may consider $W^\varphi$ as an extension of $\bigotimes_{u\in V_\mathcal{T}}M^{\lvert\varphi_u^{-1}\rvert}$ from $N$ to $I$; we also use the fact that 
\begin{equation}\label{ResIKW}
\Res^I_KW^\varphi=\bigotimes_{u\in V_\mathcal{T}}M^{\lvert\varphi_u^{-1}\rvert},
\end{equation} 
where $\bigotimes_{u\in V_\mathcal{T}}M^{\lvert\varphi_u^{-1}\rvert}$ is endowed with the natural representation of $K$. It is useful to write down the action on tensor products corresponding to the permutation representation $W^\varphi$.
\begin{lemma}\label{lemcdotacttens}
If $F_u\in M^{\lvert\varphi_u\rvert}$ then, for all $g\in I$,
\begin{equation}\label{exttensprod}
g\left[\otimes_{u\in V_\mathcal{T}} F_u\right]=\otimes_{u\in V_\mathcal{T}} H_u,\quad\text{ where } \quad H_u\left(\psi_u\right)=F_{g^{-1}u}\left(\omega_{\psi\left(g^{-1}u\right),\psi(u)}\circ\psi_{u}\circ g_{g^{-1}u}\right).
\end{equation}
\end{lemma}
\begin{proof}
If $\psi\in(\mathcal{P},\lvert\varphi^{-1}\rvert)_{\gamma_0}$ then
$\left[\otimes_{u\in V_\mathcal{T}} F_u\right](\psi)=\prod_{u\in V_\mathcal{T}}F_u(\psi_u)$; if also $g\in N$ then
\[
\left\{g\left[\otimes_{u\in V_\mathcal{T}} F_u\right]\right\}(\psi)=\prod_{u\in V_\mathcal{T}}F_u\left(\psi_u\circ g_u\right)=\prod_{u\in V_\mathcal{T}}\left[g_uF_u\right]\left(\psi_u\right)\;\Longrightarrow\; g\left[\otimes_{u\in V_\mathcal{T}} F_u\right]=\otimes_{u\in V_\mathcal{T}}\left[g_uF_u\right].
\]
Therefore if $g\in I$ then $[gF](\psi)=F\left((g^{-1})\cdot\psi\right)$ so that 
\[
\begin{split}
\left\{g\left[\otimes_{u\in V_\mathcal{T}} F_u\right]\right\}(\psi)=&\prod_{u\in V_\mathcal{T}}F_u\left(\left[(g^{-1})\cdot\psi\right]_u\right)=\prod_{u\in V_\mathcal{T}}F_u\left(\omega_{\psi(u),\psi\left(gu\right)}\circ\psi_{gu}\circ g_u\right)\\
(u\rightarrow g^{-1}u)\quad=&\prod_{u\in V_\mathcal{T}}F_{g^{-1}u}\left(\omega_{\psi\left(g^{-1}u\right),\psi(u)}\circ\psi_{u}\circ g_{g^{-1}u}\right).\\
\end{split}
\]
\end{proof}

After that, we set
\[
\begin{split}
X\coloneqq &I/K\cong (I/N)/(K/N)\cong\left[\prod_{a\in V_\mathcal{Q}}\Sym\left(\delta^{-1}(a)\right)\right]/\left\{\prod_{a\in V_\mathcal{Q}}\left[\prod_{x\in\alpha^{-1}(a)}\Sym\left(\varphi^{-1}(x)\right)\right]\right\}\\
&\cong\prod_{a\in V_\mathcal{Q}}\left\{\Sym\left(\delta^{-1}(a)\right)/\left[\prod_{x\in\alpha^{-1}(a)}\Sym\left(\varphi^{-1}(x)\right)\right]\right\}\cong
\prod_{a\in V_\mathcal{Q}}S_{\lambda^a},
\end{split}
\]
where $S_{\lambda^a}$ is the Young subgroup associated to the partition $\lambda^a$. It follows that the permutation representation of $I$ on $X$ coincides with the inflation of $\bigotimes_{a\in V_\mathcal{Q}}M^{\lambda^a}$ from $I/N$ to $I$. 

\begin{theorem}
We have
\begin{equation}\label{MvarphiCli}
M^{\lvert\varphi^{-1}\rvert}=\Ind_I^{\text{\rm Aut}(\mathcal{T})}\left\{\left[\widetilde{\bigotimes}_{u\in V_\mathcal{T}}M^{\lvert\varphi_u^{-1}\rvert}\right]\bigotimes\left[\overline{\bigotimes}_{a\in V_\mathcal{Q}}M^{\lambda^a}\right]\right\},
\end{equation}
where where $\widetilde{\otimes}$ and $\overline{\otimes}$ denote, respectively, the extension by means of \eqref{exttensprod} and the inflation described above.
\end{theorem}
\begin{proof}
Let $\Sigma$ be a set of representatives for the left cosets of $K/N$ in $I/N$. Then $\Gamma_\varphi=\bigsqcup_{\sigma\in \Sigma}\{\sigma\gamma_0\}$ and therefore 
\begin{equation}\label{gamGamsigSig}
\bigsqcup_{\gamma\in\Gamma_\varphi}(\mathcal{P},\lvert\varphi^{-1}\rvert)_\gamma=\bigsqcup_{\sigma\in \Sigma}(\mathcal{P},\lvert\varphi^{-1}\rvert)_{\sigma\gamma_0}.
\end{equation}
If $U^\varphi$ denotes the permutation representation of $I$ on $\bigsqcup_{\gamma\in\Gamma_\varphi}(\mathcal{P},\lvert\varphi^{-1}\rvert)_\gamma$ (with respect to \eqref{defgvaract}) then
\begin{align*}
U^\varphi=&\Ind_K^I \left[\bigotimes_{u\in V_\mathcal{T}}M^{\lvert\varphi_u^{-1}\rvert}\right]&(\text{by }\eqref{gamGamsigSig})\\
=&\Ind_K^I\Res_K^IW^\varphi& (\text{by }\eqref{ResIKW})\\
=&\left[\widetilde{\bigotimes}_{u\in V_\mathcal{T}}M^{\lvert\varphi_u^{-1}\rvert}\right]\bigotimes\left[\overline{\bigotimes}_{a\in V_\mathcal{Q}}M^{\lambda^a}\right].&(\text{by }\cite[\rm Corollary\; 11.1.17]{book4})
\end{align*}
Finally, if $T$ a set of representatives for the left cosets of $\prod_{c\in V_\mathcal{C}}\prod_{a\in \varepsilon^{-1}(c)}\left[\Sym\left(\delta^{-1}(a)\right)\right]$ in \eqref{squaresigmap17} then we have the compositions $\widetilde{\Gamma}_\varphi=\bigsqcup_{\tau\in T}\tau\Gamma_\varphi$ and 
\[
(\mathcal{P},\lvert\varphi^{-1}\rvert)=\bigsqcup_{\gamma\in\widetilde{\Gamma}_\varphi}(\mathcal{P},\lvert\varphi^{-1}\rvert)_\gamma=\bigsqcup_{\tau\in T}\left[\bigsqcup_{\gamma\in\Gamma_\varphi}(\mathcal{P},\lvert\varphi^{-1}\rvert)_{\tau\gamma}\right]
\]
so that $M^{\lvert\varphi^{-1}\rvert}=\Ind_I^{\text{\rm Aut}(\mathcal{T})}U^\varphi$.
\end{proof}

Now we apply Clifford theory to the construction of the irreducible representations. We denote by $\widehat{\Aut(\mathcal{T})}$ the set of all equivalence classes of irreducible representations of $\Aut(\mathcal{T})$.

\begin{theorem}\label{irredClifftheo}
There is a natural bijection between $\widehat{\Aut(\mathcal{T})}$ and $\mathsf{Par}(\mathcal{T})$; we denote by $Z^\Lambda$ the irreducible representation corresponding to $\Lambda\in\mathsf{Par}(\mathcal{T})$.
Moreover, if $\varphi$ is a tree composition of $\mathcal{T}$ with $\Lambda$ as its tree of partitions then it is possible to construct $Z^\Lambda$ in such a way that it is a subspace of $M^{\lvert\varphi^{-1}\rvert}$ (where it has multiplicity one). 
\end{theorem}
\begin{proof}
The proof is by induction on $H(\mathcal{T})$. For $H(\mathcal{T})=1$ it is just the natural bijection between irreducible representations of $\Sym\left(V_\mathcal{T}\right)$ and the integer partitions of $\lvert V_\mathcal{T}\rvert$ and the fact that the irreducible representation $S^\lambda$ is contained in the permutation module $M^\lambda$, with multiplicity one, is well known. 

Now suppose that $H(\mathcal{T})>1$ and that we have constructed a bijection for trees of height $<H(\mathcal{T})$. If $(\Lambda',\mathcal{Q}')$ is a tree of partitions of height $<H(\mathcal{T})$ we denote by $Z^{\Lambda'}$ the associated irreducible representation and we suppose that it may be constructed in such a way that it is contained, with multiplicity one, in the permutation representation $M^{\lvert\varphi'\rvert}$, if $(\Lambda',\mathcal{Q}')$ is the tree of partitions associated to $\varphi'$. 

In order to apply Clifford theory to $\Aut(\mathcal{T})$, we must choose an irreducible representation of each group $\Aut\left(\mathcal{T}_u\right)$, $u\in V_\mathcal{T}\setminus W_\mathcal{T}$ and an irreducible representation of $\Sym\left(W_\mathcal{T}\right)$; by the inductive hypothesis, this means that we choose $(\Lambda^{(u)},\mathcal{Q}^{(u)})\in\mathsf{Par}\left(\mathcal{T}_u\right)$ and a partition $\mu$ of $\lvert W_\mathcal{T}\rvert$. 
Then we introduce the following equivalence relation: for $c\in V_\mathcal{C}\setminus W_\mathcal{T}$, $u,v\in\tv_\mathcal{T}^{-1}(c)$, we write $u\approx v$ if the trees of partitions $(\Lambda^{(u)},\mathcal{Q}^{(u)})$ and $(\Lambda^{(v)},\mathcal{Q}^{(v)})$ are isomorphic (that is, if the irreducible representations $Z^{\Lambda^{(u)}}$ and $Z^{\Lambda^{(v)}}$ are equivalent). We denote by $A_c$ the quotient set $\tv_\mathcal{T}^{-1}(c)/\approx$, by $\delta\colon \tv_\mathcal{T}^{-1}(c)\rightarrow A_c$ the quotient map and by $v_a$ the choice of a representative in $\delta^{-1}(a)$, for each $a\in A_c$. After that, we set
\begin{equation}\label{irredsigc}
Z\coloneqq\bigotimes_{u\in V_\mathcal{T}}Z^{\Lambda^{(u)}}\sim\bigotimes_{c\in V_\mathcal{C}}\bigotimes_{a\in A_c}\left[\bigotimes_{u\in \delta^{-1}(a)}Z^{\Lambda^{(u)}}\right]\sim\bigotimes_{c\in V_\mathcal{C}}\bigotimes_{a\in A_c}\left(Z^{\Lambda^{(v_a)}}\right)^{\otimes^{\lvert\delta^{-1}(a)\rvert}}
\end{equation}
which is an irreducible representation of $N$ whose inertia group is $I$ (cf. \eqref{irredsigcbis}); just replace, for the moment, $\varepsilon^{-1}(a)$ with $A_c$). Subsequentially, we choose an irreducible representation $S^{\lambda^a}$ of each symmetric group $\Sym\left(\delta^{-1}(a)\right)$, so that $\lambda^a$ is a partition of $\left\lvert \delta^{-1}(a)\right\rvert$, we set
\begin{equation}\label{irredrhoc}
S\coloneqq\bigotimes_{c\in V_\mathcal{C}}\bigotimes_{a\in A_c}S^{\lambda^a}
\end{equation}
and we denote by $\overline{S}$ the inflation of $S$ from $\prod_{c\in V_\mathcal{C}}\prod_{a\in A_c}\Sym\left(\delta^{-1}(a)\right)$ to $I$.

Now we may construct a tree of partitions $(\Lambda,\mathcal{Q})$ in the obvious way: $V_\mathcal{Q}\setminus W_\mathcal{Q}\coloneqq \sqcup_{c\in V_\mathcal{C}}A_c$, $W_\mathcal{Q}$ is made up of one vertex if $W_\mathcal{T}\neq \emptyset$, otherwise it is empty, the label of $a\in V_\mathcal{Q}\setminus W_\mathcal{Q}$
is $\lambda^a$, the label of the unique vertex in $W_\mathcal{Q}$ is $\mu$. Finally, to each $a\in V_\mathcal{Q}\setminus W_\mathcal{Q}$ we attach a copy of the tree $\mathcal{Q}^{(v_a)}$ with its labeling $\Lambda^{(v_a)}$. In the following picture, $V_\mathcal{Q}\setminus W_\mathcal{Q}=\{a,a',\dotsc,a''\}$.

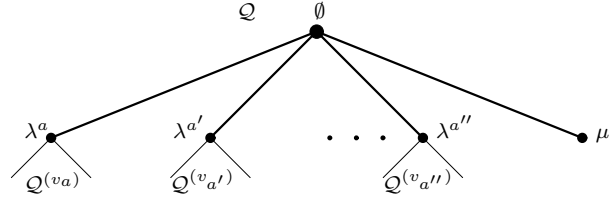
\begin{picture}(450,85)
\put(170,75){${\scriptstyle \mathcal{Q}}$}
\put(200,70){\circle*{5}}
\put(199,75){${\scriptstyle \emptyset}$}

\thicklines
\drawline(200,70)(100,30)
\put(100,30){\circle*{3}}
\put(90,30){${\scriptstyle \lambda^{a}}$}

\drawline(200,70)(160,30)
\put(160,30){\circle*{3}}
\put(146,30){${\scriptstyle \lambda^{a'}}$}

\put(205,30){\circle*{1}}
\put(215,30){\circle*{1}}
\put(225,30){\circle*{1}}

\drawline(200,70)(240,30)
\put(240,30){\circle*{3}}
\put(245,30){${\scriptstyle \lambda^{a''}}$}

\drawline(200,70)(300,30)
\put(300,30){\circle*{3}}
\put(305,30){${\scriptstyle \mu}$}

\thinlines
\put(100,30){\line(-1,-1){15}}
\put(100,30){\line(1,-1){15}}
\put(90,10){${\scriptstyle \mathcal{Q}^{(v_a)}}$}

\put(160,30){\line(-1,-1){15}}
\put(160,30){\line(1,-1){15}}
\put(145,10){${\scriptstyle \mathcal{Q}^{(v_{a'})}}$}

\put(240,30){\line(-1,-1){15}}
\put(240,30){\line(1,-1){15}}
\put(225,10){${\scriptstyle \mathcal{Q}^{(v_{a''})}}$}
\end{picture}

\noindent
The corresponding realization $\delta$ of $(\Lambda,\mathcal{Q})$ has been previously defined on $V_\mathcal{T}$; by the inductive hypothesis, there are a realizations $\delta_u\colon\mathcal{T}_u\rightarrow\mathcal{Q}_a$, $a\in A_c$, $u\in\delta^{-1}(a)$ and it is obvious that we get a $\delta$ that satisfies the conditions in Corollary \ref{CortreepartT}\eqref{CortreepartTa}. Moreover, from Corollary \ref{CortreepartT} we deduce the existence of a tree composition $\varphi$ of $\mathcal{T}$ such that $(\Lambda,\mathcal{Q})$ is the associated tree of partitions and $\varphi$ {\em is associated to the same realization } $\delta$.

By the inductive hypothesis, $Z^{\Lambda^{(u)}}\subseteq M^{\lvert\varphi_u^{-1}\rvert}$ and we may construct the extension $W^\varphi$ as above (cf. \eqref{ResIKW}). We denote by $\widetilde{Z}$ the extension of $Z$ from $N$ to $I$ {\em constructed by means of } \eqref{exttensprod}; this ensures that also $\widetilde{Z}\subseteq W^\varphi$. Finally,
\begin{equation}\label{irredClifford}
Z^\Lambda\coloneqq\Ind_I^{\text{\rm Aut}(\mathcal{T})}\left(\widetilde{Z}\otimes\overline{S}\right)
\end{equation}
is an irreducible representation of $\Aut(\mathcal{T})$ and each of its irreducible representations is of this form. This is the $Z^\Lambda$ in the statement.

Now we show that the correspondence $\Aut(\mathcal{T})\longrightarrow\mathsf{Par}(\mathcal{T})$ is well defined: every irreducible representation $\Ind_I^{\text{\rm Aut}(\mathcal{T})}\left(\widetilde{Z'}\otimes\overline{S'}\right)$ equivalent to \eqref{irredClifford} gives rise to the same tree of partitions, just with a possibly different realization $\delta'$. Indeed, from Clifford theory we deduce that if they are equivalent then $\widetilde{Z'}$ is an extension from $N$ to $I'$ of a representation of the form
$Z'=\bigotimes_{u\in V_\mathcal{T}}Z^{\Xi^{(u)}}$, where $I'$ is the inertia group of $Z'$. It follows that 
 $Z'$ is obtained from $Z$ by means of a $\pi\in\prod_{c\in V_\mathcal{C}}\Sym\left(\tv_\mathcal{T}^{-1}(c)\right)$ which permutes the factors of the tensor product, that is in \eqref{irredsigcbis} and in \eqref{irredsigc} the set $\delta^{-1}(a)$ is replaced by $\pi\left(\delta^{-1}(a)\right)$ and there exist corresponding realizations $\delta'_{\pi(u)}$. Moreover, $S'$ is defined as in \eqref{irredrhoc}, with the same $\lambda^a$'s but replacing the group with $\prod_{c\in V_\mathcal{C}}\prod_{a\in A_c}\Sym\left(\pi\left(\delta^{-1}(a)\right)\right)$, and $\overline{S'}$ is its inflation to $I'$. By the inductive hypothesis and \eqref{Rtauomeg2}, there exist isomorphisms $g_u\colon\mathcal{T}_u\rightarrow\mathcal{T}_{\pi(u)}$ such that $\delta_u=\delta'_{\pi(u)}\circ g_u$; by setting $g\rvert_{V_\mathcal{T}}\coloneqq \pi$ and $\delta'\rvert_{V_\mathcal{T}}\coloneqq \delta\circ \pi^{-1}$ we complete the definition of $g$ and $\delta'$, so that 
\[
\left.\left(\delta'\circ g\right)\right\rvert_{V_\mathcal{T}}=\delta\circ \pi^{-1}\circ\pi=\delta=\delta\rvert_{V_\mathcal{T}}
\]
and finally $(\delta'\circ g)_u=\delta'_{\pi(u)}\circ g_u=\delta_u$. That is, $\delta=\delta'\circ g$.

The map $\Aut(\mathcal{T})\longrightarrow\mathsf{Par}(\mathcal{T})$ is clearly a bijection: given $(\Lambda,\mathcal{Q})$, with a realization $\delta$, we may easily reverse the construction above getting the unique irreducible representation $Z^\Lambda$ whose image is $(\Lambda,\mathcal{Q})$. Moreover, from $\widetilde{Z}\subseteq W^\varphi$ (with multiplicity one) and $\overline{S}\subseteq \overline{\bigotimes}_{a\in V_\mathcal{Q}}M^{\lambda^a}$ (with multiplicity one) it follows that $Z^\Lambda\subseteq M^{\lvert\varphi^{-1}\rvert}$, with multiplicity one; cf. \eqref{MvarphiCli} and \eqref{irredClifford}.
\end{proof}

Note that $Z^\lambda$ is effectively a subspace of $M^{\lvert\varphi^{-1}\rvert}$, it is not only isomorphic to a subpace of the permutation representation: we say that $Z^\lambda$ {\em is constructed by means of the realization of } $\varphi$. By Theoren \ref{Theoorbits}\eqref{Theoorbitsb}, this notion does not depend on $\varphi$ but only on its type $\lvert\varphi^{-1}\rvert$.

\section{Common refinements of tree compositions}\label{Secrefinements}

If $V=\bigsqcup_{i=1}^nA_i=\bigsqcup_{j=1}^mB_j$ are two compositions of $V$ their intersection composition is given by $V=\bigsqcup_{A_i\cap B_j\neq \emptyset}(A_i\cap B_j)$ and every refinement of them is also a refinement of their intersection composition. We describe these elementary facts by means of surjective maps as a reference for the case of tree compositions.
\begin{example}\label{excomrefh1}
{\rm Suppose that $V,A,X,Y$ are sets and that we have the following commutative diagram of surjective maps:
\[
\xymatrix@R=15pt{
&X\\
V\ar[r]^{\delta}\ar[ur]^{\varphi}\ar[dr]_{\psi}&A\ar[u]_{\vartheta}\ar[d]^{\tau}\\
&Y
}
\]
Then
\begin{equation}\label{starp18}
V=\underbrace{\bigsqcup_{a\in A}\delta^{-1}(a)}_{(*)}=\underbrace{\bigsqcup_{x\in X}\varphi^{-1}(x)}_{(**)}=\underbrace{\bigsqcup_{y\in Y}\psi^{-1}(y)}_{(***)}
\end{equation} 
are compositions of $V$; cf. Example \ref{exampleh1}. Moreover, for all $x\in X, y\in Y$,
\begin{equation}\label{quadp19}
\varphi^{-1}(x)\cap\psi^{-1}(y)=\delta^{-1}(\theta^{-1}(x))\cap\delta^{-1}(\tau^{-1}(y))
=\bigsqcup_{a\in \theta^{-1}(x)\cap\tau^{-1}(y)}\delta^{-1}(a).
\end{equation} 
We set $A'\coloneqq\{(x,y)\in X\times Y\colon\varphi^{-1}(x)\cap\psi^{-1}(y)\neq\emptyset\}$
and we define the map $\rho\colon V\rightarrow A'$ by setting, for all $v\in V$,
\begin{equation}\label{stellap18varep}
\rho(v)=(x,y)\quad\text{ if }\quad v\in\varphi^{-1}(x)\cap\psi^{-1}(y),
\end{equation}
and also $\vartheta'\colon A'\rightarrow X$ by setting, for all $(x,y)\in A'$, $\vartheta'(x,y)=x$, $\tau'\colon A'\rightarrow Y$ by setting, $\tau'(x,y)=y$, and finally $\gamma\colon A\rightarrow A'$ by setting $\gamma(a)=(x,y)$ if $a\in\theta^{-1}(x)\cap\tau^{-1}(y)$. 
Then
\begin{equation}\label{intpart}
V=\bigsqcup_{(x,y)\in A'}\left[\varphi^{-1}(x)\cap\psi^{-1}(y)\right]\equiv \bigsqcup_{(x,y)\in A'}\rho^{-1}(x,y)
\end{equation}
is the {\em intersection composition} of $(**)$ and $(***)$ in \eqref{starp18}. From \eqref{quadp19} we get:
$V=\bigsqcup_{(x,y)\in A'}\;\bigsqcup_{a\in \gamma^{-1}(x,y)}\delta^{-1}(a)$,
which is a refinement of \eqref{intpart} and coincides with $(*)$ in \eqref{starp18}, because $A=\bigsqcup_{(x,y)\in A'}\gamma^{-1}(x,y)$.
We have illustrated the obvious fact that every common refinement of $(**)$ and $(***)$ in \eqref{starp18} is also a refinement of \eqref{intpart}.
}
\end{example}

\begin{definition}\label{defcommref}
{\rm 
\begin{enumerate}
\item
Given a tree composition $\varphi\in\mathsf{C}(\mathcal{T},\mathcal{P})$
a {\em refinement} of $\varphi$ is another tree composition $\rho\in\mathsf{C}(\mathcal{P},\mathcal{R})$ with a tree composition $\vartheta\in\mathsf{C}(\mathcal{R},\mathcal{P})$ such that the following diagram is commutative:
\[\xymatrix@R=5pt{
\mathcal{T}\ar[dr]^{\varphi}\ar[dd]_{\rho} & \\
&\mathcal{P}\\
\mathcal{R}\ar[ur]_{\vartheta}&}
\]
We also say that $\varphi$ is {\em coarser} than $\rho$.

\item
Given two tree compositions $\varphi\in\mathsf{C}(\mathcal{T},\mathcal{P})$ and $\psi\in\mathsf{C}(\mathcal{T},\mathcal{M})$, of the same tree $\mathcal{T}$, a {\em common refinement} of the couple $(\varphi,\psi)$ is a tree composition $\rho\in\mathsf{C}(\mathcal{T},\mathcal{R})$ together with two tree compositions $\vartheta\in\mathsf{C}(\mathcal{R},\mathcal{P})$ and $\tau\in\mathsf{C}(\mathcal{R},\mathcal{M})$ such that the following diagram is commutative:
\begin{equation}\label{diagdefcommref}
\xymatrix@R=15pt{
&\mathcal{P}\\
\mathcal{T}\ar[r]^{\rho}\ar[ur]^{\varphi}\ar[dr]_{\psi}&\mathcal{R}\ar[u]_{\vartheta}\ar[d]^{\tau}\\
&\mathcal{M}
}
\end{equation}
\end{enumerate}
}
\end{definition}

A common refinement is {\em trivial} when $\mathcal{R}\equiv\mathcal{T}$ and $\rho\in\Aut(\mathcal{T})$; in particular, common refinements always exist. Given two couples of tree compositions
\begin{equation}\label{couplestreecomp}\xymatrix@R=.5pt{
&\mathcal{P}&&&&\mathcal{P}\\
\mathcal{T}\ar[ur]^{\varphi}\ar[dr]_{\psi}&&&&\mathcal{T}'\ar[ur]^{\varphi'}\ar[dr]_{\psi'}&&\\
&\mathcal{M}&&&&\mathcal{M}
}
\end{equation} 
the common refinements $\rho\in\mathsf{C}(\mathcal{T},\mathcal{R})$ (of $(\varphi,\psi)$) and $\rho'\in\mathsf{C}(\mathcal{T}',\mathcal{R}')$ (of $(\varphi',\psi')$) are {\em equivalent} if there exist tree isomorphisms
$\zeta\colon\mathcal{T}\rightarrow\mathcal{T}'$ and $\gamma\colon\mathcal{R}\rightarrow\mathcal{R}'$ such that the following diagram is commutative:
\begin{equation}\label{diagepsilon2}
\xymatrix@R=5pt{
&&\mathcal{P}&&\\
&&&&\\
\mathcal{T}\ar[dd]_{\zeta}\ar[uurr]^{\varphi}\ar[rrrr]_{\rho}\ar[ddddrr]^(.8){\psi}&&&&\mathcal{R}\ar[dd]^{\gamma}\ar[uull]_{\vartheta}\ar[ddddll]_(.8){\tau}\\
&&&&\\
\mathcal{T}'\ar[ddrr]_{\psi'}\ar[rrrr]^{\rho'}\ar[uuuurr]_(.8){\varphi'}&&&&\mathcal{R}'\ar[ddll]^{\tau'}\ar[uuuull]^(.8){\vartheta'}\\
&&&&\\
&&\mathcal{M}&&
}
\end{equation}
If $\mathcal{R}=\mathcal{R}'$, $\vartheta=\vartheta'$, $\tau=\tau'$ and $\gamma$ is the identity of $\Aut(\mathcal{R})$ we say that they are {\em strictly equivalent}. Clearly, if $\rho$ and $\rho'$ are equivalent as in \eqref{diagepsilon2} then $\gamma\circ\rho$ and $\rho'$ are strictly equivalent.

The couples $(\varphi,\psi)$ and $(\varphi',\psi')$ in \eqref{couplestreecomp} are 
{\em strictly equivalent} if there exists an isomorphism $\zeta\colon\mathcal{T}\rightarrow\mathcal{T}'$ such that the following diagram is commutative:
\begin{equation}\label{defstreqphipsi}\xymatrix@R=15pt{
&\mathcal{P}\\
\mathcal{T}\ar[r]^{\zeta}\ar[ur]^{\varphi}\ar[dr]_{\psi}&\mathcal{T}'\ar[u]_{\varphi'}\ar[d]^{\psi'}\\
&\mathcal{M}
}
\end{equation}
\begin{proposition}\label{Propos4p15}
The couples in \eqref{couplestreecomp} are strictly equivalent 
if and only if every common refinement of one couple is equivalent to a common refinement of the other couple, as in \eqref{diagepsilon2}. Moreover, if this is the case, then we may take the same $\zeta$ in \eqref{defstreqphipsi} and in \eqref{diagepsilon2}.
\end{proposition}
\begin{proof}
For the if part, just note that the diagram \eqref{defstreqphipsi} is part of \eqref{diagepsilon2}. For the only if part, assume the existence of $\zeta$ and the commutativity of \eqref{defstreqphipsi}. Take $\rho\in\mathsf{C}(\mathcal{T},\mathcal{R})$, a common refinement of $\varphi$ and $\psi$ as in \eqref{diagdefcommref}. Define $\rho'\colon\mathcal{T}'\rightarrow\mathcal{R}$ by setting $\rho'\coloneqq\rho\circ \zeta^{-1}$, take $\mathcal{R}'=\mathcal{R}$, $\gamma=$ the identity and set $\vartheta'\coloneqq\vartheta$, $\tau'\coloneqq\tau$. Then
\[
\vartheta'\circ\rho'=\vartheta\circ\rho\circ\zeta^{-1}=\varphi\circ \zeta^{-1}=\varphi'\qquad\text{ and }\qquad\tau'\circ\rho'=\tau\circ\rho\circ\zeta^{-1}=\psi\circ\zeta^{-1}=\psi'
\]
and we have two equivalent common refinements.
\end{proof}

\begin{remark}\label{remcommref}
{\rm
Suppose that we have a common refinement as in \eqref{diagdefcommref} and that $w,v\in\mathcal{T}$ are distinct vertices such that 
$\rho(v)=\rho(w)\eqqcolon a$. Then we have the following strict equivalence of common refinements:
\[
\xymatrix@R=5pt{
&&\mathcal{P}_{\vartheta(a)}&\\
\mathcal{T}_v\ar[dd]_{\zeta}\ar[urr]^{\varphi_v}\ar[drrr]^(.7){\rho_v}\ar[dddrr]^(.8){\psi_v}&&&\\
&&&\mathcal{R}_a\ar[uul]_{\vartheta_a}\ar[ddl]^{\tau_a}\\
\mathcal{T}_w\ar[drr]_{\psi_w}\ar[urrr]_(.7){\rho_w}\ar[uuurr]_(.8){\varphi_w}&&&\\
&&\mathcal{M}_{\tau(a)}&
}
\]
Indeed, we can use the isomorphism $\zeta\colon\mathcal{T}_v\rightarrow\mathcal{T}_w$ given by the tree composition $\rho$ (cf. the regularity condition in Lemma \ref{lemmaregcond}). In particular, both the couples $(\varphi_v,\psi_v)$ and $(\varphi_w,\psi_w)$ and the common refinements $\rho_v$ and $\rho_w$ are strictly equivalent.
}
\end{remark}

\begin{remark}\label{01Remark}{\rm
In the notation of Example \ref{excomrefh1}, suppose that $\varphi'\colon V\rightarrow X$, $\psi'\colon V\rightarrow Y$ is another couple of surjective maps and set 
$A''\coloneqq\{(x,y)\in X\times Y\colon (\varphi')^{-1}(x)\cap(\psi')^{-1}(y)\neq\emptyset\}$.
Then $(\varphi,\psi)$ and $(\varphi',\psi')$ are strictly equivalent (cf. \eqref{defstreqphipsi}) if and only if 
\begin{equation}\label{intmatV}
\left\lvert\varphi^{-1}(x)\cap\psi^{-1}(y)\right\rvert=\left\lvert(\varphi')^{-1}(x)\cap(\psi')^{-1}(y)\right\rvert,\quad \text{for all }(x,y)\in X\times Y. 
\end{equation}
Indeed, if this is the case, we can define $\pi\in\Sym(V)$ in such a way that 
\begin{equation}\label{intmatVbis}
\pi\left[\varphi^{-1}(x)\cap\psi^{-1}(y)\right]=(\varphi')^{-1}(x)\cap(\psi')^{-1}(y),
\end{equation}
for all $(x,y)\in A'\equiv A''$. Moreover, the common refinements of $(\varphi,\psi)$ and $(\varphi',\psi')$ are strictly equivalent: the latter is defined by setting
$\delta(v)=(x,y)$, for all $v\in (\varphi')^{-1}(x)\cap(\psi')^{-1}(y)$ and the following diagram is commutative:  
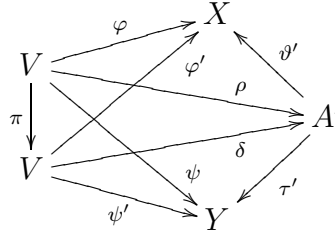
\[
\xymatrix@R=5pt{
&&X&\\
V\ar[dd]_{\pi}\ar[drrr]^(.7){\rho}\ar[urr]^{\varphi}\ar[dddrr]^(.8){\psi}&&&\\
&&&A'\ar[uul]_{\vartheta'}\ar[ddl]^{\tau'}\\
V\ar[urrr]_(.7){\delta}\ar[drr]_{\psi'}\ar[uuurr]_(.8){\varphi'}&&&\\
&&Y&
}
\]
}
\end{remark}

\begin{theorem}\label{Propcoarsercommref}
Given two tree compositions $\varphi\in\mathsf{C}(\mathcal{T},\mathcal{P})$ and $\psi\in\mathsf{C}(\mathcal{T},\mathcal{M})$ there exists a common refinement as in \eqref{diagdefcommref} such that the following {\em coarseness condition} is verified: if $a,b\in\mathcal{R}$ are distinct siblings, $\vartheta(a)=\vartheta(b)\eqqcolon x$, $\tau(a)=\tau(b)\eqqcolon y$, $u,v\in\mathcal{T}$ and
\begin{equation}\label{squarep17}
\rho(u)=a,\qquad\qquad\qquad\rho(v)=b,
\end{equation}
then the common refinements 
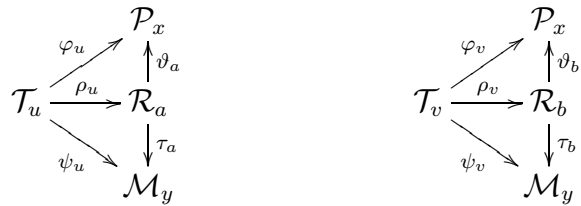
\[
\xymatrix@R=15pt{
&\mathcal{P}_x&&&&\mathcal{P}_x\\
\mathcal{T}_u\ar[r]^{\rho_u}\ar[ur]^{\varphi_u}\ar[dr]_{\psi_u}&\mathcal{R}_a\ar[u]_{\vartheta_a}\ar[d]^{\tau_a}&&&\mathcal{T}_v\ar[r]^{\rho_v}\ar[ur]^{\varphi_v}\ar[dr]_{\psi_v}&\mathcal{R}_b\ar[u]_{\vartheta_b}\ar[d]^{\tau_b}&\\
&\mathcal{M}_y&&&&\mathcal{M}_y
}
\]
are {\em not} equivalent (cf. \eqref{diagepsilon2} and Remark \ref{remcommref}); if $a,b$ are leaves this must be interpreted in the following way: it is impossible that both $\vartheta(a)=\vartheta(b)$ and $\tau(a)=\tau(b)$. Moreover, if 
\begin{equation}\label{arbcommref}
\xymatrix@R=15pt{
&\mathcal{P}\\
\mathcal{T}\ar[r]^{\delta}\ar[ur]^{\varphi}\ar[dr]_{\psi}&\mathcal{U}\ar[u]_{\eta}\ar[d]^{\omega}\\
&\mathcal{M}
}
\end{equation}
is an arbitrary common refinement of $(\varphi,\psi)$ then there exists a tree composition $\gamma\in\mathsf{C}(\mathcal{U},\mathcal{R})$ such that the following diagram is commutative:
\begin{equation}\label{losangap20}\xymatrix@R=30pt{
&&&&\mathcal{P}\\
\mathcal{T}\ar@/_1pc/[rrrr]_(.55){\rho}\ar[rr]^(.7){\delta}\ar[urrrr]^{\varphi}\ar[drrrr]_{\psi}&&\mathcal{U}\ar[rr]^{\gamma}\ar[urr]_(.7){\eta}\ar[drr]^(.7){\omega}&&\mathcal{R}\ar[u]_{\vartheta}\ar[d]^{\tau}\\
&&&&\mathcal{M}
} 
\end{equation}
In particular, the common refinement $\mathcal{R}$ satisfying the coarseness condition is unique up to equivalences (cf. \eqref{diagepsilon2}).
\end{theorem}
\begin{proof}
By induction on $H(\mathcal{T})$. If $H(\mathcal{T})=1$ we are in the setting of Example \ref{exampleh1} and Example \ref{excomrefh1}: we have $\mathcal{T}\equiv V$, $\mathcal{P}\equiv X$, $\mathcal{M}\equiv Y$ and we set $\mathcal{R}\coloneqq A'$, $\vartheta\coloneqq\vartheta'$, $\tau\coloneqq\tau'$, $\rho$ as in \eqref{stellap18varep} and \eqref{intpart}. 

Now suppose $H(\mathcal{T})\geq 2$. First of all, arguing for $W_\mathcal{T}$ as in the case $H(\mathcal{T})=1$, we get $W_\mathcal{R}$, $\theta\rvert_{W_\mathcal{R}}$, $\tau\rvert_{W_\mathcal{R}}$ and $\rho\rvert_{W_\mathcal{T}}$. Then, from the inductive hypothesis, we deduce that for each $v\in V_\mathcal{T}\setminus W_\mathcal{T}$ there exists a unique common refinement
$\rho'_v\colon\mathcal{T}_v\rightarrow\mathcal{R}_{\rho'(v)}$, with associated $\vartheta'_{\rho'(v)}\colon\mathcal{R}_{\rho'(v)}\rightarrow\mathcal{P}_{\varphi(v)}$ and $\tau'_{\rho'(v)}\colon\mathcal{R}_{\rho'(v)}\rightarrow\mathcal{M}_{\psi(v)}$,
that verifies the coarseness condition in the statement. Consider the following equivalence relation (a refinement of \eqref{intpart}): for  $u,v\in V_\mathcal{T}\setminus W_\mathcal{T}$ we say that $u\sim v$ if $\varphi(u)=\varphi(v)\eqqcolon x$, $\psi(u)=\psi(v)\eqqcolon y$ and the common refinements $\rho'_u$ and $\rho'_v$ are equivalent, as in the following commutative diagram (cf. \eqref{diagepsilon2}):

\begin{equation}\label{stepcpmmref}
\xymatrix@R=8pt{
&&&\mathcal{P}_x&&&\\
&&&&&&\\
\mathcal{T}_u\ar[dd]_{\zeta_{v,u}}\ar[uurrr]^{\varphi_u}\ar[rrrrrr]_{\rho'_u}\ar[ddddrrr]^(.8){\psi_u}&&&&&&\mathcal{R}_{\rho'(u)}\ar[dd]^{\gamma_{\rho'(u),\rho'(v)}}\ar[uulll]_{\vartheta'_{\rho'(u)}}\ar[ddddlll]_(.8){\tau'_{\rho'(u)}}\\
&&&&&&\\
\mathcal{T}_v\ar[ddrrr]_{\psi_v}\ar[rrrrrr]^{\rho'_v}\ar[uuuurrr]_(.8){\varphi_v}&&&&&&\mathcal{R}_{\rho'(v)}\ar[ddlll]^{\tau'_{\rho'(v)}}\ar[uuuulll]^(.8){\vartheta'_{\rho'(v)}}\\
&&&&&&\\
&&&\mathcal{M}_y&&&
}
\end{equation}
Then we complete the definition of the first level of $\mathcal{R}$ by setting $V_\mathcal{R}\setminus W_\mathcal{R}\coloneqq \left(V_\mathcal{T}\setminus W_\mathcal{T}\right)/\sim$. For each $a\in V_\mathcal{R}\setminus W_\mathcal{R}$, we choose a representative $u_a$ in the equivalence class corresponding to $a$ and we set $\mathcal{R}_a\coloneqq\mathcal{R}_{u_a}$. After that, for all $v$ in the equivalence class corresponding to $a$ we set $\rho(v)=a$, $\vartheta(v)=\varphi(v), \tau(v)=\psi(v)$ and then
\[
\begin{split}
&\rho_v\coloneqq\rho'_{u_a}\circ\zeta_{u_a,v}\equiv \gamma_{\rho'(u_a),\rho'(v)}\circ\rho'_v,\qquad(\text{by }\eqref{stepcpmmref})\\
&\theta_a\coloneqq\theta'_{\rho'(u_a)}\qquad\text{and}
\qquad\tau_a\coloneqq\tau'_{\rho'(u_a)}.
\end{split}
\]
From the commutativity of \eqref{stepcpmmref} it follows that 
\[
\theta_{\rho(v)}\circ\rho_v=\theta'_{\rho'(u_a)}\circ \gamma_{\rho'(u_a),\rho'(v)}\circ\rho'_v=\theta'_{\rho'(v)}\circ\rho'_v=\varphi_v
\]
and similarly $\tau_{\rho(v)}\circ\rho_v=\psi_v$.
Then we can use \eqref{firstldell} to define $\rho$, $\vartheta$ and $\tau$ and we get a common refinement.

The coarseness condition is clearly verified at the first level while, at lower levels, it follows from the inductive hypothesis when $u,v$ in \eqref{squarep17} are distinct siblings. Otherwise, we may take $v'$ sibling of $u$ such that $\rho(v)=\rho(v')=b$ and we can invoke Remark \ref{remcommref} to deduce that $\rho_v$ and $\rho_{v'}$ are strictly equivalent; but $\rho_u$ and $\rho_{v'}$ are {\em not} equivalent by the inductive hypothesis.

Finally, suppose that we have an arbitrary common refinement as in \eqref{arbcommref}. The construction in Example \ref{excomrefh1} applied to $W_\mathcal{T}$ yields the diagram \eqref{losangap20} restricted to the leaves at the first level. Then for each $u\in V_\mathcal{T}\setminus W_\mathcal{T}$ we may apply the inductive hypothesis to the common refinement $\delta_u\in\mathsf{C}\left(\mathcal{T}_u,\mathcal{U}_c\right)$ deducing that there exists $\gamma_c\in\mathsf{C}\left(\mathcal{U}_c,\mathcal{R}_{\rho(u)}\right)$ such that the following diagram, for the moment without $\mathcal{T}_v$ and its arrows, is commutative:

\begin{equation}\label{bigdiagram2}\xymatrix@R=15pt@C=50pt{
&&\mathcal{P}_{\varphi(u)}\\
\mathcal{T}_v\ar[dd]_{\zeta}\ar@/_-1.2pc/[urr]^{\varphi_v}\ar@/_1.2pc/[dddrr]_(.7){\psi_v}\ar[dr]^(.25){\delta_v}&&\\
&\mathcal{U}_c\ar[r]^{\gamma_c}\ar[uur]_(.5){\eta_c}\ar[ddr]^(.65){\omega_c}&\mathcal{R}_{\rho(u)}\ar[uu]_{\vartheta_a}\ar[dd]^{\tau_a}\\
\mathcal{T}_u\ar@/_1.2pc/[urr]_(.25){\rho_u}\ar[ur]_(.45){\delta_u}\ar@/^1.2pc/[uuurr]^(.7){\varphi_u}\ar@/_1.2pc/[drr]_{\psi_u}&&\\
&&\mathcal{M}_{\psi(u)}
}
\end{equation}
At the first level, we may construct $\gamma\colon V_\mathcal{U}\setminus W_\mathcal{U}\rightarrow V_\mathcal{R}\setminus W_\mathcal{R}$ (cf. \eqref{losangap20}) arguing as in Example \ref{excomrefh1}, taking into account Remark \ref{remcommref} and the equivalence relation $\sim$ introduced in the preceding part of the proof. Now we give the details. If $v\in V_\mathcal{T}\setminus W_\mathcal{T}$, $v\neq u$ and also $\delta(v)=c$ then the regularity condition in Definition \ref{deftreecomp} applied to $\delta$ yields an isomorphism $\zeta\colon\mathcal{T}_v\rightarrow\mathcal{T}_u$ such that $\delta_v=\delta_u\circ\zeta$. We have also $\varphi_v=\eta_c\circ\delta_v$ and $\psi_v=\omega_c\circ\delta_v$, by \eqref{diagdefcommref}. Then the whole diagram \eqref{bigdiagram2} is commutative: for instance 
\[
\varphi_v=\eta_c\circ\delta_v=\eta_c\circ\delta_u\circ\zeta=\varphi_u\circ\zeta.
\]
It follows that $\rho_u\circ\zeta\colon\mathcal{T}_v\rightarrow\mathcal{R}_{\rho(u)}$ yields a common refinement of $\varphi_v$ and $\psi_v$, which satisfies the coarseness condition because it is strictly equivalent to $\rho_u$. Therefore, the construction of $\mathcal{R}$ (that is the definition of $\sim$) forces $\rho(v)=\rho(u)$. In conclusion, we may set $\gamma(c)=\rho(u)$ and this ends the construction of $\gamma$ at the first level of $\mathcal{U}$ and also the proof.
\end{proof}

\begin{definition}\label{defcoarcomref}
{\rm
The tree composition $\rho\in\mathsf{C}(\mathcal{T},\mathcal{R})$ in Theorem \ref{Propcoarsercommref} will be called the {\em coarsest common refinement} of the couple $(\varphi,\psi)\in\mathsf{C}(\mathcal{T},\mathcal{P})\times \mathsf{C}(\mathcal{T},\mathcal{M})$. We say that the coarsest common refinement is {\em trivial} if $\rho$ is an isomorphism.}
\end{definition}

Clearly, for every internal $v\in\mathcal{T}$ the tree composition $\rho_v\colon\mathcal{T}_v\rightarrow\mathcal{R}_{\rho(v)}$, with the maps $\vartheta_{\rho(v)}$ and $\tau_{\rho(v)}$, is the coarsest common refinement of $(\varphi_v,\psi_v)$.

\begin{proposition}\label{Propos4p15bis}
Two couples of tree compositions as in \eqref{couplestreecomp} are strictly equivalent (cf. \eqref{defstreqphipsi})
if and only if their coarsest common refinements are equivalent, as in \eqref{diagepsilon2}. Moreover, if this is the case, then we may take the same isomorphism $\zeta$ in both equivalences.
\end{proposition}
\begin{proof}
Specialize Proposition \ref{Propos4p15} to the case of the coarsest common refinements.
\end{proof}

\begin{lemma}\label{Propcoarreftriv}
The coarsest common refinement of tree compositions $\varphi\in \mathsf{C}(\mathcal{T},\mathcal{P})$ and $\psi\in\mathsf{C}(\mathcal{T},\mathcal{M})$ is trivial if and only if the following condition is verified: for all pair of distinct siblings $u,v\in\mathcal{T}$,
$\varphi(u)=\varphi(v)$ and $\psi(u)=\psi(v)$ imply that
the couples $(\varphi_u,\psi_u)$ and $(\varphi_v,\psi_v)$ are {\em not} strictly equivalent. If $u,v$ are leaves then this condition must interpreted in the following way: it is impossible that both $\varphi(u)=\varphi(v)$ and $\psi(u)=\psi(v)$.
\end{lemma}
\begin{proof}
If the condition in the statement is verified then two distinct siblings cannot have the same image under a common refinement $\rho$ (cf. Remark \ref{remcommref}) and therefore $\rho$ must be injective. Conversely, if $\rho$ is injective and $u,v$ are distinct siblings as in the statement then $\rho(u)\neq\rho(v)$ so that, by Theorem \ref{Propcoarsercommref}, the couples $(\varphi_u,\psi_u)$ and $(\varphi_v,\psi_v)$ are not strictly equivalent. 
\end{proof}

\begin{lemma}\label{Lemmatrivaut}
Suppose that $\varphi\in\mathsf{C}(\mathcal{T},\mathcal{P})$ and $\psi\in\mathsf{C}(\mathcal{T},\mathcal{Q})$
are tree compositions of $\mathcal{T}$. Then, with the notation in \eqref{defgvaract}, the group of all $g\in\Aut(\mathcal{T})$ such that $g\varphi=\varphi$ and $g\psi=\psi$ is trivial if and only if their coarsest common refinement is trivial.
\end{lemma}
\begin{proof}
Suppose that the coarsest common refinement is trivial and that $g\varphi=\varphi, g\psi=\psi$. If $H(\mathcal{T})=1$, from the characterization of $\mathcal{R}\equiv A'$ in Example \ref{excomrefh1} we deduce that $g\equiv\pi$ is trivial: $\varphi'=\varphi$, $\psi'=\psi$ and, by \eqref{stellap18varep}, $\lvert \varphi^{-1}(x)\cap\psi^{-1}(y)\rvert\leq 1$ for all $x\in X, y\in Y$, so that \eqref{intmatVbis} forces $\pi=$ the identity. If $H(\mathcal{T})\geq 2$ then from Lemma \ref{Propcoarreftriv} and the definition of $\sim$ in the proof of Theorem \ref{Propcoarsercommref} it follows that $g$ is trivial at the first level and then we can apply induction  to each $\mathcal{T}_u, u\in V_\mathcal{T}\setminus W_\mathcal{T}$, to deduce that $g_u$ is also trivial.

If the coarsest common refinement $\rho$ is not injective then there exists an internal $v\in\mathcal{T}$ such that 
$\rho|_{\text{\rm Ch}(v)}\colon\text{\rm Ch}(v)\rightarrow\text{\rm Ch}(\rho(v))$]
is not injective, so that there exist distinct $u,w\in\text{\rm Ch}(v)$ such that $\rho(u)=\rho(w)\eqqcolon a$. Then the tree compositions $\rho_u$ and $\rho_w$ are strictly equivalent and therefore we can define a nontrivial
 $g\in\Aut(\mathcal{T})$ by setting
\begin{equation}\label{defgsign}
\begin{split}
&g \text{ switches }\mathcal{T}_u\text{ and }\mathcal{T}_w\text{ by means of the strict equivalence;}\\
&g\text{ fixes all the vertices outside }\mathcal{T}_u\text{ and }\mathcal{T}_w.
\end{split}
\end{equation}
\noindent
It follows that $\rho=\rho\circ g$ so that, by \eqref{diagdefcommref}, $g\varphi=\vartheta\circ\rho\circ g^{-1}=\vartheta\circ \rho =\varphi$; similarly $g\psi=\psi$. 
\end{proof}

\begin{definition}\label{def01inters}{\rm
We say that $\varphi\in\mathsf{C}(\mathcal{T},\mathcal{P})$ and $\psi\in\mathsf{C}(\mathcal{T},\mathcal{M})$ have $0-1$ {\em intersection} if the following condition is verified: for each $v\in \mathcal{T}$ which is internal we have
\begin{equation}\label{01cond}
\left\lvert \varphi^{-1}(x)\cap\psi^{-1}(y)\right\rvert\leq 1,
\end{equation}
for all $x\in\text{\rm Ch}_\mathcal{P}(\varphi(v)), y\in\text{\rm Ch}_\mathcal{M}(\psi(v))$.
}
\end{definition}

\begin{remark}\label{Rem01cond}
{\rm
The apparently weaker condition $\left\lvert\text{\rm Ch}(v)\cap \varphi^{-1}(x)\cap\psi^{-1}(y)\right\rvert\leq 1$ is equivalent to \eqref{01cond}. Indeed, if $w\in\text{\rm Ch}(v)\cap\varphi^{-1}(x)\cap\psi^{-1}(y)$ then we may consider the paths 
$\{\emptyset\}, v_1, v_2,\dotsc, v_{k-1}=v,v_k=w$ in $\mathcal{T},\;$
$\{\emptyset\}, x_1, x_2,\dotsc, x_k=x$ in $\mathcal{P}\;$ and $\{\emptyset\}, y_1, y_2,\dotsc, y_k=y\text{ in }\mathcal{M}$.
Therefore $x_j=\varphi(v_j)$ and $y_j=\psi(v_j)$, $j=1,2,\dotsc,k-1$ and we can apply the (apparently) weaker condition to $x_1,y_1$, then to $x_2,y_2$, and so on, getting  that $\varphi^{-1}(x_j)\cap\psi^{-1}(y_j)=\{v_j\}$, $j=1,2,\dotsc,k-1,k$.
}
\end{remark}

\begin{lemma}\label{Lemma01cond}
If $\varphi$ and $\psi$ have $0-1$ intersection 
 then their coarsest common refinement is trivial.
\end{lemma}
\begin{proof}
If \eqref{01cond} is verified, then the equivalence relation $\sim$ in the proof of Theorem \ref{Propcoarsercommref} is trivial (each equivalence class is a singleton). By induction, also the common refinements $\rho'_v\colon\mathcal{T}_v\rightarrow\mathcal{R}_{\rho'(v)}$ must be trivial.
\end{proof}

The converse of Lemma \ref{Lemma01cond} is not true: just note that, in the proof of Theorem \ref{Propcoarsercommref} for $H(\mathcal{T})\geq2$, the equivalence relation $\sim$ leads to a refinement of \eqref{intpart}. 

\section{The total order and the transpose}\label{Sectransp}

If $(\Lambda,\mathcal{Q})$ and $(\Xi,\mathcal{N})$ are tree of partitions, not necessarily of the same height, we need to compare them by means of a suitable total order which will be written in the form $\Xi<\Lambda$. It will have the following property:
\begin{equation}\label{propertyp30b}
H(\mathcal{Q})>H(\mathcal{N})\Longrightarrow\Lambda>\Xi.
\end{equation}
We will write $\Lambda\equiv\Xi$ to indicate that $(\Lambda,\mathcal{Q})$ and $(\Xi,\mathcal{N})$ are isomorphic.

We define the order by induction on $h$, the maximum height of the trees. If $h=1$ we use the lexicographic order between integer partitions; cf. \cite[1.4.5]{JK} or \cite[p. 222]{Sagan}. We specify it because we need to compare also partitions of different integers.  If $\lambda=(\lambda_1,\lambda_2,\dotsc,\lambda_k)$ and $\xi=(\xi_1,\xi_2,\dotsc,\xi_l)$ are distinct partitions then we write $\lambda>\xi$ if one of the following two cases is verified: (1) there exists and index $i$ such that $\lambda_1=\xi_1, \lambda_2=\xi_2, \dotsc, \lambda_{i-1}=\xi_{i-1}$ and $\lambda_i>\xi_i$;
(2) $k>l$ and $\lambda_i=\xi_i$, $i=1,2,\dotsc,l$.

Assume now that $h\geq 2$ and that we have introduced a total order $\geqq_{h-1}$ between trees of partitions of height $\leq h-1$. First of all, we need to introduce an equivalence relation $\equiv'$ and then a total order $>'$ between first level vertices of trees of partitions of height $\leq h$. Suppose that $a\in V_\mathcal{Q}$, $b\in V_\mathcal{N}$, where $(\Lambda,\mathcal{Q})$ and $(\Xi,\mathcal{N})$ are trees of partitions of height $\leq h$; note that we need to compare also two vertices of the same tree, so that it is possible that $(\Lambda,\mathcal{Q})$ and $(\Xi,\mathcal{N})$ are isomorphic. As usual, $\Lambda_a$, $a\in\mathcal{Q}$ internal, does not have a value on $a$ because it is the root of $\mathcal{Q}_a$; cf. Definition \ref{Deftreepart}. Moreover, as specified above, $\equiv_{h-1}$ will denote isomorphism of two tree of partitions of height $\leq h-1$. This is the equivalence relation: we write that

\begin{equation}\label{comparingblocksc}
a\equiv'b
\end{equation}
when one of the following conditions is satisfied:
\begin{itemize}
\item
$a$ and $b$ are both leaves and $\lambda^a=\lambda^b$; 

\item
$a$ and $b$ are internal,
$\Lambda_a\equiv_{h-1}\Xi_b$ and $\lambda^a=\lambda^b$.
\end{itemize}

In other words, $a\equiv'b$ means that we have an isomorphism between the tree of partitions $\mathcal{Q}_a$ and $\mathcal{N}_a$ where also to the root is associated a partition. 
After that, we define the total order: if $a\not\equiv'b$ we write that 

\begin{equation}\label{comparingblocksb}
a>'b
\end{equation} 
if one of the following conditions is verified:
\begin{enumerate}
\item
$a$ and $b$ are both leaves and $\lambda^a>\lambda^b$; 
\item
$a$ is internal while $b$ is a leaf;
\item
$a$ and $b$ are internal and $\Lambda_a>_{h-1}\Xi_b$;
\item 
$\Lambda_a\equiv_{h-1}\Xi_b$ but $\left\lvert\lambda^a\right\rvert>\left\lvert\lambda^b\right\rvert$;
\item
$\Lambda_a\equiv_{h-1}\Xi_b$, $\left\lvert\lambda^a\right\rvert=\left\lvert\lambda^b\right\rvert$,
 but $\lambda^a>\lambda^b$.
\end{enumerate}

If $a,a'\in V_\mathcal{Q}$, $a\neq a'$ then it is impossible that $\Lambda_a\equiv_{h-1}\Lambda_{a'}$, because, by Definition \ref{Deftreepart}, the tree of partitions $\Lambda_a$ and $\Lambda_{a'}$ are not isomorphic. Therefore we may order the vertices of $V_\mathcal{Q}$ and $V_\mathcal{N}$ accordingly to $<'$:
\[
V_\mathcal{Q}=\{a_1>' a_2>'\dotsb >'a_l\}\quad\text{ and }\quad V_\mathcal{N}=\{b_1>' b_2>'\dotsb >'b_k\}.
\]
Finally, we set $\Lambda>_h\Xi$ 
\begin{equation}\label{deftotord}
\begin{split}
&\bullet\text{ if }\; l>k, \;\text{ and } a_1\equiv'b_1,a_2\equiv'b_2,\dotsc,a_k\equiv'b_k, \\
&\bullet\text{ or if there exists } j \text{ such that } a_1\equiv'b_1,a_2='b_2,\dotsc,a_{j-1}\equiv'b_{j-1}\text{ and }a_j>'b_j.
\end{split}
\end{equation}  
Clearly, also $\geqq_h$ respect the condition in \eqref{propertyp30b}. Note also that if $l=k$ and $a_1\equiv'b_1,a_2\equiv'b_2,\dotsc,a_l\equiv'b_l$ then $\Lambda\equiv \Xi$ so that if $\Lambda\not\equiv \Xi$ then $\Lambda>\Xi$ or $\lambda<\Xi$.

By Theorem \ref{Theoorbits}\eqref{Theoorbitsa}, the order just introduced may be seen as an order between equivalence classes of tree compositions. Therefore, if $\varphi$ and $\psi$ are tree compositions and $(\Lambda,\mathcal{Q}),(\Xi,\mathcal{N})$ are the respective tree of partitions, then we will write 
\begin{equation}\label{Totord}
\varphi\geqq \psi
\end{equation} 
when $\Lambda\geqq \Xi$, noting that the equality in \eqref{Totord} means that $\varphi$ and $\psi$ are equivalent; if this is the case we will write $\varphi\equiv\psi$.

If $\lambda$ is an integer partition we indicate by $\lambda^t$ its conjugate; cf. \cite{CST3, JK}. If $(\Lambda,\mathcal{Q})$ is a tree of partitions then we denote by $\Lambda^t$ the partitional labeling of $\mathcal{Q}$ given by $\Lambda^t\colon a\mapsto\left(\lambda^a\right)^t$. Clearly, also $(\Lambda^t,\mathcal{Q})$ is a tree of partitions; cf. Definition \ref{Deftreepart}.

\begin{theorem}\label{PropositionTransposed}
Suppose that $\varphi\in\mathsf{C}(\mathcal{T},\mathcal{P})$, $\alpha\in\mathsf{C}(\mathcal{P},\mathcal{Q})$ is its quotient and that $(\Lambda,\mathcal{Q})$ is the associated tree of partitions.
There exists a tree composition $\varphi^t\in\mathsf{C}(\mathcal{T},\mathcal{P}^t)$ (a {\em transpose} of $\varphi$) such that:
\begin{enumerate}[\rm (a)]

\item\label{transp1} if $\alpha^t\in\mathsf{C}(\mathcal{P}^t,\mathcal{Q}^t)$ is the quotient of $\varphi^t$ then $\mathcal{Q}^t=\mathcal{Q}$, $\alpha^t\circ\varphi^t=\alpha\circ\varphi$ and the tree of partitions of $\varphi^t$ is $(\Lambda^t,\mathcal{Q})$;

\item\label{transp3}
the tree compositions $\varphi$ and $\varphi^t$ have $0-1$ intersection (cf. Definition \ref{def01inters}). 
\end{enumerate}
We may also construct $\varphi^t$ in such a way that it has the same $\tau$'s of $\varphi$; cf. Definition \ref{Defpseudo}. 
\end{theorem}
\begin{proof}
The proof is by induction on $H(\mathcal{T})$. For $H(\mathcal{T})=1$ this is a standard construction; see the {\em existence} of $m^\lambda$ in \cite[Section 7]{CST3} or of the $0-1$ matrix in \cite[2.1.1]{JK}. We just point out that if $\varphi\colon V_\mathcal{T}\rightarrow V_\mathcal{P}$ determines a composition of type $\lambda$ while $\psi\colon V_\mathcal{T}\rightarrow V_{\mathcal{P}^t}$ determines a composition of type $\lambda^t$ then there exists $\pi\in\Sym\left(V_\mathcal{T}\right)$ such that
\begin{equation}\label{01matrixphvph1}
\left\lvert \varphi^{-1}(x)\cap \left(\psi\circ\pi\right)^{-1}(y) \right\rvert\leq1,\quad\text{ for all }x\in V_\mathcal{P}, y\in V_{\mathcal{P}^t}.
\end{equation}

Suppose now that $H(\mathcal{T})\geq 2$. First of all, note that $\delta\coloneqq\alpha\circ\varphi$ is also a realization of $\left(\Lambda^t,\mathcal{Q}\right)$ as a tree of partitions of $\mathcal{T}$, just because $\left\lvert\left(\lambda^a\right)^t\right\rvert=\left\lvert\lambda^a\right\rvert$, for every $a\in\mathcal{Q}$; cf. Definition \ref{defIrrG}. On the other hand, from Proposition \ref{ProptreepartT} we deduce the existence of a tree composition $\psi\in\mathsf{C}\left(\mathcal{T},\mathcal{P}^t\right)$ whose tree of partitions is $\left(\Lambda^t,\mathcal{Q}\right)$; we denote by $\alpha^t\in\mathsf{C}\left(\mathcal{P}^t,\mathcal{Q}\right)$ the quotient of $\psi$. Then $\delta'\coloneqq \alpha^t\circ\psi$ is another realization of $\left(\Lambda^t,\mathcal{Q}\right)$ so that, by Corollary \ref{CortreepartT}\eqref{CortreepartTb}, there exists $g'\in\Aut(\mathcal{T})$ such that $\delta=\delta'\circ g'$. By replacing $\psi$ with the strictly equivalent tree composition $\psi\circ g'$, we may suppose that indeed $\delta'=\delta$.

Now, by induction, we prove the existence of $g\in\Aut(\mathcal{T})$ such that $\psi\circ g$ and $\varphi$ have $0-1$ intersection and $\delta\circ g=\delta$, that is $\delta$ is also the realization of $\psi\circ g$. For $H(\mathcal{T})=1$ it is just the $\pi$ in \eqref{01matrixphvph1}. Then, for each $a\in V_\mathcal{Q}\setminus W_\mathcal{Q}$, we fix $v_a\in\delta^{-1}(a)$ so that, by the inductive hypothesis, there exists $\widetilde{g}_a\in\Aut\left(\mathcal{T}_{v_a}\right)$ such that $\varphi^t_a\coloneqq\psi_{v_a}\circ \widetilde{g}_a$ has $0-1$ intersection with $\varphi_a$ and $\delta_{v_a}=\delta_{v_a}\circ \widetilde{g}_a$.

As in the case $H(\mathcal{T})=1$, for each $a\in V_\mathcal{Q}\setminus W_\mathcal{Q}$ there exists $\pi\in\Sym\left(V_\mathcal{T}\right)$ such that $\pi\left(\delta^{-1}(a)\right)=\delta^{-1}(a)$ for all $a\in V_\mathcal{Q}$ and $\psi\circ\pi\rvert_{V_\mathcal{T}}$ has $0-1$ intersection with $\varphi\rvert_{V_\mathcal{T}}$. Then, by setting $g\rvert_{V_\mathcal{T}}\coloneqq \pi$ and $\varphi^t\rvert_{V_\mathcal{T}}\coloneqq\psi\circ\pi\rvert_{V_\mathcal{T}}$,  \eqref{transp1} and \eqref{transp3} are verified at the first level and also $\delta\rvert_{V_\mathcal{T}}=\delta\rvert_{V_\mathcal{T}}\circ\pi$. After that, denote by $\tau_{u,v}$ a set of first level $\tau$'s of $\varphi$, by $\tau'_{u,v}$ a set of first level $\tau$'s of $\psi$ and by $\omega^t_{x,y}$ a set of first level $\omega$'s for $\psi$  (cf. Definition \ref{Defpseudo}). For all $a\in V_\mathcal{Q}\setminus W_\mathcal{Q}$, $x\in \left(\alpha^t_a\right)^{-1}(a)$ and $u\in \left(\varphi^t\rvert_{V_\mathcal{T}}\right)^{-1}(x)$, set
\begin{equation}\label{defgforphit}
g_u\coloneqq\tau'_{\pi(u),v_a}\circ \widetilde{g}_a\circ\tau_{v_a,u},
\end{equation}
so that $g_u\colon\mathcal{T}_u\rightarrow\mathcal{T}_{\pi(u)}$ and $g$ is completely defined. For $u\in V_\mathcal{T}\setminus W_\mathcal{T}$ we have
\[
\begin{split}
(\delta\circ g)_u=&\alpha^t_{\psi(\pi(u))}\circ\psi_{\pi(u)}\circ\tau'_{\pi(u),v_a}\circ \widetilde{g}_a\circ\tau_{v_a,u}\\
(\text{by }\eqref{diagTRSuv})\quad=&\alpha^t_{\psi(v_a)}\circ\psi_{v_a}\circ \widetilde{g}_a\circ\tau_{v_a,u}\\
(\delta'=\delta, \delta_{v_a}=\delta_{v_a}\circ \widetilde{g}_a)\quad=&\alpha_{\varphi(v_a)}\circ\varphi_{v_a}\circ\tau_{v_a,u}\\
(\text{by }\eqref{diagTRSuv})\quad=&\delta_u,
\end{split}
\]
so that $\delta=\delta\circ g$. Finally, we set $\varphi^t\coloneqq\psi\circ g$. For $u,w\in V_\mathcal{T}\setminus W_\mathcal{T}$ we have:
\[
\begin{split}
\varphi^t_u\circ\tau_{u,w}=&\psi_{\pi(u)}\circ g_u\circ\tau_{u,w}\\
(\text{by }\eqref{proptau2b}\text{ and }\eqref{defgforphit})\quad=&\psi_{\pi(u)}\circ\tau'_{\pi(u),\pi(w)}\circ g_w\\
(\text{by }\eqref{diagTRSuv})\quad=&\omega^t_{\psi(\pi(u)),\psi(\pi(w))}\circ\psi_{\pi(w)}\circ g_w\\
(\varphi^t=\psi\circ g)\quad=&\omega^t_{\varphi^t(u),\varphi^tw)}\circ\varphi^t_w.
\end{split}
\]
Therefore $\varphi^t$ has the same $\tau$'s of $\varphi$ and the $\omega$'s of $\psi$.

We end by proving the $0-1$ intersection property at lower levels. Suppose that $u\in\mathcal{T}$ is internal, $x\in\text{\rm Ch}_\mathcal{P}\left(\varphi(u)\right)$ and $y\in\text{\rm Ch}_{\mathcal{P}^t}\left(\varphi^t(u)\right)$. Denote by $\widetilde{u}$, $\widetilde{x}$ and $\widetilde{y}$ 
respectively the ancestor of $u$, $\varphi(u)$ and $\varphi^t(u)$ at the first level.
From the $0-1$ property at the first level it follows that $\varphi^{-1}(\widetilde{x})\bigcap\left(\varphi^t\right)^{-1}(\widetilde{y})=\{\widetilde{u}\}$, so that $\varphi^{-1}(\mathcal{P}_{\widetilde{x}})\bigcap\left(\varphi^t\right)^{-1}(\mathcal{P}_{\widetilde{y}})\subseteq\mathcal{T}_{\widetilde{u}}$. Therefore
\[
\begin{split}
\varphi^{-1}(x)\bigcap\left(\varphi^t\right)^{-1}(y)=&\left(\varphi_{\widetilde{u}}\right)^{-1}(x)\bigcap\left(\varphi_{\widetilde{u}}^t\right)^{-1}(y)\\
(\text{by }\eqref{diagTRSuv})\quad=&\left[\tau_{\widetilde{u},v_a}\circ\left(\varphi_{v_a}\right)^{-1}\circ\omega_{\varphi(v_a),\widetilde{x}}\right](x)\bigcap\left[\tau_{\widetilde{u},v_a}\circ\left(\varphi_{v_a}^t\right)^{-1}\circ\omega^t_{\varphi^t(v_a),\widetilde{y}}\right](y)\\
=&\tau_{\widetilde{u},v_a}\left\{\left[\left(\varphi_{v_a}\right)^{-1}\left(\omega_{\varphi(v_a),\widetilde{x}}(x)\right)\right]\bigcap\left[\left(\varphi^t_{v_a}\right)^{-1}\left(\omega^t_{\varphi^t(v_a),\widetilde{y}}(y)\right)\right]\right\}\\
\end{split}
\]
so that
\begin{equation}\label{starp41}
\left\lvert\varphi^{-1}(x)\bigcap\left(\varphi^t\right)^{-1}(y)\right\rvert=\left\lvert\left[\left(\varphi_{v_a}\right)^{-1}\left(\omega_{\varphi(v_a),\widetilde{x}}(x)\right)\right]\bigcap\left[\left(\varphi^t_{v_a}\right)^{-1}\left(\omega^t_{\varphi^t(v_a),\widetilde{y}}(y)\right)\right]\right\rvert\leq 1
\end{equation}
and now we prove the last inequality in \eqref{starp41}: we have
\[
\begin{split}
\varphi_{v_a}\circ\tau_{v_a,\widetilde{u}}(u)=&\omega_{\varphi(v_a),\widetilde{x}}\circ\varphi_{\widetilde{u}}(u)=\omega_{\varphi(v_a),\widetilde{x}}(\varphi(u))\quad\text{ and }\quad x\in\text{\rm Ch}_\mathcal{P}\left(\varphi(u)\right)\\
&\Longrightarrow\quad \omega_{\varphi(v_a),\widetilde{x}}(x)\in\text{\rm Ch}_\mathcal{P}\left[\varphi_{v_a}(\tau_{v_a,\widetilde{u}}(u))\right]
\end{split}
\]
and similarly $\omega^t_{\varphi^t(v_a),\widetilde{y}}(y)\in\text{\rm Ch}_{\mathcal{P}^t}\left[\varphi^t_{v_a}(\tau_{v_a,\widetilde{u}}(u))\right]$,
so that, by the inductive hypothesis, we may apply the $0-1$ property within $\mathcal{T}_{v_a}$.
\end{proof}

\begin{definition}{\rm
A {\em transpose} of $\varphi$ is a tree composition $\varphi^t$ with all the properties in Theorem \ref{PropositionTransposed}.
}
\end{definition}

\begin{remark}\label{Remtransp}{\rm
Note that $\varphi^t$ is not unique nor a tree composition strictly equivalent to $\varphi^t$ is still a transpose; these facts are obvious even in the case $H(\mathcal{T})=1$. On the other hand, if $\psi\in\mathsf{C}(\mathcal{T}',\mathcal{P})$ is strictly equivalent to $\varphi$, say $\psi=\varphi\circ\beta$ where $\beta\colon\mathcal{T}'\rightarrow\mathcal{T}$ is an isomorphism, then $\varphi^t\circ\beta$ is a transpose of $\psi$, because all the properties in Theorem \ref{PropositionTransposed} are clearly preserved by $\beta$; cf. Remark \ref{Remobviousfacts}, Remark \ref{quotstequiv} and Theorem \ref{Theoorbits}. Finally, it is obvious that $\varphi$ is a transpose of $\varphi^t$.
}
\end{remark}

\begin{example}\label{transptvid}{\rm
In the notation of Example \ref{idtreecomp}, $\tv_\mathcal{T}$ is a transpose of $\id_\mathcal{T}$ and vice versa; see also Example \ref{treeparttvid} for the case of a spherically homogeneous tree. 
}
\end{example}

\begin{theorem}\label{TheoGalRy}
Suppose that $\varphi\in\mathsf{C}(\mathcal{T},\mathcal{P})$ and $\psi\in\mathsf{C}(\mathcal{T},\mathcal{M})$ are two  tree compositions of $\mathcal{T}$. If their coarsest common refinement is trivial (cf. Definition \ref{defcoarcomref}) then $\varphi\leqq\psi^t$ (and 
also, by symmetry, $\psi\leqq\varphi^t$) with respect to the total order in \eqref{Totord}.
\end{theorem}
\begin{proof}
The proof is by induction on $h\coloneqq H(\mathcal{T})$, and for every $h$ we will use also induction on $\lvert V_\mathcal{Q}\rvert$, the number of first level vertices of $\mathcal{Q}$.
For $h=1$ the theorem is just a weak form of the easy part of the Gale-Ryser Theorem: we have replaced the majorization order with the lexicographic order; see \cite[1.4.17]{JK}. 

Suppose now that $h\geq2$, assume that $\varphi$ and $\psi^t$ are not equivalent (otherwise there is nothing to prove) and suppose that $\alpha\in\mathsf{C}(\mathcal{P},\mathcal{Q})$ is the quotient of $\varphi$, $\beta\in\mathsf{C}(\mathcal{M},\mathcal{N})$ is the quotient of $\psi$. Let $a_1$ be the greatest element in $V_\mathcal{Q}$ and $b_1$ the greatest element in $V_\mathcal{N}$, with respect to the order \eqref{comparingblocksb}. We denote by $\Lambda\colon\mathcal{Q}\ni a\mapsto\lambda^a$, respectively $\Xi\colon\mathcal{N}\ni b\mapsto\xi^b$, the tree of partitions associated to $\varphi$, respectively to $\psi^t$. We will use the following algorithm, based on   \eqref{comparingblocksc}, \eqref{comparingblocksb}, \eqref{deftotord} and \eqref{Totord}.

\begin{enumerate}
\item  
If there exists $b\in V_\mathcal{N}$, $b\neq b_1$, such that $(\alpha\circ\varphi)^{-1}(a_1)\cap(\beta\circ\psi)^{-1}(b)\neq\emptyset$ then we may choose a vertex $v$ in this nonempty set. By Lemma \ref{Propcoarreftriv}, the coarsest common refinement of $\varphi_v$ and $\psi_v$ is trivial and therefore, by the inductive hypothesis, we have $\varphi_v\leqq\psi_v^t$, so that $\Lambda_{a_1}\leq\Xi_b$. Since $\psi_v^t<\psi_w^t$ for every $w\in(\psi^t)^{-1}(b_1)$ then $\Xi_b<\Xi_{b_1}$, and we conclude that $a_1<'b_1$. From \eqref{deftotord} and we deduce that $\varphi<\psi^t$.

\item
If $(\alpha\circ\varphi)^{-1}(a_1)\subseteq(\beta^t\circ\psi^t)^{-1}(b_1)$ then again $\varphi_v
\leqq\psi_v^t$ for every $v\in (\alpha\circ\varphi)^{-1}(a_1)$. If actually $\varphi_v<\psi_v^t$ we deduce that $\varphi<\psi^t$.

\item
If $(\alpha\circ\varphi)^{-1}(a_1)\subseteq(\beta^t\circ\psi^t)^{-1}(b_1)$ and $\varphi_v,\psi_v^t$ are equivalent but $\lambda^{a_1}=\left\lvert(\alpha\circ\varphi)^{-1}(a_1)\right\rvert<\left\lvert(\beta^t\circ\psi^t)^{-1}(b_1)\right\rvert=\xi^{b_1}$ again $\varphi<\psi^t$. 

\item
If $(\alpha\circ\varphi)^{-1}(a_1)=(\beta^t\circ\psi^t)^{-1}(b_1)$ and $\varphi_v,\psi_v^t$ are equivalent but $\lambda^{a_1}<\xi^{b_1}$ again $\varphi<\psi^t$.

\item
The last case: if $(\alpha\circ\varphi)^{-1}(a_1)=(\beta^t\circ\psi^t)^{-1}(b_1)$, $\varphi_v,\psi_v^t$ are equivalent and also $\lambda^{a_1}=\xi^{b_1}$ (so that $a_1='b_1$) then we can delete $a_1$ and $b_1$, that is the subtrees $\mathcal{T}_v$, $v\in (\alpha\circ\varphi)^{-1}(a_1)$, and their images in $\mathcal{P}$, $\mathcal{M}$, $\mathcal{Q}$, $\mathcal{N}$. Denote by $\mathcal{T}'$, $\mathcal{P}'$, etc. the trees obtained after deletion and $\varphi', \psi'$, etc. the relative maps, which are obviously tree compositions, quotients, etc. By Lemma \ref{Propcoarreftriv}, the coarsest common refinement of $\varphi', \psi'$ is still trivial;  moreover, $\varphi'$ and $(\psi^t)'$ cannot be equivalent. Otherwise, we might construct an equivalence between $\varphi$ and $\psi^t$ by means of Theorem \ref{Theorem38}, using also the fact that $a_1='b_1$.

\end{enumerate}

Now we show that, assuming the result for height=$h-1$, we may prove it for height=$h$ by induction on $\lvert  V_\mathcal{Q}\rvert$.
If $\lvert V_\mathcal{Q}\rvert=1$ then we have two possibilities. If $\lvert  V_\mathcal{N}\rvert\geq2$ then the algorithm stops at step 1. If $\lvert  V_\mathcal{N}\rvert=1$ then it stops at step 2. or 4., because we have assumed that $\varphi$ and $\psi^t$ are not equivalent. 
If $\lvert  V_\mathcal{Q}\rvert\geq2$ we may apply the algorithm; if it arrives at step 5 we delete $a_1$ and $b_1$ and invoke induction on $\lvert  V_\mathcal{Q}\rvert$. 
\end{proof}

\begin{theorem}\label{PropositionTransposed2}
Suppose that $\varphi\in\mathsf{C}(\mathcal{T},\mathcal{P})$, $\alpha\in\mathsf{C}(\mathcal{P},\mathcal{Q})$ is its quotient and that 
$(\Lambda,\mathcal{Q})$ is the associated tree of partitions.
If $\psi\in\mathsf{C}(\mathcal{T},\mathcal{P}^t)$ is strictly equivalent to a transpose $\varphi^t$ of $\varphi$ and the coarsest common refinement $\rho\colon\mathcal{T}\rightarrow\mathcal{R}$ of the couple $(\varphi,\psi)$ is trivial then there exists $g\in \Aut(\mathcal{T})$ such that:
\[
g\varphi=\varphi\qquad\text{ and }\qquad g\varphi^t=\psi.
\]
In particular, the couples $(\varphi,\psi)$ and $(\varphi,\varphi^t)$ are strictly equivalent and $\rho$ is equivalent to the coarsest common refinement of $(\varphi,\varphi^t)$ (in the sense of \eqref{diagepsilon2}).
\end{theorem}
\begin{proof}
The proof is by induction on $H(\mathcal{T})$. For $H(\mathcal{T})=1$ this is the {\em uniqueness} of $m^\lambda$ in \cite[Section 7]{CST3} or of the $0-1$ matrix in \cite[2.1.1]{JK} but we need to recast it in our setting. The strict equivalence between $\psi$ and $\varphi^t$ implies that $\left\lvert\psi^{-1}(y)\right\rvert=\left\lvert\left(\varphi^t\right)^{-1}(y)\right\rvert$, for all $y\in V_{\mathcal{P}^t}$ while 
the triviality of $\rho$ and Remark \ref{01Remark} imply that $\left\lvert \varphi^{-1}(x)\cap \psi^{-1}(y)\right\rvert\leq 1$, for all $x\in V_\mathcal{P}$ and $y\in V_{\mathcal{P}^t}$. Then, from the aforementioned references, it follows that 
\[
\left\lvert\varphi^{-1}(x)\cap (\varphi^t)^{-1}(y)\right\rvert=\left\lvert \varphi^{-1}(x)\cap \psi^{-1}(y)\right\rvert,
\]
for all $x\in V_\mathcal{P}$ and $y\in V_{\mathcal{P}^t}$. Therefore there exists $\pi\in\Sym\left(V_\mathcal{T}\right)$ such that
\[
\pi\left(\varphi^{-1}(x)\cap (\varphi^t)^{-1}(y)\right)=\varphi^{-1}(x)\cap \psi^{-1}(y).
\]
In particular, $\pi\in\prod_{x\in V_\mathcal{P}}\Sym\left(\varphi^{-1}(x)\right)$, which implies that $\varphi\circ\pi=\varphi$, and $\pi\bigl((\varphi^t)^{-1}(y)\bigr)= \psi^{-1}(y)$, for all $y\in V_{\mathcal{P}}^t$, which implies that $\psi\circ\pi=\varphi^t$.

Now suppose that $H(\mathcal{T})\geq2$. On $W_\mathcal{T}$ we may argue as in the case $H(\mathcal{T})=1$.
Then we may suppose that $\varphi,\psi$ and $\varphi^t$ have the same quotient tree $\mathcal{Q}$ and that $\alpha^t$ is also the map of the quotient of $\psi$ (cf. Remark \ref{quotstequiv}) with the same tree of partitions (cf. Theorem \ref{Theoorbits}\eqref{Theoorbitsb}). We have 
\begin{equation}\label{alphavarphi}
(\alpha\circ\varphi)^{-1}(a)=(\alpha^t\circ\varphi^t)^{-1}(a)=(\alpha^t\circ\psi)^{-1}(a),\quad\text{ for all }a\in V_\mathcal{Q},
\end{equation}
where the first equality follows from $\alpha\circ\varphi=\alpha^t\circ\varphi^t$ in Theorem \ref{PropositionTransposed}\eqref{transp1} while the second will be proved now. First of all, from \eqref{lamalphell} it follows that
\begin{equation}\label{lvertalph}
\lvert(\alpha^t\circ\varphi^t)^{-1}(a)\rvert=\lvert(\alpha^t\circ\psi)^{-1}(a)\rvert;
\end{equation}
moreover, the tree compositions
\begin{equation}\label{lvertalph2}
\left\{\varphi^t_v:v\in(\alpha^t\circ\varphi^t)^{-1}(a)\right\}\;\text{ and }\;\left\{\psi_u:u\in(\alpha^t\circ\psi)^{-1}(a)\right\}
\end{equation}
are all equivalent, because $\varphi^t$ and $\psi$ are strictly equivalent. The second identity in \eqref{alphavarphi} is obvious for the unique $a\in W_\mathcal{Q}$ (if it exists). Now suppose that $a_1,a_2,\dotsc,a_m$ are the elements of $V_\mathcal{Q}\setminus W_\mathcal{Q}$ numbered in such a way that $\varphi^t_{v_1}>\varphi^t_{v_2}>\dotsb>\varphi^t_{v_m}$, if $v_i\in(\alpha^t\circ\varphi^t)^{-1}(a_i)$, $i=1,2,\dotsc,m$. Take $u\in (\alpha^t\circ\varphi^t)^{-1}(a_m)$; since $(\varphi_u,\psi_u)$ have a trivial common refinement (by the hypothesis and Lemma \ref{Propcoarreftriv}), from Theorem \ref{TheoGalRy} it follows that $\psi_u\leqq \varphi_u^t$. From \eqref{lvertalph2} and the minimality of $\varphi^t_{v_m}$ we deduce that $\psi_u$ is equivalent to $\varphi^t_u$ so that \eqref{lvertalph} and \eqref{lvertalph2} imply that the second equality in \eqref{alphavarphi} holds true for $a=a_m$. By iterating the procedure, we may prove the second equality in \eqref{alphavarphi} for all $a_i$.

Now we suppose that $g'\in\Aut(\mathcal{T})$ yields the strict equivalence between $\psi$ and $\varphi^t$, that is $\psi\circ g'=\varphi^t$, and we set $\sigma\coloneqq g'\rvert_{V_\mathcal{T}}$. 
Then
\begin{equation}\label{propgprimo} 
\psi(\sigma(u))=\varphi^t(u)\quad\text{ and }\quad\psi_{\sigma(u)}\circ g'_u=\varphi^t_u,
\end{equation}
for all $u\in V_\mathcal{T}\setminus W_\mathcal{T}$. In particular, $\psi_{\sigma(u)}$ and $\varphi^t_u$ are strictly equivalent, so that from \eqref{alphavarphi} it follows that $\sigma\left((\alpha\circ\varphi)^{-1}(a)\right)=(\alpha\circ\varphi)^{-1}(a)$. By means of the $\tau$'s of $\varphi$ and $\varphi^t$ (cf. Theorem \ref{PropositionTransposed}),  we define $g''_u\in\Aut\left(\mathcal{T}_u\right)$ by setting,
for all $u\in V_\mathcal{T}\setminus W_\mathcal{T}$,
\begin{equation}\label{defgsecbis}
g''_u\coloneqq g'_{\sigma^{-1}(u)}\circ\tau_{\sigma^{-1}(u),u}\qquad\qquad\xymatrix{
\mathcal{T}_u\ar[rrd]^{g''_u}\ar[d]_{\tau_{\sigma^{-1}(u),u}} && \\
\mathcal{T}_{\sigma^{-1}(u)}\ar[rr]_{g'_{\sigma^{-1}(u)}}&& \mathcal{T}_u\\
}
\end{equation}
Denoting by $\omega^t$ the $\omega$'s of $\varphi^t$, from \eqref{diagTRSuv}, \eqref{propgprimo} and \eqref{defgsecbis} it follows that 
\[
\psi_u\circ g''_u=\psi_u\circ g'_{\sigma^{-1}(u)}\circ\tau_{\sigma^{-1}(u),u}=\varphi^t_{\sigma^{-1}(u)}\circ\tau_{\sigma^{-1}(u),u}=\omega^t_{\psi(u),\varphi^t(u)}\circ\varphi^t_u.
\]
Then $\omega^t_{\psi(u),\varphi^t(u)}\circ\varphi^t_u$ is strictly equivalent to $\psi_u$ and it is a transpose of $\varphi_u$, because $\omega^t_{\psi(u),\varphi^t(u)}$ is an isomorphism. By the hypothesis and and Lemma \ref{Propcoarreftriv}, the coarsest common refinement of the couple $\varphi_u$ and $\psi_u$ is trivial. Now we can invoke induction: for every  $u\in V_\mathcal{T}\setminus W_\mathcal{T}$ there exists $\widetilde{g}_u\in\Aut\left(\mathcal{T}_u\right)$ such that 
\begin{equation}\label{defgtildeu}
\varphi_u\circ \widetilde{g}_u=\varphi_u\quad\text{ and }\quad\psi_u\circ \widetilde{g}_u=\omega^t_{\psi(u),\varphi^t(u)}\circ\varphi^t_u.
\end{equation}

Now suppose that $u,v\in(\alpha\circ\varphi)^{-1}(a)$, $\varphi(u)=\varphi(v)$ and $\psi(u)=\psi(v)$. From \eqref{proptau2b}, \eqref{diagTRSuv} and \eqref{diagTRSuv2b} it follows that  
\[
\varphi_u\circ\tau_{u,v}=\varphi_v\quad\text{ and }\quad\omega^t_{\psi(u),\varphi^t(u)}\circ\varphi^t_u\circ\tau_{u,v}=\omega^t_{\psi(u),\varphi^t(u)}\circ\omega^t_{\varphi^t(u),\varphi^t(v)}\circ\varphi^t_v=\omega^t_{\psi(v),\varphi^t(v)}\circ\varphi^t_v.
\]
Taking also \eqref{defgtildeu} into account, we deduce that the couples 
\[
\left(\varphi_u,\psi_u\right),\quad \left(\varphi_u,\omega^t_{\psi(u),\varphi^t(u)}\circ\varphi^t_u\right), \quad\left(\varphi_v,\omega^t_{\psi(v),\varphi^t(v)}\circ\varphi^t_v\right)\quad \text{ and }\quad\left(\varphi_v,\psi_v\right)
\]
are all strictly equivalent. From Lemma \ref{Propcoarreftriv} and the triviality of $\rho$ it follows that $u=v$, so that, within each set, $(\alpha\circ\varphi)^{-1}(a)$ the restriction of the map $\rho$ determines a trivial common refinement of the restrictions of $\varphi$ and $\psi$. Since the associated integer partitions are $\lambda^a$ and $(\lambda^a)^t$, by the case $H(\mathcal{T})=1$, there exists $\pi\in\prod\limits_{x\in V_\mathcal{P}}\Sym\left(\varphi^{-1}(x)\right)$ such that 
\begin{equation}\label{vargpsi}
\varphi\rvert_{V_\mathcal{T}}\circ \pi=\varphi\rvert_{V_\mathcal{T}}\quad\text{ and }\quad\psi\rvert_{V_\mathcal{T}}\circ \pi=\varphi^t\rvert_{V_\mathcal{T}}.
\end{equation}
After that, for all $u\in V_\mathcal{T}\setminus W_\mathcal{T}$, we define an isomorphism $g_u\colon\mathcal{T}_u\rightarrow\mathcal{T}_{\pi(u)}$ by setting
\begin{equation}\label{defgud}
g_u\coloneqq\widetilde{g}_{\pi(u)}\circ\tau_{\pi(u),u}\qquad\qquad
\xymatrix{
\mathcal{T}_u\ar[rrd]^{g_u}\ar[d]_{\tau_{\pi(u),u}} && \\
\mathcal{T}_{\pi(u)}\ar[rr]_{\widetilde{g}_{\pi(u)}}&& \mathcal{T}_{\pi(u)}\\
}
\end{equation}
Then we have
\begin{align*}
\varphi_{\pi(u)}\circ g_u
=&\varphi_{\pi(u)}\circ\widetilde{g}_{\pi(u)}\circ\tau_{\pi(u),u}=\varphi_{\pi(u)}\circ\tau_{\pi(u),u}&(\text{by }\eqref{defgud}\text{ and }\eqref{defgtildeu})\\
=&\varphi_u&(\text{by }\eqref{diagTRSuv2b}\text{ and }\eqref{vargpsi})\\
\end{align*}
and
\begin{align*}
\psi_{\pi(u)}\circ g_u
=&\omega^t_{\psi(\pi(u)),\varphi^t(\pi(u))}\circ\varphi^t_{\pi(u)}\circ\tau_{\pi(u),u}&(\text{by }\eqref{defgud}\text{ and }\eqref{defgtildeu})\\
=&\omega^t_{\varphi^t(u),\varphi^t(\pi(u))}\circ\varphi^t_{\pi(u)}\circ\tau_{\pi(u),u}&(\text{by }\eqref{vargpsi})\\
=&\varphi^t_u.&(\text{by }\eqref{diagTRSuv})\\
\end{align*}
Finally, we define $g\in\Aut(\mathcal{T})$ by means of \eqref{firstldell}, setting $g\rvert_{W_\mathcal{T}}\coloneqq \pi\rvert_{W_\mathcal{T}}$ and then $g(u,w)\coloneqq (\pi(u),g_u(w))$, for all $u\in V_\mathcal{T}\setminus W_\mathcal{T}$, $w\in \widetilde{\mathcal{T}}_u$. Taking also \eqref{vargpsi} into account, from \eqref{firstldell2} we deduce that $\varphi\circ g=\varphi$ and $\psi\circ g=\varphi^t$, that is $g\varphi=\varphi$ and $g\varphi^t=\psi$; cf. \eqref{defgvaract}.
\end{proof}

\section{The irreducible representations via Mackey theory}\label{SecrepAuttree}

Now we set the stage for the application of Mackey theory of induced representations for which we refer to \cite{Bump, CST3, book2, book4, Sternberg}. 

If $\mathcal{P}, \mathcal{M}$ are compositional trees for $\mathcal{T}$ we denote by $\mathsf{I}(\mathcal{T};\mathcal{P},\mathcal{M})$ the set of all equivalence classes of coarsest common refinements of couples $(\varphi,\psi)\in\mathsf{C}(\mathcal{T},\mathcal{P})\times\mathsf{C}(\mathcal{T},\mathcal{M})$; cf. \eqref{diagepsilon2} and Definition \ref{defcoarcomref}.
We denote by $[\mathcal{R}]$ an element of $\mathsf{I}(\mathcal{T};\mathcal{P},\mathcal{M})$. If $ (\varphi,\psi)\in\mathsf{C}(\mathcal{T},\mathcal{P})\times\mathsf{C}(\mathcal{T},\mathcal{M})$ then we denote by $[\mathcal{R}](\varphi,\psi)$ the equivalence classes in $\mathsf{I}(\mathcal{T};\mathcal{P},\mathcal{M})$ containing the coarsest common refinement of $(\varphi,\psi)$. The diagonal action of  $\Aut(\mathcal{T})$
on $\mathsf{C}(\mathcal{T},\mathcal{P})\times\mathsf{C}(\mathcal{T},\mathcal{M})$ is given by:
\[
g(\varphi,\psi)\coloneqq(\varphi\circ g^{-1},\psi\circ g^{-1}),\qquad\text{for } (\varphi,\psi)\in\mathsf{C}(\mathcal{T},\mathcal{P})\times\mathsf{C}(\mathcal{T},\mathcal{M}), \quad g\in \Aut(\mathcal{T});
\]
cf. \eqref{defgvaract}.
From Proposition \ref{Propos4p15bis} it follows that $\mathsf{I}(\mathcal{T};\mathcal{P},\mathcal{M})$ parametrizes the strict equivalence classes of the couples $(\varphi,\psi)$, that is the orbits of $\Aut(\mathcal{T})$
on $\mathsf{C}(\mathcal{T},\mathcal{P})\times\mathsf{C}(\mathcal{T},\mathcal{M})$.

\begin{definition}{\rm
If $\varphi\in\mathsf{C}(\mathcal{T},\mathcal{P})$ and $\psi\in\mathsf{C}(\mathcal{T},\mathcal{M})$ are fixed tree compositions we denote by $\mathsf{I}(\mathcal{T};\lvert\varphi^{-1}\rvert,\lvert\psi^{-1}\rvert)$ the set of all equivalence classes of coarsest common refinements of couples $(\varphi',\psi')\in(\mathcal{P},\lvert\varphi^{-1}\rvert)\times(\mathcal{P},\lvert\psi^{-1}\rvert)$ (cf. \eqref{deforbitphi}). }
\end{definition}

Clearly $\mathsf{I}(\mathcal{T};\lvert\varphi^{-1}\rvert,\lvert\psi^{-1}\rvert)$ parametrizes the diagonal orbits of $\Aut(\mathcal{T})$ on $(\mathcal{P},\lvert\varphi^{-1}\rvert)\times(\mathcal{P},\lvert\psi^{-1}\rvert)$ and the $K_\varphi$-orbits on $(\mathcal{P},\lvert\psi^{-1}\rvert)$; see \cite[Corollary 10.4.13]{book4}. Then for each $[\mathcal{R}]\in\mathsf{I}(\mathcal{T};\lvert\varphi^{-1}\rvert,\lvert\psi^{-1}\rvert)$ we can choose $\psi_{[\mathcal{R}]}\in(\mathcal{P},\lvert\psi^{-1}\rvert)$ such that $[\mathcal{R}]\left(\varphi,\psi_{[\mathcal{R}]}\right)=[\mathcal{R}]$ and then $t_{[\mathcal{R}]}\in\Aut(\mathcal{T})$ such that $t_{[\mathcal{R}]}\psi=\psi_{[\mathcal{R}]}$. It follows that $t_{[\mathcal{R}]}K_{\psi}t_{[\mathcal{R}]}^{-1}=K_{\psi_{[\mathcal{R}]}}$. Moreover,
\begin{equation}\label{defHR}
H_{[\mathcal{R}]}\coloneqq K_{\varphi}\cap\left(t_{[\mathcal{R}]}K_{\psi}t_{[\mathcal{R}]}^{-1}\right)\;\text{ is the stabilizer in }\;K_{\varphi}\;\text{ of }\;\psi_{[\mathcal{R}]}
\end{equation}
and the orbits of $K_{\varphi}$ on $(\mathcal{P},\lvert\psi^{-1}\rvert)$ are $\left\{k\psi_{[\mathcal{R}]}\colon k\in K_{\varphi}\right\}\equiv\left\{kt_{[\mathcal{R}]}\psi\colon k\in K_{\varphi}\right\}$, $[\mathcal{R}]\in\mathsf{I}(\mathcal{T};\lvert\varphi^{-1}\rvert,\lvert\psi^{-1}\rvert)$.
Finally,
\begin{equation}\label{DDdec}
\Aut(\mathcal{T})=\bigsqcup_{[\mathcal{R}]\in\mathsf{I}(\mathcal{T};\lvert\varphi^{-1}\rvert,\lvert\psi^{-1}\rvert)}K_{\varphi}t_{[\mathcal{R}]}K_{\psi}
\end{equation}
is the decomposition of $\Aut(\mathcal{T})$ into double $K_{\varphi}-K_{\psi}$ cosets, cf. \cite[Section 11.3]{book4}.

We define the {\em sign representation} $\text{sign}\colon\Aut(\mathcal{T})\rightarrow\{\pm1\}$ of $\Aut(\mathcal{T})$ inductively as follows. If $H(\mathcal{T})=1$ then it is the usual $\text{sign}$ representation of the symmetric group. If $H(\mathcal{T})>1$ then we use the notation in Example \ref{idtreecomp}, we assume that sign is defined on the groups $\Aut\left(\mathcal{T}_{u_c}\right)$, $c\in V_\mathcal{C}\setminus W_\mathcal{C}$ and we set
\begin{equation}\label{Defsign}
\text{sign}(g)\coloneqq\text{sign}\left(g\rvert_{V_\mathcal{T}}\right)\cdot\prod_{c\in V_\mathcal{C}\setminus W_\mathcal{C}}\prod\limits_{v\in\text{\rm tv}^{-1}(c)}\text{sign}\left(\tau_{u_c,g(v)}\circ g_v\circ \tau_{v,u_c}\right),
\end{equation}
where we have used a system of $\tau$'s for the trivial composition (cf. Definition \ref{Defpseudo}). Clearly $\tau_{u_c,g(v)}\circ g_v\circ\tau_{v,u_c}\in\Aut\left(\mathcal{T}_{u_c}\right)$.
\begin{lemma}
The function \text{\rm sign} in \eqref{Defsign} is a representation of $\Aut(\mathcal{T})$.
\end{lemma}
\begin{proof}
If $g,g'\in\Aut(\mathcal{T})$ then 
\[
\begin{split}
\text{sign}(g'g)=&\text{\rm sign}\left[(g'g)\rvert_{V_\mathcal{T}}\right]\cdot\prod_{\substack{c\in V_\mathcal{C}\setminus W_\mathcal{C}\\ v\in\text{\rm tv}^{-1}(c)}}\text{sign}\left[\tau_{u_c,g'g(v)}\circ( g'g)_v\circ \tau_{v,u_c}\right]\\
(\text{by }\eqref{firstldell2}\text{ and }\eqref{proptau2b})\quad=&\text{sign}\left(g'\rvert_{V_\mathcal{T}}\right)\cdot\text{sign}\left(g\rvert_{V_\mathcal{T}}\right)\cdot\prod_{\substack{c\in V_\mathcal{C}\setminus W_\mathcal{C}\\ v\in\text{\rm tv}^{-1}(c)}}\Bigl\{\text{sign}\left[\tau_{u_c,g'g(v)}\circ g'_{g(v)}\circ \tau_{g(v),u_c}\right]\cdot\\
&\qquad\cdot\text{sign}\left[\tau_{u_c,g(v)}\circ g_v\circ \tau_{v,u_c}\right]\Bigr\}\\
(\text{by }\eqref{squaresigmap17})\quad=&\text{sign}(g')\cdot\text{sign}(g).
\end{split}
\]
\end{proof}
We denote by $\iota_G$ the {\em trivial representation} of the group $G$. In the notation of Definition \ref{defYoungmodules}, we have $M^{\lvert\varphi^{-1}\rvert}=\Ind_{K_{\varphi}}^{\text{\rm Aut}(\mathcal{T})}\iota_{K_{\varphi}}$. Similarly, we set $\widetilde{M}^{\lvert\varphi^{-1}\rvert}\coloneqq\Ind_{K_{\varphi}}^{\text{\rm Aut}(\mathcal{T})}\text{sign}$.

\begin{proposition}\label{PropinvMMt1}
If $\varphi$ and $\psi$ are equivalent as compositions of $\mathcal{T}$ then $\widetilde{M}^{\lvert\varphi^{-1}\rvert}$ and $\widetilde{M}^{\lvert\psi^{-1}\rvert}$ are equivalent as $\Aut(\mathcal{T})$-representations.
\end{proposition}
\begin{proof}
Indeed, $\widetilde{M}^{\lvert\psi^{-1}\rvert}$ is the space of all functions $f\colon\Aut(\mathcal{T})\rightarrow\mathbb{C}$ such that $f(g\circ k)=\text{sign}(k^{-1})f(g)$, for all $g\in \Aut(\mathcal{T})$, $k\in K_\psi$, with the left regular representation. From \eqref{Kconj} it follows that the map $T\colon\widetilde{M}^{\lvert\psi^{-1}\rvert}\rightarrow \widetilde{M}^{\lvert\varphi^{-1}\rvert}$ given by $[Tf](g)\coloneqq f(g\circ\tau^{-1})$, for all $g\in\Aut(\mathcal{T})$, is an equivalence between the induced representations.
\end{proof}

Denote by $\mathsf{I}_0(\mathcal{T};\lvert\varphi^{-1}\rvert,\lvert\psi^{-1}\rvert)$ the set of all $[\mathcal{R}]\in\mathsf{I}(\mathcal{T};\lvert\varphi^{-1}\rvert,\lvert\psi^{-1}\rvert)$ which represent equivalence classes of {\em trivial} common refinements.
\begin{proposition}
If $\varphi$ and $\psi$ are tree compositions of $\mathcal{T}$ then 
\begin{equation}\label{Mackey0}
\text{\rm dim}\Hom_{\text{\rm Aut}(\mathcal{T})}\left(M^{\lvert\varphi^{-1}\rvert},M^{\lvert\psi^{-1}\rvert}\right)=\left\lvert\mathsf{I}(\mathcal{T};\lvert\varphi^{-1}\rvert,\lvert\psi^{-1}\rvert)\right\rvert
\end{equation}
and
\begin{equation}\label{Mackey1}
\text{\rm dim}\Hom_{\text{\rm Aut}(\mathcal{T})}\left(M^{\lvert\varphi^{-1}\rvert},\widetilde{M}^{\lvert\psi^{-1}\rvert}\right)=\left\lvert\mathsf{I}_0(\mathcal{T};\lvert\varphi^{-1}\rvert,\lvert\psi^{-1}\rvert)\right\rvert.
\end{equation}
\end{proposition}
\begin{proof}
The identity in \eqref{Mackey0} follows from \eqref{DDdec} and Mackey intertwining numbers Theorem (cf. \cite[Corollary 11.4.5]{book4}, \cite[p. 64]{Sternberg}). Similarly, Mackey formula for the invariants (cf. \cite[Corollary 11.4.4]{book4}, \cite[p. 164]{Sternberg}) yields
\[
\Hom_{\text{\rm Aut}(\mathcal{T})}\left(M^{\lvert\varphi^{-1}\rvert},\widetilde{M}^{\lvert\psi^{-1}\rvert}\right)=\bigoplus_{[\mathcal{R}]\in\mathsf{I}(\mathcal{T};\lvert\varphi^{-1}\rvert,\lvert\psi^{-1}\rvert)}\Hom_{H_{[\mathcal{R}]}}\left(\iota,\text{sign}\right),
\]
where $\iota$ and $\text{sign}$ are respectively the restriction of trivial and of the sign representations to $H_{[\mathcal{R}]}$. But if $\mathcal{R}$ is not trivial then $H_{[\mathcal{R}]}$ contains some $g$ with $\text{sign}(g)=-1$; for instance, from the inductive definition \eqref{Defsign} it follows that the sign of the $g$ in \eqref{defgsign} is $\,-1$ and from Lemma \ref{Lemmatrivaut} and \eqref{defHR} it follows that this $g$ belongs to $H_{[\mathcal{R}]}$. If this is the case then the trivial and the sign representations are not equivalent and therefore the corresponding $\Hom$ space is trivial. By contrast, if $\mathcal{R}$ is trivial then also $H_{[\mathcal{R}]}$ is trivial and the corresponding $\Hom$ space is one-dimensional and this ends the proof of \eqref{Mackey1}.
\end{proof}

\begin{corollary}\label{Corirr}
The $\Aut(\mathcal{T})$-representations $M^{\lvert\varphi^{-1}\rvert}$ and $\widetilde{M}^{\lvert(\varphi^t)^{-1}\rvert}$ have a unique irreducible common constituent, which is contained in both of them with multiplicity one.
\end{corollary}
\begin{proof}
From Theorem \ref{PropositionTransposed2} it follows that if $\varphi^t$ is a transpose of $\varphi$ then $\left\lvert\mathsf{I}_0(\mathcal{T};\lvert\varphi^{-1}\rvert,\lvert(\varphi^t)^{-1}\rvert)\right\rvert=1$ and then we may apply \eqref{Mackey1}.
\end{proof}

\begin{definition}\label{DefSLambda}
{\rm
In the notation of Definition \ref{defIrrG}, let $\varphi\in\mathsf{C}(\mathcal{T},\mathcal{P})$ and denote by $\Lambda$ the associated tree of partitions. Then we define $S^\Lambda$ as the unique irreducible common constituent of $M^{\lvert\varphi^{-1}\rvert}$ and $\widetilde{M}^{\lvert(\varphi^t)^{-1}\rvert}$ (cf. Corollary \ref{Corirr}). By Theorem \ref{Theoorbits}\eqref{Theoorbitsa}, Proposition \ref{PropinvMMt}, Theorem \ref{PropositionTransposed}\eqref{transp1} and Proposition \ref{PropinvMMt1}, $S^\Lambda$ does not depend on the choice of $\varphi$ and of its transposed $\varphi^t$ but just on $\Lambda$.} 
\end{definition}

\begin{theorem}\label{Theoirredrep}
The set $\left\{S^\Lambda:\Lambda\in\mathsf{Par}(\mathcal{T})\right\}$ is a complete list of pairwise, inequivalent representations of the group $\Aut(\mathcal{T})$.
\end{theorem}
\begin{proof}
First of all, we prove that if $(\Lambda,\mathcal{Q}),(\Xi,\mathcal{N})\in\mathsf{Par}(\mathcal{T})$ are not isomorphic, then $S^\Lambda$ and $S^\Xi$ are not equivalent. Indeed, if $\varphi\in\mathsf{C}(\mathcal{T},\mathcal{P})$ is associated to $(\Lambda,\mathcal{Q})$ and $\psi\in\mathsf{C}(\mathcal{T},\mathcal{M})$ to $(\Xi,\mathcal{N})$, by \eqref{Mackey1} and Definition \ref{DefSLambda}, the equivalence of $S^\Lambda$ and $S^\Xi$ would lead to $\left\lvert\mathsf{I}_0(\mathcal{T};\lvert\varphi^{-1}\rvert,\lvert(\psi^t)^{-1}\rvert)\right\rvert\neq0$, so that, by Theorem \ref{TheoGalRy} and recalling that $\psi$ is a transpose of $\psi^t$, $\varphi<\psi$; by symmetry, we get also $\psi<\varphi$ and this is impossible. Therefore the irreducible representations in the statement are pairwise inequivalent and, by Corollary \ref{Corconjclasses}, this system is also complete.
\end{proof}

\noindent
From the proof of Theorem \ref{Theoirredrep}, we can deduce immediately the following Corollary.

\begin{corollary}\label{Corirredrep}
If $(\Lambda,\mathcal{Q})$ is the tree of partitions of $\varphi$ and $S^\Lambda$ is contained in $M^{\lvert\psi^{-1}\rvert}$ then $\varphi\geqq \psi$.
\end{corollary}

\begin{example}\label{Exsigntens}{\rm
In the notation of Example \ref{idtreecomp} we have: $K_{\text{\rm id}_\mathcal{T}}=\{1_{\text{\rm Aut}(\mathcal{T})}\}$ and $K_{\text{\rm tv}_\mathcal{T}}=\Aut(\mathcal{T})$. Therefore $M^{\lvert\text{\rm id}_\mathcal{T}^{-1}\rvert}$ is the regular representation of $\Aut(\mathcal{T})$, $M^{\lvert\text{\rm tv}_\mathcal{T}^{-1}\rvert}$ coincides with the trivial representation $\iota_{\text{\rm Aut}(\mathcal{T})}$ and $\widetilde{M}^{\lvert\text{\rm tv}_\mathcal{T}^{-1}\rvert}$ is the sign representation. If we denote by $\Lambda_{\text{\rm id}}$ and $\Lambda_{\text{\rm tv}}$ the trees of partitions respectively of the identity and trivial composition (see Example \ref{treeparttvid} for the case of a spherically homogeneous tree), taking Example \ref{transptvid} into account, we have
\[
S^{\Lambda_{\text{\rm tv}}}\text{ is the trivial representation},\qquad
S^{\Lambda_{\text{\rm id}}}\text{ is the sign representation}.
\]
Finally, arguing as in \cite[p. 36]{JK}, we can prove that, for every tree composition $\varphi$ of $\mathcal{T}$ with associated tree of partitions $\Lambda$, we have
\[
S^{\Lambda_{\text{\rm id}}}\otimes M^{\lvert\varphi^{-1}\rvert}\sim \widetilde{M}^{\lvert\varphi^{-1}\rvert}\quad\text{ and }\quad S^{\Lambda_{\text{\rm id}}}\otimes S^\Lambda\sim S^{\Lambda^t}
\]
where $\Lambda^t$ is the tree of partitions associated to any transpose $\varphi^t$ of $\varphi$. Indeed, the first equivalence follows from an easy application of \cite[Corollary 11.1.17]{book4}; the second equivalence follows from the fact that $S^{\Lambda_{\text{\rm id}}}\otimes S^\Lambda$ is contained both in $S^{\Lambda_{\text{\rm id}}}\otimes M^{\lvert\varphi^{-1}\rvert}\sim \widetilde{M}^{\lvert\varphi^{-1}\rvert}$ and in $S^{\Lambda_{\text{\rm id}}}\otimes \widetilde{M}^{\lvert(\varphi^t)^{-1}\rvert}\sim M^{\lvert(\varphi^t)^{-1}\rvert}$, so that it coincides with $S^{\Lambda^t}$.
}
\end{example}

Finally, we prove that both Clifford theory and Mackey theory lead to the same parametrization of the irreducible representations of $\Aut(\mathcal{T})$.

\begin{theorem}
With the notation in Theorem \ref{irredClifftheo} and Theorem \ref{Theoirredrep}, $Z^\Lambda$ and $S^\Lambda$ are equivalent $\Aut(\mathcal{T})$ irreducible representations. Moreover, they coincide as multiplicity one subspaces of $M^{\lvert\varphi^{-1}\rvert}$ if $Z^\Lambda$ is constructed by means of the realization of $\varphi$.
\end{theorem}
\begin{proof}
 First of all, by induction on $H(\mathcal{T})$, we prove that $S^{\Lambda_{\text{\rm id}}}\otimes Z^\Lambda\sim Z^{\Lambda^t}$; cf. Example \ref{Exsigntens}. For $H(\mathcal{T})=1$ it is a well known fact in the representation theory of the symmetric group. Denote by $Z'$ and $S'$ the representations in \eqref{irredClifford} but constructed by means of $\Lambda^t$; recall that, by Theorem \ref{PropositionTransposed}\eqref{transp1}, $\Lambda^t$ has the same realization $\delta$ of $\Lambda$, and therefore in the application of \eqref{irredClifford} we may use the same group $I$ in \eqref{irredsigcbis}. After that, for $g\in I$ we set 
\[
\widetilde{\text{sign}}(g)\coloneqq\prod_{c\in V_\mathcal{C}\setminus W_\mathcal{C}}\prod\limits_{v\in\text{\rm tv}^{-1}(c)}\text{sign}\left(\tau_{u_c,g(v)}\circ g_v\circ \tau_{v,u_c}\right)\quad\text{ and }\quad\overline{\text{sign}}(g)\coloneqq\text{sign}\left(g\rvert_{V_\mathcal{T}}\right).
\]
Clearly, $\widetilde{\text{sign}}$ is an extension of $\text{sign}$ from $N$ to $I$ while $\overline{\text{sign}}$ is the inflation of $\text{sign}$ from $I/N$ to $I$. From \eqref{Defsign} it follows that $\Res^{\text{\rm Aut}(\mathcal{T})}_I\text{\rm sign }=\widetilde{\text{sign}}\otimes\overline{\text{sign}}$ and therefore 
\begin{align*}
\text{\rm sign}\otimes Z^\Lambda=&\text{\rm sign }\otimes \Ind_I^{\text{\rm Aut}(\mathcal{T})}\left(\widetilde{Z}\otimes\overline{S}\right)&(\text{by }\eqref{irredClifford})\\
=&\Ind_I^{\text{\rm Aut}(\mathcal{T})}\left[\left(\Res_I^{\text{\rm Aut}(\mathcal{T})}\text{\rm sign}\right)\otimes\left(\widetilde{Z}\otimes\overline{S}\right)\right]&(\text{by }\cite[\rm Theorem 11.1.16]{book4})\\
=&\Ind_I^{\text{\rm Aut}(\mathcal{T})}\left[\left(\widetilde{\text{\rm sign }\otimes Z}\right)\otimes\left(\overline{\text{\rm sign }\otimes S}\right)\right]&\\
=&\Ind_I^{\text{\rm Aut}(\mathcal{T})}\left( Z' \otimes S'\right)&(\text{inductive hypothesis})\\
=&Z^{\Lambda^t}&
\end{align*}

We deduce that $S^{\Lambda_{\text{\rm id}}}\otimes Z^\Lambda$ is contained both in $S^{\Lambda_{\text{\rm id}}}\otimes M^{\lvert\varphi^{-1}\rvert}\sim \widetilde{M}^{\lvert\varphi^{-1}\rvert}$ (cf. Example \ref{Exsigntens}) and in $M^{\lvert(\varphi^t)^{-1}\rvert}$, so that $S^{\Lambda_{\text{\rm id}}}\otimes Z^\Lambda\sim S^{\Lambda^t}$ and therefore
$Z^\Lambda\sim  S^{\Lambda_{\text{\rm id}}}\otimes S^{\Lambda^t}\sim S^\Lambda$. Actually, $S^\Lambda$ and $Z^\Lambda$ coincide as subspaces of $M^{\lvert\varphi^{-1}\rvert}$, because both have multiplicity one.
\end{proof}

\vspace{1cm}

\noindent
Fabio Scarabotti, Dipartimento SBAI, Sapienza Universit\`a di Roma, via A. Scarpa 8, 00161, Roma, Italy\\
\noindent
Istituto Nazionale di Alta Matematica ``Francesco Severi'', 00185 Roma, Italy\\
\noindent
 {\it e-mail:} {\tt fabio.scarabotti@sbai.uniroma1.it}

\end{document}